\documentclass[11pt,reqno]{amsart}
\usepackage{amssymb,mathrsfs,graphicx,subfigure, enumerate}
\usepackage{amsmath,amsfonts,amssymb,amscd,amsthm,bbm}
\usepackage{extpfeil}
\usepackage{graphicx,colortbl}
\usepackage{float}
\usepackage{epsfig}
\usepackage{caption}
\usepackage{bm}
\graphicspath{{Figure/}}
\usepackage{kotex}
\usepackage[margin=1in]{geometry}

\makeatletter
\@namedef{subjclassname@2020}{\textup{2020} Mathematics Subject Classification}
\makeatother

\title[Navier-Stokes-Korteweg equations]{Time-asymptotic stability of composite wave for the one-dimensional compressible fluid of Kortwewg type}

\author[Han]{Sungho Han}
\address[Sungho Han]{\newline Department of Mathematical Sciences \newline Korea Advanced Institute of Science and Technology, Daejeon  34141, Republic of Korea}
\email{sungho\_han@kaist.ac.kr}

\author[Kim]{Jeongho Kim}
\address[Jeongho Kim]{\newline Department of Applied Mathematics, \newline Kyung Hee University, 1732, Deogyeong-daero, Giheung-gu, Yongin-si, Gyeonggi-do 17104, Republic of Korea}
\email{jeonghokim@khu.ac.kr}

\begin{document}
%%%%%%%%%%%%%%%%%%%%%%%%%%%
\newtheorem{theorem}{Theorem}[section]
\newtheorem{lemma}{Lemma}[section]
\newtheorem{corollary}{Corollary}[section]
\newtheorem{proposition}{Proposition}[section]
\newtheorem{remark}{Remark}[section]
\newtheorem{definition}{Definition}[section]
%%%%%%%%%%%%%%%%%%%%%%%%%%%
\renewcommand{\theequation}{\thesection.\arabic{equation}}
\renewcommand{\thetheorem}{\thesection.\arabic{theorem}}
\renewcommand{\thelemma}{\thesection.\arabic{lemma}}
%%%%%%%%%%%%%%%%%%%%%%%%%%%
\newcommand{\bbr}{\mathbb R}
\newcommand{\R}{\mathbb{R}}
\newcommand{\e}{\varepsilon}
\newcommand{\pa}{\partial}
%%%%%%%%%%%%%%%%%%%%%%%%%%%
\newcommand{\oU}{\overline{U}}
\newcommand{\ou}{\overline{u}}
\newcommand{\ov}{\overline{v}}
\newcommand{\ow}{\overline{w}}
\newcommand{\ur}{u^R}
\newcommand{\us}{u^S}
\newcommand{\vr}{v^R}
\newcommand{\vs}{v^S}
\newcommand{\rw}{w^R}
\newcommand{\ws}{w^S}
\newcommand{\m}{\overline{\mu}}
\newcommand{\mr}{\mu^R}
\newcommand{\ms}{\mu^S}
\newcommand{\ok}{\overline{\kappa}}
\newcommand{\kr}{\kappa^R}
\newcommand{\ks}{\kappa^S}
%%%%%%%%%%%%%%%%%%%%%%%%%%%
\newcommand{\norm}[1]{\left\lVert#1\right\rVert}
%%%%%%%%%%%%%%%%%%%%%%%%%%%



\thanks{S. Han was partially supported by the National Research Foundationof Korea (NRF-2019R1C1C1009355 and NRF-2019R1A5A1028324). J. Kim was supported by Samsung Science and Technology Foundation under Project Number SSTF-BA2401-01.}

\begin{abstract} 
	We study the asymptotic stability of a composition of rarefaction and shock waves for the one-dimensional barotropic compressible fluid of Korteweg type, called the Navier-Stokes-Korteweg(NSK) system. Precisely, we show that the solution to the NSK system asymptotically converges to the composition of the rarefaction wave and shifted viscous-dispersive shock wave, under certain smallness assumption on the initial perturbation and strength of the waves. Our method is based on the method of $a$-contraction with shift developed by Kang and Vasseur \cite{KV16}, successfully applied to obtain contraction or stability of nonlinear waves for hyperbolic systems.
\end{abstract}

\maketitle

\section{Introduction}\label{sec:1}
\setcounter{equation}{0}

A theory on fluid dynamics with internal capillarity originates from the works of Van der Waals \cite{V94} and Korteweg \cite{K01}, who introduced stress tensors involving higher-order derivatives to model capillarity effects. Later, Dunn and Serrin \cite{DS85} rigorously derived the governing equations for the fluid with capillarity, known as the Navier-Stokes-Korteweg(NSK) equations, from the physical principles of thermodynamics. The NSK equations have been applied in various fields, including quantum fluid dynamics and two-phase flow modeling \cite{AS22, BGV19, BVY21}. 

In the present paper, we study the one-dimensional barotropic compressible NSK equations in terms of the Lagrangian mass coordinates:
\begin{equation}\label{eq:NSK}
\begin{aligned}
&v_t-u_x=0,\quad (t,x)\in \bbr_+\times\bbr,\\
&u_t+p(v)_x=\left(\mu(v)\frac{u_x}{v}\right)_x + \left( \kappa(v) \left(  -\frac{v_{xx}}{v^5} + \frac{5v_x^2}{2v^6}\right)-\frac{1}{2} \kappa'(v) \frac{(v_x)^2}{v^5} \right)_x,
\end{aligned}
\end{equation}
 supplemented with the initial data
 $(v,u)(0,x)=(v_0,u_0)(x)$
 and the far-field states 
\begin{equation}\label{end states}
\lim_{x \to \pm \infty} (v_0(x),u_0(x))=(v_\pm,u_\pm).
\end{equation}
Here,  $v=v(t,x)$ and $u=u(t,x)$ represent the specific volume and velocity of the fluid, respectively. The pressure follows $\gamma$-law $p(v)=v^{-\gamma}$ with $\gamma>1$ and the coefficient $\mu=\mu(v)>0$ and $\kappa=\kappa(v)>0$ represent the viscosity and capillarity of the fluid. In particular, when the capillary effect is neglected, that is $\kappa\equiv0$, the NSK system \eqref{eq:NSK} is reduced to the standard barotropic compressible Navier-Stokes(NS) equations:
\begin{equation}
\begin{aligned}\label{eq:NS}
&v_t-u_x=0,\quad (t,x)\in\bbr_+\times\bbr,\\
&u_t +p(v)_x=\left(\mu(v)\frac{u_x}{v}\right)_x.
\end{aligned}
\end{equation}

The goal of the present paper is to investigate the time-asymptotic behavior of solutions to \eqref{eq:NSK}--\eqref{end states}. Similar to the case of the NS equations \eqref{eq:NS}, the asymptotic patterns of the solutions to the NSK equations \eqref{eq:NSK} are determined by the Riemann problem for the corresponding Euler equations:
\begin{equation}\label{eq:Euler}
\begin{aligned}
&v_t-u_x=0,\quad (t,x)\in\bbr_+\times\bbr,\\
&u_t+p(v)_x=0,
\end{aligned}
\end{equation}
subject to the Riemann initial data
\begin{equation}\label{Riemann-data}
(v(0,x),u(0,x))=
\begin{cases} 
(v_-,u_-), &x<0,\\
(v_+,u_+), &x>0.
\end{cases}
\end{equation}
The Riemann solutions of \eqref{eq:Euler} consist of the rarefaction wave, shock wave, or their compositions. We focus on the case of the composite wave consisting of the 1-rarefaction wave and the 2-shock wave, where there exists an intermediate state $(v_m,u_m)$ so that $(v_-,u_-)$ and $(v_+,u_+)$ are connected to $(v_m,u_m)$ via 1-rarefaction curve and 2-shock curve respectively as below. %This configuration involves intricate interactions between rarefaction and shock waves, providing significant mathematical insights.
%We assume that $(v_-,u_-) \in R_1S_2(v_+,u_+)$, which means that there exists a unique intermediate state $(v_m,u_m)$ such that the left-end state $(v_-,u_-)$ satisfies $(v_-,u_-)$ is connected to the intermediate state $(v_m,u_m)$ via the 1-rarefaction wave $R_1(v_m,u_m)$, while the intermediate state $( v_m,u_m)$ is connected to the right-end state $(v_+,u_+)$ through the $2$-shock wave $S_2(v_+,u_+)$.

The 1-rarefaction curve $ R_1(v_m,u_m) $ is defined as an integral curve of the first characteristic field associated with $ \lambda_1(v) = -\sqrt{-p'(v)} $ for the hyperbolic part of \eqref{eq:Euler}:
\[
R_1(v_m,u_m) := \left\{ (v, u) \;\middle|\; u = u_m - \int_{v_m}^v \lambda_1(s) \, ds, \quad v < v_m \right\}.
\]
For $ (v_-,u_-) \in R_1(v_m,u_m)$, the $ 1 $-rarefaction wave $ (v^r, u^r)(x/t) $ is a self-similar solution to the Riemann problem \eqref{eq:Euler}--\eqref{Riemann-data} given as
\begin{equation}\label{eq:rarefaction wave}
(v^r, u^r)(x/t) =
\begin{cases}
(v_-,u_-), & x < \lambda_1(v_-)t, \\ 
\begin{cases}
v^r = \lambda_1^{-1}(x/t), \\ 
u^r = u_m - \int_{v_m}^{v_1^r} \lambda_1(s) \, ds
\end{cases}
& \lambda_1(v_-)t \leq x \leq \lambda_1(v_m)t, \\ 
(v_m,u_m), & x > \lambda_1(v_m)t,
\end{cases}
\end{equation}
together with the 1-Riemann invariant
\begin{equation*}
     z_1((v^r, u^r)(x/t))=z_1(v_-,u_-)=z_1(v_m,u_m), \quad z_1(v,u):=u+\int^v\lambda_1(s)\,ds.
\end{equation*}
On the other hand, the 2-shock curve $ S_2(v_+,u_+) $ is defined as a Hugoniot curve corresponding to the second characteristic field associated with $\lambda_2(v)=\sqrt{-p'(v)}$:
\[
S_2(v_+,u_+) := \left\{ (v, u) ~\middle|~ u = u_+ + \sigma (v_+ - v), \; v < v_+ \right\},\quad \sigma = \sqrt{-\frac{p(v) - p(v_+)}{v - v_+}} .
\]
For any $ (v_m,u_m) \in S_2(v_+,u_+) $, the Riemann solution with the far-field states $(v_m,u_m)$ and $(v_+,u_+)$ is given as the 2-shock wave with the shock speed $\sigma := \sqrt{-\frac{p(v_m)-p(v_+)}{v_m-v_+}}$:
\begin{equation} \label{eq:shock wave}
(v^s, u^s)(x, t) =
\begin{cases}
(v_m,u_m), & x < \sigma t, \\ 
(v_+,u_+), & x > \sigma t.
\end{cases}
\end{equation}
Therefore, for any $(v_-,u_-) \in R_1S_2(v_+,u_+)$, there exists a unique intermediate state $(v_m,u_m)\in \mathbb{R}_+ \times \mathbb{R}$ such that $(v_-,u_-) \in R_1(v_m,u_m)$, and $(v_m,u_m) \in S_2(u_+,v_+)$. In this case, the Riemann solution $(\tilde{v},\tilde{u})$ is a composite wave consisting of a 1-rarefaction wave $(v^r,u^r)$ and a 2-shock wave $(v^s,u^s)$:
\[(\tilde{v},\tilde{u})(t,x)=(v^r,u^r)(x/t)+(v^s,u^s)(x-\sigma t)-(v_m,u_m).\]

\subsection{Related works} 
Asymptotic behaviors of the solutions to compressible fluid models, especially the NS equations \eqref{eq:NS}, have been extensively investigated in previous literature. For the NS equations, the viscous counterpart to the shock wave \eqref{eq:shock wave}, called a viscous shock, should be introduced. The viscous shock is a traveling wave solution $(v^S(\xi),u^S(\xi))$ to \eqref{eq:NS} where $\xi=x-\sigma t$, whose profile satisfies the following system of ODEs:

\begin{equation*} %\label{viscous-shock}
\begin{cases}
-\sigma \vs_\xi-\us_\xi=0,\\
-\sigma \us_\xi+p(\vs)_\xi=\left(\mu(\vs)\frac{\us_\xi}{v^S}\right)_\xi,
\end{cases}
\end{equation*}
subject to the boundary conditions $(\vs,\us)(\pm\infty)=(v_\pm,u_\pm)$. A study on the asymptotic behavior of the NS equations was initiated by Matsumura and Nishihara \cite{MN85} where they proved an asymptotic stability of the viscous shock wave by introducing the anti-derivative variables under the zero-mass condition. Goodman \cite{G86} proved a similar stability result of the viscous shock wave for a system with artificial diffusion under similar conditions. Later on, the zero-mass condition is alleviated by introducing the constant shift on the viscous shock \cite{L85} and the diffusion waves \cite{SX93}. Masica and Zumbrun \cite{MZ04} also investigated the stability of the viscous shock wave based on the spectral stability. On the other hand, the stability of the rarefaction waves and their composition for the NS equations was proved in \cite{MN86,MN92}, based on the standard energy method.

However, the stability of the composition of two viscous shock waves and the composition of the rarefaction and the shock waves remained as open problems until very recently. The stabilities of these cases are proved in \cite{HKK23} and \cite{KVW23} respectively, by using the method of $a$-contraction with shift developed by Kang and Vasseur \cite{KV16}. This method works on the physical variable, not the anti-derivative variable, and therefore it can naturally integrate the stability analysis of the rarefaction and shock waves. We refer to \cite{Kang-V-1,KV21,KV-Inven,KVW23} for the intuition and further discussion on the method of $a$-contraction with shifts.

We now turn our focus on the asymptotic behavior of the NSK equations \eqref{eq:NSK}. For the case of NSK equations \eqref{eq:NSK}, the counterpart of the shock wave is a viscous-dispersive (or dissipative-dispersive) shock wave, which is a traveling wave $(\vs,\us)(\xi)$ satisfying the following ODEs:
\begin{equation}\label{viscous-dispersive-shock}
\begin{cases}
-\sigma \vs_\xi-\us_\xi=0,\\
-\sigma \us_\xi+p(\vs)_\xi=\left(\mu(\vs)\frac{\us_\xi}{v^S}\right)_\xi + \left( \kappa(\vs) \left(  -\frac{\vs_{\xi\xi}}{(\vs)^5} + \frac{5(\vs_\xi)^2}{2(\vs)^6}\right)-\frac{1}{2} \kappa'(\vs) \frac{(\vs_\xi)^2}{(\vs)^5} \right)_\xi,
\end{cases}
\end{equation}
subject to the boundary conditions $(\vs,\us)(\pm\infty)=(v_\pm,u_\pm)$. The existence of the traveling wave solution \eqref{viscous-dispersive-shock} and its stability was studied in \cite{CHZ15}, where they relied on the traditional anti-derivative method under the zero-mass condition. In \cite{FPZ23}, the spectral stability of the viscous-dispersive shock waves for the quantum NS equations, a special case of the NSK equations, was established. Recently, the authors applied the method of $a$-contraction with shift to show the stability of the viscous-dispersive shock wave, removing the zero-mass constraint on the initial data. On the other hand, the stability of the rarefaction waves for the NSK equations was proved in \cite{C12} using the standard energy method.

\subsection{Main Theorem}
Therefore, it is natural to ask whether the stability of the composition of the rarefaction and shock waves can be obtained for the case of the NSK equations, and our goal of the present work is to answer this question. We present the main theorem as follows, which states the stability of the composition waves of the rarefaction and the viscous-dispersive shock wave.
    
\begin{theorem}\label{thm:main}
    For a given right-end state $(v_+,u_+)\in \mathbb{R}_+\times\mathbb{R}$, there exist positive constants $\delta_0$ and $\varepsilon_0$ such that the following statements hold.
        
   For any $(v_m,u_m) \in S_2(v_+,u_+)$ and $(v_-,u_-) \in R_1(v_m,u_m)$ such that 
   \[|(v_+,u_+)-(v_m-u_m)|+|(v_m-u_m)-(v_-,u_-)|<\delta_0,\]
   let $(v^r,u^r)(\frac{x}{t})$ be the 1-rarefaction \eqref{eq:rarefaction wave} with end states $(v_-,u_-)$ and $(v_m,u_m)$, and  $(\vs,\us)(x-\sigma t)$ be the 2-viscous-dispersive shock wave \eqref{viscous-dispersive-shock} with the end states $(v_m,u_m)$ and $(v_+,u_+)$. Let $(v_0,u_0)$ be any initial data such that
   \[\sum_{\pm}\left\|(v_0-v_{\pm},u_0-u_{\pm})\right\|_{L^2(\R_\pm)}+\|v_{0x}\|_{H^1(\R)}+\|u_{0x}\|_{L^2(\R)}<\e_0,\]
   where $\R_-:=-\R_+=(-\infty,0)$. Then, the Navier--Stokes--Korteweg system \eqref{eq:NSK} admits a unique global-in-time solution $(v,u)$. Moreover, there exists a Lipschitz continuous shift function $X:[0,\infty)\to \bbr$ such that $X(0)=0$ and 
   \begin{equation*}
   \begin{aligned}
      &v(t,x)-\left(v^r(x/t) + \vs(x-\sigma t -X(t)) -v_m \right)\in C(0,\infty;H^2(\R)),\\
      &u(t,x)-\left(u^r(x/t) + \us(x-\sigma t -X(t)) -v_m \right)\in C(0,\infty;H^1(\R))\cap L^2(0,\infty;H^2(\R)).
   \end{aligned}
   \end{equation*}
   In addition, we have
   \begin{equation*}
   \begin{aligned}
   &\lim_{t\to\infty}\|v(t,\cdot)-\left(v^r(\cdot/t)+v^S(\cdot-\sigma t-X(t))-v_m\right)\|_{W^{1,\infty}(\R)}=0,\\
   &\lim_{t\to\infty}\|u(t,\cdot)-\left(u^r(\cdot/t)+u^S(\cdot-\sigma t-X(t))-u_m\right)\|_{L^{\infty}(\R)}=0,
   \end{aligned}
   \end{equation*}
	and
	\begin{equation}\label{xlimit}
		\lim_{t\to\infty} |\dot{X}(t)|=0.
	\end{equation}
\end{theorem}
    
\begin{remark}
    (1) Since the convergence \eqref{xlimit} implies
    $$
    \lim_{t\rightarrow+\infty}\frac{X(t)}{t}=0,
    $$
    the shift function $X(t)$ grows at most sub-linearly as $t\to\infty$, which implies that the linear term $\sigma t$ dominates the shift $X(t)$. Thus, the shifted wave $(\vs,\us)(x-\sigma t-X(t))$ asymptotically tends to the original wave $(\vs,\us)(x-\sigma t)$.\\
    
    \noindent (2) Thanks to the $H^2$-regularity of $v$-perturbation, compared to the $H^1$-regularity of $u$, the convergence of $v(t,x)$ towards the composite wave can be obtained in $W^{1,\infty}$-norm, while the convergence of $u(t,x)$ towards the composite wave is only in $L^\infty$-norm.\\
    
    \noindent (3) The results of Theorem \ref{thm:main} include the cases of $\mu(v)=v^{-\alpha}$ and $\kappa(v)=v^{-\beta}$ for $\alpha, \beta\in\bbr$:
    \begin{equation}
    \begin{aligned}\label{eq:NSK-general}
        &v_t-u_x=0,\quad (t,x)\in\bbr_+\times\bbr,\\
        &u_t +p(v)_x=\left(\frac{u_x}{v^{\alpha+1}}\right)_x+\left(-\frac{v_{xx}}{v^{\beta+5}}+\frac{\beta+5}{2}\frac{v_x^2}{v^{\beta+6}}\right)_x,
    \end{aligned}
    \end{equation}
	which has been studied as a special case of the NSK equations \cite{CCDZ}. In particular, when $\beta=-1$, the system \eqref{eq:NSK-general} represents the one-dimensional quantum Navier--Stokes equations in the Lagrangian mass coordinate. 
\end{remark}

The rest part of the paper is organized as follows. In Section \ref{sec:prelim}, we review several preliminary items such as the properties of the relative functionals and the elementary waves, which will be used in the later analysis. Section \ref{sec:apriori} presents a priori estimates on the stability, from which we deduce the global-in-time existence and the asymptotic behavior of the solution. At the end of this section, we also introduce the notations that will be used in the following sections and appendices. In Section \ref{sec:rel_ent} and Section \ref{sec:high-order}, we provide a detailed proof of the a priori estimate using the method of $a$-contraction with shift and the estimate on the high-order perturbation. In Appendix \ref{sec:app-rarefaction-proof}--Appendix \ref{app:proof_high_order_3}, we provide lengthy proofs of several lemmas.

\section{Preliminaries}\label{sec:prelim}
    \setcounter{equation}{0}
    In this section, we present several basic estimates for the relative quantities of pressure and internal energy, discuss useful properties of elementary waves with a particular focus on decay estimates, and explain the construction of composite waves. We also introduce the augmented system and provide detailed discussions on its formulation. %Finally, we define several $O(1)$-constants and derive related estimates for these constants.

    \subsection{Estimates on the relative quantities}
   Let $F:\R_+\to \bbr$ be a differentiable function, and let $v,\overline{v}>0$. The relative functional $F(v|\overline{v})$ is defined as
    \[F(v|\overline{v}):=F(v)-F(\overline{v})-F'(\overline{v})(v-\overline{v}).\]
    When $F$ is convex, the relative quantity is non-negative and vanishes if and only if $v=\overline{v}$. We review several lower and upper bounds on the relative quantities for the pressure $p(v)=v^{-\gamma}$ and the internal energy $Q(v)=\frac{v^{1-\gamma}}{\gamma-1}$.
    
    \begin{lemma}[\cite{KV21}]\label{lem : Estimate-relative} Let $\gamma>1$ and $v_+$ be given constants. Then, there exist positive constants $C$ and $\delta_*$ such that the following assertions hold:
        \begin{enumerate}
            \item For any $v$ and $\bar{v}$ satisfying $0<\bar{v}<2v_+$ and $0<v<3v_+$,
            \begin{equation*}
            |v-\bar{v}|^2 \le C Q(v|\bar{v}), \quad |v-\bar{v}|^2 \le C p(v|\overline{v}).
            \end{equation*}
            \item For any $v,\bar{v}$ satisfying $v,\bar{v} > v_+ /2$,
            \begin{equation*}
            |p(v)-p(\bar{v})| \le C |v-\bar{v}|.
            \end{equation*}
            \item For any $0<\delta<\delta_*$ and any $(v,\bar{v}) \in \mathbb{R}^2_+$ satisfying $|p(v)-p(\bar{v})|<\delta$ and $|p(\bar{v})-p(v_+)| < \delta,$
            \begin{equation*}
            \begin{aligned}
            &p(v|\bar{v}) \le \left( \frac{\gamma +1 }{2 \gamma } \frac{1}{p(\bar{v})} +C \delta \right) |p(v)-p(\bar{v})|^2,\\ 
            &Q(v|\bar{v}) \ge \frac{p(\bar{v})^{- \frac{1}{\gamma}-1}}{2 \gamma}|p(v)-p(\bar{v})|^2- \frac{1+\gamma}{3\gamma^2}p(\bar{v})^{- \frac{1}{\gamma}-2}(p(v)-p(\bar{v}))^3,\\ 
            &Q(v|\bar{v}) \le \left( \frac{p(\bar{v})^{- \frac{1}{\gamma}-1}}{2 \gamma} +C \delta \right)|p(v)-p(\bar{v})|^2.
            \end{aligned}
            \end{equation*}
            
        \end{enumerate} 
    \end{lemma}
    
    \subsection{Rarefaction wave and smooth approximation} 

    Since the rarefaction wave $(v^r,u^r)$ is merely Lipschitz continuous, it is convenient to use its smooth approximation in the later analysis. Consider the smooth solution $w(t,x)$ to the Burgers equation:
    \begin{equation*} %\label{inviscid Bugers}
        \begin{aligned}
            &w_t+w w_x=0,\\
            &w(0,x)=w_0(x)=\frac{w_m+w_-}{2}+\frac{w_m-w_-}{2}\tanh x,
        \end{aligned}
    \end{equation*}
    where $w_-:=\lambda_1(v_-)$ and $w_m:=\lambda_1(v_m)$ where $\lambda_1(v)=-\sqrt{-p'(v)}$. Then, we define a smooth approximation $(\vr,\ur)(t,x)$ of the 1-rarefaction wave $(v^r,u^r)$ as
    \begin{equation*} %\label{smooth rarefaction}
        \begin{aligned}
            &\vr(t,x)=\lambda_1^{-1}(w(1+t,x)),\\
            &z_1(\vr,\ur)(t,x)=z_1(v_-,u_-)=z_1(v_m,u_m),
        \end{aligned}
    \end{equation*}
    Then, the smooth approximate rarefaction wave $(\vr, \ur)$ satisfies the Euler equations:
    \begin{equation*} %\label{eq: smooth rarefaction}
        \begin{aligned}
        &v^R_t-u^R_x=0, \\
        &u^R_t+p(\vr)_x=0.
        \end{aligned}
    \end{equation*}
    In the following lemma, we review the properties of the smooth rarefaction wave $(\vr,\ur)$, which are mainly derived in \cite{MN86}.
    
    \begin{lemma}\label{lem:rarefaction_property} Let $\delta_R :=|u_m-u_-| \sim |v_m-v_-|$. Then, the smooth approximate 1-rarefaction wave $(\vr, \ur)(t, x)$ satisfies the following properties. 
        \begin{enumerate}
            \item $\ur_x=\frac{2\vr}{\gamma+1}w_x>0$ and $\vr_x=\frac{(\vr)^{\frac{\gamma+1}{2}}}{\sqrt{\gamma}}u^R_x>0$, $\forall x \in \mathbb{R}$ and $t \ge 0.$
            \item For any $p \in [1,+\infty]$, there exists a positive constant $C$ such that, for all $t\ge0$,
            \begin{align*}
            &\|(v^R_x,u^R_x)\|_{L^p(\mathbb{R})} \le C \min \left\{\delta_R, \frac{\delta_R^{1/p}}{(1+t)^{1-1/p}} \right\}, \\
            &\|(v^R_{xx},u^R_{xx})\|_{L^p(\mathbb{R})} \le C \min \left\{\delta_R, \frac{1}{(1+t)} \right\}, \\
            &\|\pa_x^j(v^R,u^R)\|_{L^p(\mathbb{R})} \le C \min \left\{\delta_R, \frac{\delta_R^{1/p}}{(1+t)^{1-1/p}} ,\frac{1}{1+t}+\frac{\delta_R^{1/p-1}}{(1+t)^{2-1/p}}\right\}, \quad j = 3,4,\\
            &|u^R_{xx}| \le C |u^R_x|, \quad \forall x \in \mathbb{R}.
            \end{align*}
            \item For $x \ge \lambda_{1}(v_m)(1+t)$ and for all $t\ge0$,
            \begin{align*}
            &|(\vr,\ur)(t,x) -(v_m,u_m)| \le C \delta_R e^{-2|x-\lambda_{1}(v_m)(1+t)|} , \\
            &|(v^R_{x},u^R_{x})(t,x)| \le C \delta_R e^{-2|x-\lambda_{1}(v_m)(1+t)|}.
            \end{align*}
            \item For $x \le \lambda_{1}(v_-)t$ and for all $t\ge0$, 
            \begin{align*}
            &|(\vr,\ur)(t,x) -(v_-,u_-)| \le C \delta_R e^{-2|x-\lambda_{1}(v_-)t|} , \\
            &|(v^R_{x},u^R_{x})(t,x)| \le C \delta_R e^{-2|x-\lambda_{1}(v_-)t|}.
            \end{align*}
            \item $\displaystyle{\lim_{t \to +\infty}\sup_{x \in \mathbb{R}}}|(\vr,\ur)(t,x)-(v^r,u^r)(x/t)|=0$.
        \end{enumerate}
    \end{lemma}
    
    \begin{proof} Since the proof can be found in \cite{MN86}, except for the estimate on the third- and fourth-order derivatives, we only need to prove the estimates on the high-order derivatives. For brevity, we present the proof in Appendix \ref{sec:app-rarefaction-proof}. 
    \end{proof}

    \subsection{Viscous-dispersive shock wave}
    Next, we review the properties of the viscous-dispersive shock wave. For a given right-end state $(v_+,u_+)$ and $(v_m,u_m)\in S_2(v_+,u_+)$, let $\sigma$ be a shock speed given by the Rankine-Hugoniot condition:
    \[\sigma =\sqrt{-\frac{p(v_+)-p(v_m)}{v_+-v_m}}.\]
    Then, the viscous-dispersive 2-shock wave connecting $(v_m,u_m)$ and $(v_+,u_+)$ is given by the traveling wave solution $(v^S,u^S)(\xi)$, which is the solution to \eqref{viscous-dispersive-shock}	subjected to the boundary condition $\lim_{\xi\to-\infty}(\vs,\us)(\xi)=(v_m,u_m)$ and $\lim_{\xi\to+\infty}(\vs,\us)(\xi)=(v_+,u_+)$. The existence of the traveling wave and its properties are proved in \cite[Appendix B]{HKKL_pre}.
    
    \begin{lemma}[\cite{HKKL_pre}]\label{lem:shock-property}
        For a given right-end state $(v_+,u_+)$, there exists a positive constant $\delta_0$ such that the following statement holds. For any left-end state $(v_m,u_m) \in S_2(v_+,u_+)$ with $|v_+-v_m| \sim |u_+ - u_m|=:\delta_S<\delta_0$, there exists a unique solution $(\vs,\us)(\xi)$ to \eqref{viscous-dispersive-shock} such that $\vs(0)=\frac{v_m+v_+}{2}$. Moreover, there exists a positive constant $C$ such that the following estimates hold:
        \begin{equation}
        \begin{aligned}\label{shock-property}
        & (\us)'<0,\quad (\vs)'>0,\quad ' = d/d\xi,\\
        &C^{-1} (\vs)'(\xi) \le |(\us)'(\xi)| \le C (\vs)' (\xi), \quad \xi \in \mathbb{R},\\
        &|\vs(\xi)-v_\pm|\le C\delta_Se^{-C\delta_S|\xi|},\quad \pm \xi>0,\\
        &|(\vs)'(\xi)|\le C\delta_S^2e^{-C\delta_S|\xi|},\quad |(\vs)''(\xi)|\le C\delta_S|(\vs)'(\xi)|.
        \end{aligned}
        \end{equation}
    \end{lemma}

    \subsection{Composition of smooth approximate of rarefaction and viscous-dispersive shock}
    In order to prove the asymptotic stability of the composite waves of the rarefaction and shock waves, we consider a composite wave of the smooth approximate rarefaction wave and the viscous-dispersive shock wave shifted by a Lipschitz continuous function $X(t)$, which will be defined later (see \eqref{ODE_X}):
    \begin{equation*} %\label{eq:composite wave}
    \begin{aligned}
        \ov(t,x)&:=\vr(t,x)+\vs(x-\sigma t -X(t))-v_m ,\\
        \ou(t,x)&:=\ur(t,x)+\us(x-\sigma t -X(t))-u_m.
    \end{aligned}
    \end{equation*}
   
Then, the composition wave $(\ov,\ou)(t,x)$ satisfies the system
\begin{align*}
    &\ov_t+\dot{X} \vs_x-\ou_x=0, \\
    &\ou_t+\dot{X} \us_x+p(\overline{v})_x= \left(\mu(\overline{v})  \dfrac{\ou_x}{\ov}\right)_x + \left(\kappa(\overline{v})\left(- \dfrac{\ov_{xx}}{\ov^{5}}+\dfrac{5}{2}\dfrac{ (\ov_x)^2}{ \ov^{6}}\right) -\dfrac{1}{2} \kappa'(\overline{v}) \dfrac{(\ov_x)^2}{\ov^5} \right)_x+Q_1,
\end{align*}
where the error term $Q_1$ is split as
\begin{equation} \label{Q1}
    Q_1=Q_1^I+Q_1^R.
\end{equation}
Here, the term $Q_1^I$ is the error comes from the interactions between waves, which is given as
\begin{equation} \label{Q1I}
\begin{aligned}
    Q_1^I&=(p(\overline{v})-p(v^R)-p(v^S))_x-\left(\mu(\overline{v})\frac{\ou_x}{\ov}-\mu(v^R)\frac{\ur_x}{\vr}-\mu(v^S)\frac{\us_x}{\vs}\right)_x\\
    &-\left( \kappa(\overline{v}) \left(-\frac{\ov_{xx}}{\ov^5} + \frac{5(\ov_x)^2}{2\ov^6} \right)-\kappa(v^R) \left(-\frac{\vr_{xx}}{(\vr)^5} + \frac{5(\vr_x)^2}{2(\vr)^6} \right)-\kappa(v^S) \left(-\frac{\vs_{xx}}{(\vs)^5} + \frac{5(\vs_x)^2}{2(\vs)^6} \right)\right)_x\\
    &+\left(  \kappa'(\overline{v}) \frac{(\ov_x)^2}{2\ov^5} -\kappa'(v^R) \frac{(\vr_x)^2}{2(\vr)^5}-\kappa'(v^S) \frac{(\vs_x)^2}{2(\vs)^5}\right)_x.
\end{aligned}
\end{equation}
 On the other hand, the error term $Q^R_1$ from the rarefaction wave equation is given by
\begin{equation*} %\label{Q1R}
\begin{aligned}
      Q_1^R&=-\left(\mu(v^R)\frac{\ur_x}{\vr} \right)_x-\left(\kappa(v^R) \left(-\frac{\vr_{xx}}{(\vr)^5} + \frac{5(\vr_x)^2}{2(\vr)^6} \right) + \kappa'(v^R) \frac{(\vr_x)^2}{2(\vr)^5} \right)_x.
\end{aligned}       
\end{equation*}

\subsection{Augmented system}
    As we explained, our method is based on the relative entropy estimate, which was well-established for the hyperbolic system with dissipation. On the other hand, the NSK system \eqref{eq:NSK} has a natural entropy given by
    \[\int_{\R} \left(\frac{|u|^2}{2}+ Q(v) + \frac{\kappa(v)}{2v^5}|v_x|^2\right)\,dx,\]
    where the first two terms are classical kinetic and internal energies of the fluid and the last term corresponds to the potential energy that comes from the capillarity of the fluid. To utilize the standard theory on the relative entropy, we define an auxiliary variable $w$ as
    \[w=- \dfrac{\sqrt{\kappa(v)} }{v^{5/2}} v_x,\]
    which enables us to rewrite the entropy in terms of extended variable $U=(v,u,w)$ as
    \begin{equation}\label{entropy}
    \eta(U):=\int_{\R} \left(\frac{|u|^2}{2}+Q(v)+\frac{|w|^2}{2 }\right)\,dx.
    \end{equation}
    The idea of introducing an additional dependent variable for Korteweg-type fluid is nothing new, and it has been exploited in some previous works, such as \cite{BDD06, BDD07}. Using the definition of $w$, it satisfies 
    \[w_t=\left( -\sqrt{\kappa(v)}\dfrac{v_t }{v^{5/2}}  \right)_x=\left( -\sqrt{\kappa(v)}\dfrac{u_x }{v^{5/2}} \right)_x,\]
    while the capillary term in the momentum equation becomes
    \[\kappa(v) \left( - \dfrac{v_{xx}}{v^5} + \dfrac{5}{2} \dfrac{(v_x)^2}{v^6} \right)-\dfrac{1}{2} \kappa'(v) \dfrac{(v_x)^2}{v^5}=\sqrt{\kappa(v)}\frac{w_x}{v^{5/2}}.\]
    Therefore, the original NSK system \eqref{eq:NSK} can be rewritten as the following extended system:
    \begin{equation}\label{NSK-w}
    \begin{aligned}
        &v_t-u_x=0, \\
        &u_t+p(v)_x=\left(\mu(v) \dfrac{u_x}{v} \right)_x + \left( \sqrt{\kappa(v)}\frac{w_x}{v^{5/2}} \right)_x, \\
        &w_t=-\left( \sqrt{\kappa(v)}\frac{u_x}{v^{5/2}}\right)_x.
    \end{aligned}
    \end{equation}
    In the following, we mainly refer to \eqref{NSK-w} as the NSK system, instead of the original system \eqref{eq:NSK}. In order to apply this extended framework to the viscous-dispersive shock, we also define
   	\[w^S:=-\sqrt{\kappa(v^S)}\frac{v^S_\xi}{(v^S)^{5/2}}.\]
   	Then, the (extended) viscous-dispersive shock wave $(v^S,u^S,w^S)$ satisfies
    \begin{align}
    \begin{aligned}\label{viscous-dispersive-shock-ext}
    &-\sigma \vs_\xi-\us_\xi=0,\\
    &-\sigma \us_\xi+(p(\vs))_\xi=\left(\mu(\vs) \frac{\us_\xi}{\vs} \right)_\xi+\left( \sqrt{\kappa(\vs)} \frac{w^S_\xi}{(\vs)^{5/2}} \right)_\xi,\\
    &-\sigma w^S_\xi=-\left(\sqrt{\kappa(\vs)}\frac{\us_\xi}{(\vs)^{5/2}} \right)_\xi.
    \end{aligned}
    \end{align}
Similarly, we extend the composite wave $(\ov,\ou)$ to $(\ov,\ou,\ow)$ by introducing
\[\ow:=-\sqrt{\kappa(\ov)}\dfrac{\ov_x}{\ov^{5/2}}.
\]
Then, it satisfies
\begin{equation}
\begin{aligned}\label{superposition wave-ext}
    &\ov_t-\ou_x=-\dot{X} \vs_x, \\
    &\ou_t+p(\ov)_x= \left(\mu(\ov) \dfrac{\ou_x}{\ov}\right)_x +  \left(\sqrt{\kappa(\ov)} \dfrac{\ow_x}{\ov^{5/2}}\right)_x-\dot{X} \us_x+Q_1,\\ 
    &\ow_t= -\left( \sqrt{\kappa(\ov)} \dfrac{\ou_x}{\ov^{5/2}}\right)_x-\dot{X} \ws_x +Q_2, 
\end{aligned}
\end{equation}
where $Q_1$ is the same term defined in \eqref{Q1}, and $Q_2$ is given by
\begin{equation} \label{Q2}
    \begin{aligned}
    Q_2=- \dot{X} \left(   \sqrt{\kappa(\vs)} \dfrac{ \vs_x}{(\vs)^{5/2}} -  \sqrt{\kappa(\ov)}\dfrac{ \vs_x}{\ov^{5/2}}  \right)_x.
    \end{aligned}
\end{equation}

The advantage of the extended system \eqref{NSK-w} over the original system \eqref{eq:NSK} is that it can be written as a general form of the hyperbolic conservation laws with dissipation. As a result, we may directly use the standard relative entropy method \cite{D96,D79} to study the stability of \eqref{NSK-w}.\\

In the later analysis, we will frequently use the following constants:
\[\sigma_m:=\sqrt{-p'(v_m)}, \quad \alpha_m := \frac{\gamma+1}{2 \gamma \sigma_m p(v_m)}=\frac{p''(v_m)}{2|p'(v_m)|^2 \sigma_m},\]
which are independent of both the rarefaction strength $\delta_R$ and the shock strengths $\delta_S$. In particular, the constant $\sigma_m$ satisfies
%These constants are indeed independent of the small shock strength $\delta_S$ since $v_+/2 \le v_m \le v_+$. Then, the following estimates on the $O(1)$-constants hold:
\begin{equation}\label{shock_speed_est}
	|\sigma -\sigma_m|=\left|\sqrt{-\frac{p(v_+)-p(v_m)}{v_+-v_m}}-\sqrt{-p'(v_m)}\right| \le C \delta_S,
\end{equation}
and
\begin{align}
	\begin{aligned}\label{shock_speed_est-2}
		&\|\sigma_m^2+p'(\vs)\|_{L^\infty}=\|p'(\vs)-p'(v_m)\|_{L^\infty} \le C \delta_S, \\
		&\left\| \frac{1}{\sigma_m^2}-\frac{p(\vs)^{-\frac{1}{\gamma}-1}}{\gamma} \right\|_{L^\infty}=\left\|\frac{(v_m)^{\gamma+1}}{\gamma }-\frac{(\vs)^{\gamma+1}}{\gamma}\right\|_{L^\infty} \le C \delta_S,\\
		&\left\| \frac{1}{\sigma_m^2}-\frac{p(\ov)^{-\frac{1}{\gamma}-1}}{\gamma} \right\|_{L^\infty}=\left\|\frac{(v_m)^{\gamma+1}}{\gamma }-\frac{\ov^{\gamma+1}}{\gamma}\right\|_{L^\infty} \le C \delta_0.
	\end{aligned}
\end{align} 
%Throughout the paper, a constant $C$ denotes may change from line to line, but it is independent of the small constants such as $\delta_S,\delta_R, \e_1$ and the lifespan $T$ given in Proposition \ref{apriori-estimate}.

\section{A priori estimate and Proof of Theorem \ref{thm:main}}\label{sec:apriori}\setcounter{equation}{0}

In this section, we present a priori estimate of the solution to \eqref{eq:NSK}, from which we obtain the global existence of the solution and the time-asymptotic stability of the composite wave.
	
\subsection{Local existence}
We first state the local-in-time existence of a strong solution for the original NSK system \eqref{eq:NSK}, and therefore its equivalent form \eqref{NSK-w}, in the following proposition.
	
\begin{proposition}\label{prop:local}
	Let $\underbar{v}$ and $\underbar{u}$ be smooth monotone functions such that
	\[(\underbar{v},\underbar{u})(x) = (v_{\pm},u_{\pm} )\quad \mbox{for}\quad\pm x\ge 1.\]
	Then, for any constants $M_0$, $M_1$, $\underline{\kappa}_0$, $\overline{\kappa}_0$, $\underline{\kappa}_{1}$, and $\overline{\kappa}_1$ satisfying
    \begin{align*}
        0<M_0<M_1 \quad \text{and} \quad 0<\underline{\kappa}_1<\underline{\kappa}_0<\overline{\kappa}_0<\overline{\kappa}_1,
    \end{align*}
	there exists a finite time $T_0>0$ such that if the initial data $(v_0,u_0)$ satisfy
    \begin{align*}
    \|v_0-\underline{v}\|_{H^2(\R)}+\|u_0-\underline{u}\|_{H^1(\R)}\le M_0 
    \quad \text{and} \quad
    \underline{\kappa}_0\le v_0(x)\le\overline{\kappa}_0,\quad\forall x\in\R,
    \end{align*}
	then the Navier-Stokes-Korteweg equations \eqref{eq:NS} admit a unique solution $(v,u)$ on $[0,T_0]$ such that 
	\begin{align*}
	&v-\underline{v}\in L^\infty ([0,T_0];H^2(\R))\cap L^2([0,T_0];H^3(\R)),\\ 
    & u-\underline{u}\in L^\infty([0,T_0];H^1(\R))\cap L^2([0,T_0];H^2(\R)).
	\end{align*}
    Moreover, the solution $(v,u)$ satisfies
	\[\|v-\underline{v}\|_{L^\infty([0,T_0];H^2(\R))}+\|u-\underline{u}\|_{L^\infty([0,T_0];H^1(\R))}\le M_1\]
	and 
	\[\underline{\kappa}_1\le v(t,x)\le \overline{\kappa}_1,\quad \forall(t,x)\in [0,T_0]\times \R.\]
\end{proposition}
	
	\begin{proof}
		The local existence can be proven, as presented in \cite{S76}, by using the standard argument of constructing a sequence of approximate solutions and applying the Cauchy estimate. For the sake of brevity, the proof is omitted.
	\end{proof}

	\subsection{Construction of shift} As mentioned in the introduction, we will use the method of $a$-contraction with shift, in which the dynamical shift function $X(t)$ is crucially used to shift the viscous shock. We define the shift function $X:\bbr_+\to\bbr$ as a solution to the following ODE:
	\begin{align}
	\begin{aligned}\label{ODE_X}
	\dot{X}(t)&=-\frac{M}{\delta_S}\Bigg(\int_{\R} a\left(t,x\right) \us_x\left(x- \sigma t- X(t)\right) (u-\overline{u})\,d x\\
	&\hspace{2cm}+\frac{1}{\sigma}\int_\R a(t,x)\pa_xp\big(\vs(x- \sigma t- X(t))\big) (u-\overline{u} )\,d x\Bigg),\\
 X(0)&=0,
	\end{aligned}
	\end{align}
	where $M=\frac{5\sigma_m^3\alpha_m}{4}$ and the weight function $a=a(t,x)$ is explicitly
    defined as
\begin{equation} \label{a}
	a(t,x):=1+\frac{u_m-u^S(x-\sigma t-X(t))}{\sqrt{\delta_S}}.
\end{equation}
From this definition, it is straightforward to verify that $1\le a\le 2$ and $a$ satisfies
\begin{equation}\label{a_x}
	a_x=-\frac{\us_x}{\sqrt{\delta_S}}=\frac{\sigma\vs_x}{\sqrt{\delta_S}}>0.
\end{equation}
The existence of shift function $X(t)$ as a Lipschitz continuous solution of the \eqref{ODE_X} is guaranteed by the standard Picard–Lindel\"of theorem, see also \cite[Lemma 3.3]{ KVW23}.

    	\begin{proposition} \cite{ KVW23}
		For any $c_1,c_2,c_3>0$, there exists a constant $C>0$ such that the following holds. For any $T>0$, and any functions $v,u\in L^\infty((0,T)\times\R)$ with
		\[c_1\le v(t,x)\le c_2,\quad |u(t,x)|\le c_3,\quad \forall (t,x)\in[0,T]\times\R,\]
		the ODE \eqref{ODE_X} has a unique Lipschitz continuous solution $X$ on $[0,T]$. Moreover, we have
		\[|X(t)|\le Ct,\quad t\in[0,T].\]
	\end{proposition}

The reason for choosing the shift $X(t)$ as in \eqref{ODE_X} will become clear in Section \ref{sec:rel_ent}, where we exploit the definition of shift function to estimate the weighted relative entropy.

\subsection{A priori estimate}

We now present the a priori estimate, which is a key estimate and serves as a fundamental tool to obtain the long-time stability of the NSK equations.

\begin{proposition}\label{apriori-estimate}
	For a given state $(v_+,u_+)\in\bbr_+\times\bbr$, there exist positive constants $C_0,\delta_0$, and $\e_1$ such that the following holds:
		
	Suppose that $(v,u,w)$ is the solution to \eqref{NSK-w} on $[0,T]$ for some $T>0$, and $(\ov,\ou,\ow)$ is the composite waves defined in \eqref{superposition wave-ext}. Let $X$ be the Lipschitz continuous solution to \eqref{ODE_X} with the weight function $a$ defined in \eqref{a}. Assume that $\delta_R,\delta_S<\delta_0$ and
	\begin{align*}
	&v-\ov\in L^\infty(0,T;H^2(\bbr))\cap L^2(0,T;H^3(\R)),\\
	&u-\ou\in L^\infty(0,T;H^1(\bbr))\cap L^2(0,T;H^2(\bbr)) ,
	\end{align*}
	with
	\begin{equation}\label{smallness}
	\|v-\ov\|_{L^\infty(0,T;H^2(\bbr))}+\|u-\ou\|_{L^\infty(0,T;H^1(\bbr))}\le \e_1.
	\end{equation}
	Then, for all $0\le t\le T$,
	\begin{equation}
	\begin{aligned}\label{a-priori-1}
	&\left(\norm{v-\ov}_{L^2(\mathbb{R})}^2 +\norm{u-\ou}_{H^1(\mathbb{R})}^2+\norm{w-\ow}_{H^1(\mathbb{R})}^2\right)+\delta_S \int_0^t | \dot{X}(s)|^2 \, d s \\ 
	&\quad +\int_0^t \left( G_1+G_3+G^S_u+G^S_v+G^R \right) \, ds +  \int_0^t \left( D_{u_1} + D_{u_2} + G_{w} + D_{w_1} + D_{w_2} \right)\, ds   \\ 
	& \le C_0 \left(\norm{(v-\ov)(0,\cdot)}_{L^2(\mathbb{R})}^2 +\norm{(u-\ou)(0,\cdot)}_{H^1(\mathbb{R})}^2+\norm{w_0-\ow}_{H^1(\mathbb{R})}^2\right)+C_0\delta_R^{1/3},
	\end{aligned}
	\end{equation}
	
	where $C_0$ is independent of $T$, and
	\begin{equation}\label{good terms}
	\begin{aligned}
	&G_1:=\int_\R |a_x|\left|p(v)-p(\ov)-\frac{u-\ou}{2C_1}\right|^2\,dx,\quad G_3:=\int_\R |a_x||w-\ow|^2\,dx,\\
	&G^S_u:=\int_\R |u_x^S||u-\ou|^2\,dx,\quad G^S_v:=\int_{\R} |u^S_x||v-\ov|^2\,dx,\quad G^R:=\int_\R u^R_x |v-\overline{v}|^2\,dx,\\
	&D_{u_1}:=\int_\R |(u-\ou)_x|^2\,dx,\quad D_{u_2}:=\int_\R |(u-\ou)_{xx}|^2\,dx,\\
	&G_w:=\int_\R |w-\ow|^2\,dx,\quad D_{w_1}:=\int_{\R}|(w-\ow)_x|^2\,dx,\quad D_{w_2}:=\int_{\R}|(w-\ow)_{xx}|^2\,dx.
	\end{aligned}
	\end{equation}
	Here, $C_1$ is a positive constant defined in \eqref{C_star}.
\end{proposition}
	
	\begin{remark}
		From the smallness of $H^1$-perturbation of $v$ and the definition of $w$-variable, $w$ is equivalent to $v_x$. Therefore, the estimate \eqref{a-priori-1} is equivalent to the following formulation for the a priori estimate:
		\begin{equation}
		\begin{aligned} \label{a-priori-2}
		& \left(\norm{v-\ov}_{H^2 (\mathbb{R})}^2 +\norm{u-\ou}_{H^1(\mathbb{R})}^2\right)+\delta_S \int_0^t | \dot{X}(s)|^2 \, d s \\ 
		&\quad +\int_0^t \left( G_1+G_3+G^S_u+G^S_v+G^R \right) \, ds +  \int_0^t \left( D_{u_1} + D_{u_2} + G_{w} + D_{w_1} + D_{w_2} \right)\, ds   \\ 
		& \le C_0 \left(\norm{(v-\ov)(0,\cdot)}_{H^2 (\mathbb{R})}^2 +\norm{(u-\ou)(0,\cdot)}_{H^1(\mathbb{R})}^2\right)+C \delta_R^{1/3},
		\end{aligned}
		\end{equation}
		where $(v,u)$ is the solution to the original NSK equations \eqref{eq:NSK} and $(\ov,\ou)$ is defined in \eqref{superposition wave-ext}
	\end{remark}

\subsection{Global existence of perturbed solution} \label{Global existence}
Based on the a priori estimate \eqref{a-priori-1}, we show that the local-in-time solution constructed by Proposition \ref{prop:local} can be extended to the global solution. Let $C_0$, $\delta_0$, and $\e_1$ be positive constants given in Proposition \ref{apriori-estimate}, where $\delta_0$ and $\e_1$ are chosen to be sufficiently small as needed. We first choose $(\underline{v}(x),\underline{u}(x))$ as a smooth monotone functions satisfying
    \begin{equation}\label{ubarvbar}
        \sum_{\pm}\|(\underline{v}-v_\pm,\underline{u}-u_\pm)\|_{L^2(\R_\pm)}+\|\pa_x\underline{v}\|_{H^1(\R)}+\|\pa_x\underline{u}\|_{L^2(\R)}\le C_* \delta_0 ,
        \end{equation}
        for some constant $C_*>0$. Then, using Lemma \ref{lem:rarefaction_property} and Lemma \ref{lem:shock-property}, we have
	\begin{equation}
	\begin{aligned}\label{est-init}
	&\norm{\underline{v}(\cdot)-\ov(0,\cdot)}_{H^2(\mathbb{R})}+\norm{\underline{u}(\cdot)-\ou(0,\cdot)}_{H^1(\mathbb{R})}\\
	&\le \sum_{\pm}\norm{(\underline{v}-v_{\pm},\underline{u}-u_{\pm})}_{L^2(\R_\pm)} +\|\pa_x \underline{v}\|_{H^1(\mathbb{R})} + \|\pa_x \underline{u}\|_{L^2(\mathbb{R})}\\
 &\quad +\|(\vr-v_m,\ur-u_m)\|_{L^2(\R_+)}+ \|(\vs-v_+,\us-u_+)\|_{L^2(\R_+)} \\
 &\quad +\|(\vr-v_-,\ur-u_-)\|_{L^2(\R_-)} +\|(\vs-v_m,\us-u_m)\|_{L^2(\R_-)}  \\ 
 &\quad  + \|\vr_x\|_{H^1(\mathbb{R})} + \|\vs_x\|_{H^1(\mathbb{R})} 
 + \|\ur_x\|_{L^2(\mathbb{R})} + \|\us_x\|_{L^2(\mathbb{R})} \\ 
	&\le C_{**} \sqrt{\delta_0}.
	\end{aligned}
	\end{equation}
for some positive constant $C_{**}$. Now, we choose $\delta_0$ sufficiently small so that for any $\delta_R,\delta_S < \delta_0$, the following inequality \[ \frac{\e_1}{4(C_0+1)}-C_{**} \sqrt{\delta_0} - C_* \delta_0 -\delta_R^{1/3}>0\] holds. From this inequality, we define two positive constants $\e_*$ and $\e_0$ as
\[\e_*:=\frac{\e_1}{2(C_0+1)}- C_{**} \sqrt{\delta_0} -\delta_R^{1/3} \quad \text{and} \quad \e_0:=\frac{\e_1}{4(C_0+1)}.\]

Now, let $(v_0,u_0)$ be initial data satisfying the smallness assumption in Theorem \ref{thm:main}:
\begin{equation} \label{v0u0}
    \sum_{\pm}\left(\|v_0-v_{\pm}\|_{L^2(\R_\pm)}+\|u_0-u_\pm\|_{L^2(\R_\pm)}\right)+\|v_{0x}\|_{H^1}+\|u_{0x}\|_{L^2}<\e_0.
\end{equation}
	Using the estimates  \eqref{ubarvbar} and \eqref{v0u0}, we estimate the perturbation between initial data $(v_0,u_0)$ and smooth function $(\underline{v},\underline{u})$ as
	\begin{align} \label{v0v_}
	&\|v_0-\underline{v}\|_{H^2(\mathbb{R})}+\|u_0-\underline{u}\|_{H^1(\mathbb{R})}
    %\\
	%&\le\sum_{\pm}\left(\|v_0-v_\pm\|_{L^2(\R_\pm)}+\|u_0-u_{\pm}\|_{L^2(\R_\pm)}+\|\underline{v}-v_\pm\|_{L^2(\R_\pm)}+\|\underline{u}-u_{\pm}\|_{L^2(\R_\pm)}\right)\\
	%&\quad + \|v_{0x}\|_{H^1(\mathbb{R})}+\|u_{0x}\|_{L^2(\mathbb{R})}+\|\underline{v}_x\|_{H^1(\mathbb{R})}+\|\underline{u}_x\|_{L^2(\mathbb{R})}\\
    < \e_*.
	\end{align}
	Since $\e_1$ was small enough, we use Sobolev embedding to deduce that
    %\begin{equation} \label{v0v_}
    %    \|v_0-\underline{v}\|_{L^\infty(\mathbb{R})} \le C \e_*.
    %\end{equation}
	\[\frac{v_-}{2}\le v_0(x)\le 2v_+,\quad x\in\R.\]
	Therefore, the local-existence result in Proposition \ref{prop:local} guarantees that there exists a unique solution $(v,u)$ on a time interval $[0,T_0]$ for some $T_0 > 0$, which satisfies 
	\begin{equation}\label{est-local-1}
	\|v-\underline{v}\|_{L^\infty(0,T_0;H^2(\mathbb{R}))}+\|u-\underline{u}\|_{L^\infty(0,T_0;H^1(\mathbb{R}))}\le \frac{\e_1}{2}
	\end{equation}
	and
	\[\frac{v_-}{3}\le v(t,x)\le 3v_+,\quad (t,x)\in[0,T_0]\times\R.\]
    
   On the other hand,  we estimate the perturbation between $(\underline{v},\underline{u})$ and $(\ov,\ou)$ as in \eqref{est-init}:
	\begin{align*}
	&\|\underline{v}(\cdot)-\ov(t,\cdot)\|_{H^2(\mathbb{R})}+\|\underline{u}(\cdot)-\ou(t,\cdot)\|_{H^1(\mathbb{R})}\\
	&\le \sum_{\pm}\|(\underline{v}-v_{\pm},\underline{u}-u_{\pm})\|_{L^2(\R_\pm)}+\|\underline{v}_x\|_{H^1(\mathbb{R})}+\|\underline{u}_x\|_{L^2(\mathbb{R})} \\
    & \quad +\|(\vs-v_+,\us-u_+)(\cdot-\sigma t-X(t))\|_{L^2(\R_+)}+\|(\vr-v_m,\ur-u_m)(t,\cdot)\|_{L^2(\R_+)}\\
    & \quad +\|(\vr-v_-,\ur-u_-)(t,\cdot)\|_{L^2(\R_-)} +\|(\vs-v_m,\us-u_m)(t,\cdot-\sigma t-X(t))\|_{L^2(\R_-)}\\
	&\quad +\|\vr_x\|_{H^1(\mathbb{R})}+\|\vs_x\|_{H^1(\mathbb{R})}+\|\ur_x\|_{L^2(\mathbb{R})} +\|\us_x\|_{L^2(\mathbb{R})}\\
	&\le C\delta_R \sqrt{1+|\lambda_1| t} + C\sqrt{\delta_S} \sqrt{(1+|\sigma| t+|X(t)|)} \\ 
    &\le C\sqrt{\delta_0}(1+\sqrt{t}).
	\end{align*}
	Then, we take small time $T_1\in(0,T_0)$ satisfying $C\sqrt{\delta_0}(1+\sqrt{T_1})\le \frac{\e_1}{2}$ to ensure
	\begin{equation}\label{est-local-2}
	\|\underline{v}-\ov\|_{L^\infty(0,T_1;H^2(\mathbb{R}))}+\|\underline{u}-\ou\|_{L^\infty(0,T_1;H^1(\mathbb{R}))}\le\frac{\e_1}{2}.
	\end{equation}
	Thus, combining \eqref{est-local-1} and \eqref{est-local-2}, the smallness assumption in Proposition \ref{apriori-estimate} is satisfied in $[0,T_1]$:
	\[\|v-\ov\|_{L^\infty(0,T_1;H^2(\mathbb{R}))}+\|u-\ou\|_{L^\infty(0,T_1;H^1(\mathbb{R}))}\le\e_1.\]
  In particular, since the shift function $X(t)$ is absolutely continuous and 
  \begin{align*}
		v-\underline{v}\in L^\infty(0,T;H^2(\bbr)), \quad u-\underline{u}\in L^\infty(0,T;H^1(\bbr)) ,
		\end{align*}
  it follows that
  \begin{align*}
		v-\overline{v}\in L^\infty(0,T;H^2(\bbr)), \quad u-\overline{u}\in L^\infty(0,T;H^1(\bbr)) .
		\end{align*}

  To attain a global-in-time solution, we use a standard continuation argument. Suppose that the maximal existence time
 \[T_M:=\sup \left\{ t>0 : \sup_{s \in [0,t]}  \left( \| (v-\overline{v})(s,\cdot) \|_{H^2(\mathbb{R})}+\|(u-\overline{u})(s,\cdot) \|_{H^1(\mathbb{R})}  \|\right) \le \e_1 \right\}\]
 is finite. Then, the continuity argument  implies that
 \begin{equation}\label{norm_TM}
 \sup_{s \in [0,T_M]} \left( \| (v-\overline{v})(s,\cdot) \|_{H^2(\mathbb{R})}+\|(u-\overline{u})(s,\cdot) \|_{H^1(\mathbb{R})} \right) =\e_1.
 \end{equation}
 However, using \eqref{est-init} and \eqref{v0v_}, we can estimate the initial perturbation as 

\begin{align*}
     \| (v-\overline{v})(0,\cdot) \|_{H^2(\mathbb{R})}+\|(u-\overline{u})(0,\cdot) \|_{H^1(\mathbb{R})}
     \le  \frac{\e_1}{2(C_0+1)} - \delta_R^{1/3}.
 \end{align*}

Then, the a priori estimate in Proposition \ref{apriori-estimate} yields
 \begin{align*}
     \sup_{s\in[0,T_M]} \left( \| (v-\overline{v})(s,\cdot) \|_{H^2(\mathbb{R})}+\|(u-\overline{u})(s,\cdot) \|_{H^1(\mathbb{R})} \right)  &\le  \frac{\e_1}{2},
 \end{align*}
 which contradicts to \eqref{norm_TM}. Hence, $T_M=+\infty$ and the solution can be globally extended. Consequently, we obtain the following global estimate
	\begin{equation}
	\begin{aligned}\label{est-infinite}
	& \sup_{t>0}\left(\norm{v-\ov}_{H^2 (\mathbb{R})}^2 +\norm{u-\ou}_{H^1(\mathbb{R})}^2\right)+\delta_S \int_0^\infty | \dot{X}(t)|^2 \, d t \\ 
	&\quad +\int_0^\infty \left( G_1+G_3+G^S_u+G^S_v+G^R \right)  \, dt+  \int_0^\infty \left( D_{u_1} + D_{u_2} + G_{w} + D_{w_1} + D_{w_2} \right)\, dt  \\ 
	& \le C_0 \left(\norm{v_0-\ov}_{H^2 (\mathbb{R})}^2 +\norm{u_0-\ou}_{H^1(\mathbb{R})}^2\right) + C_0 \delta_R^{1/3}
	\end{aligned}
	\end{equation}
	and, for $t>0$,
	\begin{equation} \label{eq: X bound}
	|\dot{X}(t)|\le C_0 \left(\|(v-\overline{v})(t,\cdot)\|_{L^\infty(\mathbb{R}_+)}+\|(u-\overline{u})(t,\cdot)\|_{L^\infty(\mathbb{R}_+)}\right).
	\end{equation}

	\subsection{Time-asymptotic behavior} \label{Time-asymptotic behavior}
	We now prove the second part of Theorem \ref{thm:main}, the time-asymptotic behavior of the perturbation. Define
	\[g(t):=\|(v-\ov)_x\|_{H^1(\mathbb{R})}^2+\|(u-\ou)_x\|_{L^2(\mathbb{R})}^2.\]
	Our goal is to show that $g\in W^{1,1}(\R_+)$, which implies $\lim_{t\to\infty}g(t)=0$. Then, the Gagliardo-Nirenberg interpolation inequality and the uniform bound estimate \eqref{est-infinite}, we conclude that 
	\begin{equation} \label{eq: v,u limit}
	\lim_{t\to\infty}\left(\|(v-\ov)(t,\cdot)\|_{W^{1,\infty}(\mathbb{R})}+\|(u-\ou)(t,\cdot)\|_{L^\infty(\mathbb{R})}\right)=0.
	\end{equation}
	Moreover, combining  \eqref{eq: X bound} and \eqref{eq: v,u limit}, we obtain
	\[ \lim_{t \to \infty} |\dot{X}(t)| \le C_0 \lim_{t \to \infty} \left( \norm{(v-\ov)(t,\cdot)}_{L^\infty(\mathbb{R})} + \norm{(u-\ou)(t,\cdot)}_{L^\infty(\mathbb{R})} \right)=0.\]
	Therefore, it remains to show that $g\in W^{1,1}(\R_+)$.\\
	
	\noindent (1) $g\in L^1(\bbr_+)$: We use the definition of $w$ variable to observe that
	\begin{align*}
	|(v-\ov)_x| 
    %&= |\kappa^{-1/2} v^{5/2} w-\ok^{-1/2} \ov^{5/2} \ow| \\ 
    %&\le C|v||w-\ow|+|\ow||v-\ov|\\
	&\le C|w-\ow|+C|\ov_x||v-\ov|.
	\end{align*}
Similarly, the second-order derivative of the perturbation $v-\ov$ can estimated as
 \begin{align*}
	&|(v-\ov)_{xx}| %&\le C|v_x||w-\ow| + C |v| |(w-\ow)_x| + C | \ov_x| |v-\ov| + C |\ov_x| |(v-\ov)_x|  \\ 
    %& \le C|v_x||w-\ow|+ C|(w-\ow)_x|+C|\ov_x||v-\ov|+C|\ov_x||w-\ow| \\ 
    %&\le C|w-\ow|^2 + C|\ov_x||v-\ov||w-\ow| +C|(w-\ow)_x|+C|\ov_x||v-\ov|+C|\ov_x||w-\ow| \\
    \le C|w-\ow| +C |(w-\ow)_x| +C |\ov_x||v-\ov|.
    \end{align*}
	Applying these estimates, we obtain
	\begin{align*}
	\int_0^\infty |g(t)|\,dt
    %&=\int_0^\infty\int_{\R} |(v-\ov)_x|^2+|(v-\ov)_{xx}|^2+|(u-\ou)_x|^2\,dxdt\\
	&\le C\int_0^\infty\int_\R |w-\ow|^2+ |(w-\ow)_x|^2+ |\ov_x|^2\left|v-\ov\right|^2 +|(u-\ou)_x|^2\,dxdt\\
	&\le C\int_0^\infty (G_w +D_{w_1}+G^S_v+G^R + D_{u_1})\,d t<+\infty,
	\end{align*}
	where we used \eqref{est-infinite} in the last inequality. This implies $g\in L^1(\R_+)$.\\
	
	\noindent (2) $g'\in L^1(\R_+)$:  From the equations \eqref{NSK-w} and \eqref{superposition wave-ext}, we get the difference equation as
	\begin{equation} \label{eq:diff}
	\begin{aligned}
	&(v-\ov)_t - (u-\ou)_x = \dot{X}(t)\vs_x,\\
	&(u-\ou)_t + (p(v)-p(\ov))_x = \left( \frac{\mu(v)}{v} u_x-\frac{\mu(\overline{v})}{\ov} \ou_x \right)_x + \left( \frac{\sqrt{\kappa(v)}}{v^{5/2}} w_x - \frac{\sqrt{\kappa(\overline{v})}}{\ov^{5/2}} \ow_x\right)_x + \dot{X} \us_x -Q_1,\\
	&(w-\ow)_t = -\left( \frac{\sqrt{\kappa(v)}}{v^{5/2}} u_x - \frac{\sqrt{\kappa(\overline{v})}}{\ov^{5/2}} \ou_x \right)_x + \dot{X} \ws_x -Q_2.
	\end{aligned}
	\end{equation}
	Moreover, the third-order derivative of $(v-\ov)$ can be bounded as
	\[|(v-\ov)_{xxx}| 
	\le  C|w-\ow| +C |(w-\ow)_x| +C |(w-\ow)_{xx}| +C |\ov_x||v-\ov|.
	\]
	Then, using the above estimate and \eqref{eq:diff}, we estimate the time-integration of $g'$ as
	\begin{equation}
	\begin{aligned}\label{est:gprime}
	&\int_0^\infty|g'(t)|\,dt\\
	&=\int_0^\infty 2\left|\int_\R (v-\ov)_x(v-\ov)_{xt}\,dx+\int_\R (v-\ov)_{xx}(v-\ov)_{xxt}\,dx+\int_\R (u-\ou)_x(u-\ou)_{xt}\,dx\right|\,dt\\
	%&\le 2\int_0^\infty \left|\int_\R (v-\ov)_x\left((u-\ou)_{xx}+\dot{X}(t)\vs_{xx}\right)\,dx\right|\,dt\\ 
    %&\quad +2\int_0^\infty \left|\int_\R (v-\ov)_{xx}\left((u-\ou)_{xxx}+\dot{X}(t)\vs_{xxx}\right)\,dx\right|\,dt\\
	%&\quad + 2\int_0^\infty \Bigg|\int_\R (u-\ou)_x\Bigg(-(p(v)-p(\ov))_{xx}+ \left( \frac{\mu(v)}{v} u_x-\frac{\mu(v)(\overline{v})}{\ov} \ou_x \right)_{xx} \\
	%&\hspace{4cm}+ \left( \frac{\sqrt{\kappa(v)}}{v^{5/2}} w_x - \frac{\sqrt{\kappa(\overline{v})}}{\ov^{5/2}} \ow_x\right)_{xx} + \dot{X} \us_{xx} -(Q_1)_x \Bigg)\Bigg|\,dt\\
	&\le 2\int_0^\infty \left|\int_\R (v-\ov)_x\left((u-\ou)_{xx}+\dot{X}(t)\vs_{xx}\right)\,dx\right|\,dt\\
    &\quad + 2\int_0^\infty \left|\int_\R (v-\ov)_{xxx}\left(-(u-\ou)_{xx}-\dot{X}(t)\vs_{xx}\right)\,dx\right|\,dt\\
	&\quad + 2\int_0^\infty \Bigg|\int_\R (u-\ou)_{xx}\Bigg((p(v)-p(\ov))_{x}-\dot{X}(t) u^S_x \\
    &\hspace{4.5cm} -\left( \frac{\mu(v)}{v} u_x-\frac{\mu(\overline{v})}{\ov} \ou_x \right)_{x}-\left( \frac{\sqrt{\kappa(v)}}{v^{5/2}} w_x - \frac{\sqrt{\kappa(\overline{v})}}{\ov^{5/2}} \ow_x\right)_{x}  +Q_1 \Bigg)\Bigg|\,dt\\
	&\le C\int_0^\infty (G_w+D_{w_1}+D_{w_2}+G^S_v+G^R+D_{u_1}+D_{u_2}+|\dot{X}(t)|^2)\,dt\\
	&\quad + C\int_0^\infty \int_\R \left|\left( \frac{\mu(v)}{v} u_x-\frac{\mu(\overline{v})}{\ov} \ou_x \right)_{x}\right|^2+\left|\left( \frac{\sqrt{\kappa(v)}}{v^{5/2}} w_x - \frac{\sqrt{\kappa(\overline{v})}}{\ov^{5/2}} \ow_x\right)_{x}\right|^2 +|Q_1|^2 \,dxdt.
	\end{aligned}
	\end{equation}
	Since the first term in the right-hand side of \eqref{est:gprime} can be bounded by \eqref{est-infinite}, we only need to estimate the last three terms. For the first and second terms, we control them as
	\begin{align*}
	&\int_0^\infty \int_\R \left|\left( \frac{\mu(v)}{v} u_x-\frac{\mu(\overline{v})}{\ov} \ou_x \right)_x\right|^2\,dxdt\\
	%&=\int_0^\infty\int_\R \Bigg|\frac{1}{v}(u-\ou)_{xx}+\ou_{xx}\left(\frac{1}{v}-\frac{1}{\ov}\right)-\frac{1}{v^2}(v-\ov)_{x}(u-\ou)_x\\
	%&\hspace{2cm}-\frac{\ov_x}{v^2}(u-\ou)_x-\frac{\ou_x}{v^2}(v-\ov)_x-\ov_x\ou_x\left(\frac{1}{v^2}-\frac{1}{(\ov)^2}\right)\Bigg|^2\,dxdt\\
	&\le C\int_0^\infty\int_\R \Bigg(|(u-\ou)_{xx}|^2 + |\ou_x|^2|v-\ov|^2+|(u-\ou)_x|^2|(v-\ov)_x|^2\\
	&\hspace{3cm}+|\ov_x|^2|(u-\ou)_x|^2+|\ou_x|^2|(v-\ov)_x|^2+|\ov_x|^2|\ou_x|^2|v-\ov|^2\Bigg)\,dxdt\\
 	&\le C\int_0^\infty\left(G_w+G^S_v+G^R+D_{u_1}+D_{u_2}\right)\,dt  
    %+ C\|(v-\ov)_x\|^2_{L^\infty((0,\infty)\times\R)}\int_0^\infty D_{u_1} dt 
    <+\infty
    \intertext{and}
	&\int_0^\infty\int_\R \left|\left( \frac{\sqrt{\kappa(v)}}{v^{5/2}} w_x - \frac{\sqrt{\kappa(\overline{v})}}{\ov^{5/2}} \ow_x\right)_x\right|^2\,dxdt\\
	%&=\int_0^\infty\int_\R\Bigg|\frac{(w-\ow)_{xx}}{v^{5/2}}+\ow_{xx}\left(\frac{1}{v^{5/2}}-\frac{1}{(\ov)^{5/2}}\right)-\frac{5}{2}\frac{1}{v^{7/2}}(v-\ov)_x(w-\ow)_x\\
	%&\hspace{2cm}-\frac{5}{2}\frac{\ov_x}{v^{7/2}}(w-\ow)_x-\frac{5}{2} \frac{\ow_x}{v^{7/2}}(v-\ov)_x-\frac52\ov_x\ow_x\left(\frac{1}{v^{7/2}}-\frac{1}{(\ov)^{7/2}}\right)\Bigg|^2\,dxdt\\
	&\le C\int_0^\infty\int_\R \Bigg(|(w-\ow)_{xx}|^2+|\ow_{xx}|^2|v-\ov|^2+|(v-\ov)_x|^2|(w-\ow)_x|^2 \\
	&\hspace{3cm}+ |\ow_x|^2|v-\ov|^2 + |\ov_x|^2|(w-\ow)_x|^2 + |\ov_x|^2|\ow_x|^2|v-\ov|^2\Bigg)\,dxdt \\
% &\le C\int_0^\infty \left(G_1+G^S+D_{w_1}+D_{w_2}\right) \,dt+ C\int_0^\infty |(v-\ov)_x|^2|(w-\ow^X)_x|^2\,dxdt\\
	&\le C\int_0^\infty \left(G^S_v+G^R+D_{w_1}+D_{w_2}\right) \,dt
    %+ C\|v-\ov\|_{L^\infty(0,\infty;H^2(\R))}\int_0^\infty D_{w_1}\,dt
    <+\infty.
	\end{align*}
	For the third term, we obtain
\begin{align*}
    &\int_0^\infty \int_\mathbb{R} |Q_1|^2 \, dxdt \\
    &\le C \int_0^\infty \int_\mathbb{R} |Q_1^I|^2 \, dxdt + \int_0^\infty \int_\mathbb{R} |Q_1^R|^2 \, dxdt\\ 
    &\le C \int_0^\infty \int_\mathbb{R}\Big[ |(\vr_x,\vr_{xx},\vr_{xx}\vr_x,\vr_{xxx},|\vr_x|^3)| |\vs-v_m| + |(\vs_x,\vs_{xx},\vs_{xx}\vs_x,\vs_{xxx},|\vs_x|^3)| | \vr-v_m|  \\
    &\hspace{2.5cm} +|( \ur_x,|\vr_x|^2,\vr_{xx})| |\vs_x|+|(\us_x, |\vs_x|^2,\vs_{xx})| |\vr_x|\Big] dx\,dt \\
    &\quad + C \int_0^\infty \int_\mathbb{R} \Big[ |\ur_{xx}|^2+|\ur_x|^2|\vr_x|^2+|\vr_{xxx}|^2+|\vr_{xx}|^2|\vr_x|^2 +|\vr_x|^6 \Big] dx\,dt< +\infty.
\end{align*}
Note that, to show that the right-hand side of the above estimate is finite, we use Lemma \ref{lem: interaction estimates} and \eqref{Q1R-L2est}, which are proved in the later sections.

Thus, we conclude that $g\in W^{1,1}(\R_+)$ and complete the proof of the asymptotic behavior of the NSK equations. In the following sections, we focus on proving the a priori estimate in Proposition \ref{apriori-estimate}.

\subsection{Notations}  For clarity and simplicity, we adopt the following notations.
\begin{enumerate} 
    \item $C$ denotes a positive $O(1)$-constant that may vary from line to line but remains independent of the small parameters $\delta_0,\delta_R,\delta_S, \varepsilon_1$ and the time $T$ .
    \item We use the abbreviated notation $L^p:=L^p(\R)$ for any $p\in[1,\infty]$ and $H^1:=H^1(\R)$.
    \item To simplify expressions, we omit function arguments when their meaning is clear from context. 
     \begin{align*}
          &(\vr,\ur,w^R) := (\vr,\ur,w^R)(t,x),  &
          &(\vs,\us,\ws) := (\vs,\us,\ws)(x-\sigma t - X(t) ).
     \end{align*}
     \item For viscosity and capillarity terms, we use the following notation:
     \begin{align*}
         &\mu:=\mu(v), \quad \m=\mu(\ov), \quad \mr:=\mu(\vr), \quad \ms:=\mu(\vs),\\
         &\kappa:=\kappa(v), \quad \ok=\kappa(\ov), \quad \kr:=\kappa(\vr), \quad \ks:=\kappa(\vs).
     \end{align*}
     
\end{enumerate}
  %With the a priori estimates from Proposition \ref{apriori-estimate}, the time-asymptotic behavior follows naturally.
%For simplicity, we suppress function arguments whenever there is no ambiguity.

\section{Estimate on the weighted relative entropy with the shift}\label{sec:rel_ent}
	\setcounter{equation}{0}
	In this section, we start to prove Proposition \ref{apriori-estimate} by presenting the zeroth-order estimate on the perturbation between the solution $(u,v,w)$ to the NSK equations \eqref{NSK-w} and the composite waves $(\ov,\ou,\ow)$ satisfying \eqref{superposition wave-ext}. The main goal of this section is to derive the following lemma.
	
	\begin{lemma} \label{Main Lemma} 
		Under the hypotheses of Proposition \ref{apriori-estimate}, there exists a positive constant $C$ such that for all $t \in [0,T],$
		\begin{equation} \label{energy-est}
		\begin{aligned}
		& \left\| \left( v-\ov, u-\ou, w-\ow \right) \right\|_{L^2}^2  +  \int_0^t \left(\delta_S|\dot{X}|^2 +G_1+ G_3 + G^R + G^S_u+ G^S_v+ D_{u_1}\right)\, ds \\ 
		&\quad \le C \left\| \left( v-\ov, u-\ou, w-\ow \right)(0,\cdot) \right\|_{L^2}^2 +C \delta_R^{1/3} + C\sqrt{\delta_0} \int_0^t \| (w-\ow)_x\|_{L^2}^2 \, ds,
		\end{aligned}
		\end{equation}
		where $G_1, G_3, G^R,  G^S_u, G^S_v,$ and $D_{u_1}$ are defined in \eqref{good terms}.
	\end{lemma}

	To prove Lemma \ref{Main Lemma}, we will use the relative entropy method with a weight function, as the relative entropy is locally equivalent to the $L^2$-norm. 
    
\subsection{Wave interaction estimates}
Before we present the estimate on the relative entropy, we first estimate the interaction between the rarefaction and shock waves, as well as the error terms $Q_1$ and $Q_2$.

\begin{lemma} \label{lem: interaction estimates}
    Let $X$ be the shift defined by \eqref{ODE_X}. Under the same hypotheses as in Proposition \ref{apriori-estimate}, the following holds for all $t\in [0,T]$:
    \begin{equation}
        \begin{aligned}\label{est-interaction}
            &\|\vs_x(\vr-v_m)\|_{L^1}+\|\vr_x \vs_x \|_{L^1} \le C \delta_R \delta_S e^{-C \delta_S t},\\
            &\|\vs_x(\vr-v_m)\|_{L^2}+\|\vr_x \vs_x \|_{L^2} \le C \delta_R \delta_S^{3/2} e^{-C \delta_S t},\\
            &\|\vr_x( \vs-v_m)\|_{L^2} \le C \delta_R \delta_S e^{-C \delta_S t}, \\
            &\|Q_1^I\|_{L^2} \le C \delta_S \delta_R e^{-C \delta_S t},\quad \|Q_2\|_{L^2} \le C \e_1 \delta_R \delta_S^{3/2} e^{-C \delta_S t}.
        \end{aligned}
    \end{equation}
\end{lemma}
\begin{proof} 
	The proofs of the first three estimates in \eqref{est-interaction} are identical to those of \cite[Lemma 4.2]{KVW23}. Since the properties of the viscous-dispersive shock wave in \eqref{shock-property} are the same as those of the viscous shock wave in \cite[Lemma 2.2]{KVW23}, we can directly apply the proof of \cite{KVW23}. Therefore, we focus on estimating $Q_1^I$ and $Q_2$.\\
	
	\noindent $\bullet$ (Estimate of $Q_1^I$): Recall the definition of $Q_1^I$ in \eqref{Q1I}:
\begin{align*}
    Q_1^I&=(p(\overline{v})-p(v^R)-p(v^S))_x-\left(\m\frac{\ou_x}{\ov}-\mu^R\frac{\ur_x}{\vr}-\mu^S\frac{\us_x}{\vs}\right)_x\\
    &-\left( \ok \left(-\frac{\ov_{xx}}{\ov^5} + \frac{5(\ov_x)^2}{2\ov^6} \right)-\kr \left(-\frac{\vr_{xx}}{(\vr)^5} + \frac{5(\vr_x)^2}{2(\vr)^6} \right)-\ks \left(-\frac{\vs_{xx}}{(\vs)^5} + \frac{5((\vs)_x)^2}{2(\vs)^6} \right)\right)_x\\
    &+\left(  \ok' \frac{(\ov_x)^2}{2\ov^5} -(\kr)' \frac{(\vr_x)^2}{2(\vr)^5}-(\ks)' \frac{(\vs_x)^2}{2(\vs)^5}\right)_x.
\end{align*}
Using $\ov=\vr+\vs-v_m$, the first two terms of $Q_1^I$ can be easily estimated as
    \begin{align*}
        (p(\overline{v})-p(v^R)-p(v^S))_x & =p'(\ov)(\vr_x+\vs_x)-p'(\vr)\vr_x-p'(\vs)\vs_x\\
        &\le C \big[ |\vr_x| |\vs-v_m| + |\vs_x| |\vr-v_m|\big],
    \end{align*}
and
    \begin{align*}
        \left(\m\frac{\ou_x}{\ov}-\mr \frac{\ur_x}{\vr}-\ms \frac{\us_x}{\vs}\right)_x 
        &\le \left| \left( \ur_x \left( \frac{\m}{\ov}-\frac{\mr}{\vr} \right) \right)_x \right| +\left|\left( \us_x \left( \frac{\m}{\ov}- \frac{\ms}{\vs}\right) \right)_x \right|\\
        &\le C \big[ |\ur_{xx}| |\vs-v_m| + |\us_{xx}| |\vr-v_m| + |\ur_x||\vs_x| +|\us_x| |\vr_x|\big].
    \end{align*}
Next, we estimate the third term of $Q^I_1$ as
    \begin{align*}
        &\left| \left( \ok \left(-\frac{\ov_{xx}}{\ov^5} + \frac{5(\ov_x)^2}{2\ov^6} \right)-\kr \left(-\frac{\vr_{xx}}{(\vr)^5} + \frac{5(\vr_x)^2}{2(\vr)^6} \right)-\ks \left(-\frac{\vs_{xx}}{(\vs)^5} + \frac{5(\vs_x)^2}{2(\vs)^6} \right) \right)_x \right| \\
        &\le \left| \left(  - \ok \frac{\ov_{xx}}{\ov^5} + \kr \frac{\vr_{xx}}{(\vr)^5} +(\ks)\frac{\vs_{xx}}{(\vs)^5} \right)_x \right| +\frac{5}{2} \left| \left( \ok  \frac{(\ov_x)^2}{\ov^6} -\kr  \frac{(\vr_x)^2}{(\vr)^6} -\ks  \frac{(\vs_x)^2}{(\vs)^6} \right)_x \right|.
    \end{align*}
    However, since $\ov_{xx}=\vr_{xx}+\vs_{xx}$, we have
    \begin{align*}
        &\left| \left(  - \ok \frac{\ov_{xx}}{\ov^5} + \kr \frac{\vr_{xx}}{(\vr)^5} +\ks \frac{\vs_{xx}}{(\vs)^5} \right)_x \right|
        \\
        &=\left| \left( \vr_{xx} \left(\frac{\ok}{\ov^5}-\frac{\kr}{(\vr)^5} \right)+ \vs_{xx} \left(\frac{\ok}{\ov^5}-\frac{\ks}{(\vs)^5} \right) \right)_x \right|\\
        &\le \left| \vr_{xxx} \left(\frac{\ok}{\ov^5}-\frac{\kr}{(\vr)^5} \right)+ \vs_{xxx} \left(\frac{\ok}{\ov^5}-\frac{\ks}{(\vs)^5} \right) \right| +\left|  \vr_{xx} \left(\frac{\ok}{\ov^5}-\frac{\kr}{(\vr)^5} \right)_x+ \vs_{xx} \left(\frac{\ok}{\ov^5}-\frac{\ks}{(\vs)^5} \right)_x \right| \\
        &\le C \big[ |\vr_{xxx}| | \vs-v_m| + |\vs_{xxx}| |\vr-v_m| +  |\vr_{xx}| |\vs_x| +  |\vr_{xx}| |\vr_x| |\vs-v_m| +  |\vs_{xx}| |\vr_x| + |\vs_{xx}| |\vs_x| |\vr-v_m| \big]\\
        &\le C\left[ |(\vr_{xx}\vr_x,\vr_{xxx})| |\vs-v_m| + |(\vs_{xx}\vs_x,\vs_{xxx})| | \vr-v_m| + |\vr_{xx}||\vs_x|+|\vs_{xx}||\vr_x|\right].
    \end{align*}
    Similarly, we use   $(\ov_{x})^2=(\vr_{x})^2+(\vs_{x})^2+2 \vr_x \vs_x$ to obtain
    \begin{align*}
        &\left| \left( \ok  \frac{(\ov_x)^2}{\ov^6} -\kr  \frac{(\vr_x)^2}{(\vr)^6} -\ks  \frac{(\vs_x)^2}{(\vs)^6} \right)_x \right|\\
        &\le \left| \left( (\vr_{x})^2 \left(\frac{\ok}{\ov^6}-\frac{\kr}{(\vr)^6} \right)+ (\vs_{x})^2 \left(\frac{\ok}{\ov^6}-\frac{\ks}{(\vs)^6} \right) \right)_x \right|  +2 \left|   \left( \frac{\ok}{\ov^6}\vr_x\vs_x  \right)_x \right| \\
        &\le C \big[  |\vr_{x}\vr_{xx}| | \vs-v_m| + |\vs_{x}\vs_{xx}| | \vr-v_m|  \\ 
        &\quad +  |\vr_{x}|^2 |\vs_x| +  |\vr_{x}|^3  |\vs-v_m| +  |\vs_{x}|^2 |\vr_x| +  |\vs_{x}|^3  |\vr-v_m| + |\vr_{xx}| |\vs_x| + |\vr_{x}| |\vs_{xx}|\big] \\
        &\le C\big[ | (\vr_x\vr_{xx},|\vr_x|^3 ) | |\vs-v_m| + | (\vs_x\vs_{xx},|\vs_x|^3 ) | |\vr-v_m| \\ 
        &\quad +|( |\vr_x|^2,\vr_{xx})| |\vs_x|+|( |\vs_x|^2,\vs_{xx})| |\vr_x|\big].
    \end{align*}
    Finally, the fourth term of $Q^I_1$ can be estimated by using similar argument as
    \begin{align*}
        &\frac{1}{2}\left| \left(  \ok' \frac{(\ov_x)^2}{\ov^5} -(\kr)' \frac{(\vr_x)^2}{(\vr)^5}-(\ks)' \frac{(\vs_x)^2}{(\vs)^5}\right)_x \right|\\
        &\le \frac{1}{2} \left| \left( (\vr_{x})^2 \left(\frac{\ok'}{\ov^5}-\frac{(\kr)'}{(\vr)^5} \right)+ (\vs_{x})^2 \left(\frac{\ok'}{\ov^5}-\frac{(\ks)'}{(\vs)^5} \right) \right)_x \right|  + \left|   \left( \frac{\ok'}{\ov^5}\vr_x\vs_x  \right)_x \right| \\
        &\le C \big[  |\vr_{x}\vr_{xx}| | \vs-v_m| + |\vs_{x}\vs_{xx}| | \vr-v_m|  \\ 
        &\quad +  |\vr_{x}|^2 |\vs_x| +  |\vr_{x}|^3  |\vs-v_m| +  |\vs_{x}|^2 |\vr_x| +  |\vs_{x}|^3  |\vr-v_m| + |\vr_{xx}| |\vs_x| + |\vr_{x}| |\vs_{xx}|\big] \\
        &\le C\big[ | (\vr_x\vr_{xx},|\vr_x|^3 ) | |\vs-v_m| + | (\vs_x\vs_{xx},|\vs_x|^3 ) | |\vr-v_m| \\ 
        &\quad +|( |\vr_x|^2,\vr_{xx})| |\vs_x|+|( (\vs_x)^2,\vs_{xx})| |\vr_x|\big].
    \end{align*}
    Combining all the estimates above, we obtain
    \begin{align*}
        |Q_1^I| &\le C\Big[ |(\vr_x,\vr_{xx},\vr_{xx}\vr_x,\vr_{xxx},|\vr_x|^3)| |\vs-v_m| + |(\vs_x,\vs_{xx},\vs_{xx}\vs_x,\vs_{xxx},|\vs_x|^3)| | \vr-v_m|  \\
        &\quad +|( \ur_x,|\vr_x|^2,\vr_{xx})| |\vs_x|+|(\us_x, |\vs_x|^2,\vs_{xx})| |\vr_x|\Big],
    \end{align*}
and using Lemma \ref{lem:rarefaction_property} and Lemma \ref{lem:shock-property}, one has
\[\|Q_1^I\|_{L^2} \le C \delta_S \delta_R e^{-C \delta_S t}. \]

\noindent $\bullet$ (Estimate of $Q_2$): Again, we recall the definition of $Q_2$
\[ Q_2= - \dot{X} \left(   \sqrt{\kappa(\vs)} \dfrac{ \vs_x}{(\vs)^{5/2}} -  \sqrt{\kappa(\ov)}\dfrac{ \vs_x}{\ov^{5/2}}  \right)_x,\]
from which we directly obtain 
\begin{align*}
    |Q_2| &\le C |\dot{X}| \left[ |\vs_{xx}| |\vr-v_m| + |\vs_{x}| |\vr_x|\right].
\end{align*}
Therefore, we again use Lemma \ref{lem:rarefaction_property} and Lemma \ref{lem:shock-property} to derive
\[\|Q_2\|_{L^2} \le C \e_1 \delta_R \delta_S^{3/2} e^{-C \delta_S t}.\]
\end{proof}

\subsection{Relative entropy method} As we mentioned in Section \ref{sec:prelim}, the NSK system \eqref{NSK-w} can be written in the general form of the hyperbolic system as follows:
\begin{equation}\label{eq:NS-abs}
\partial_t U +\partial_x A(U) = \partial_x (M(U) \partial_x D \eta (U)),
\end{equation}
where 
\[U:=\begin{pmatrix}
    v\\u\\w
\end{pmatrix},\, 
A(U):=\begin{pmatrix}
-u\\  p(v) \\0
\end{pmatrix},\, 
D\eta(U):=\begin{pmatrix}
-p(v)\\  u \\w,
\end{pmatrix},\, M(U):=
\begin{pmatrix}
0 &0 &0\\
0 &\frac{\mu(v)}{v}  &\frac{\sqrt{\kappa(v)}}{v^{5/2}}\\
0 &-\frac{\sqrt{\kappa(v)}}{v^{5/2}} &0
\end{pmatrix}.\]
Here, $\eta$ is the convex entropy defined as in \eqref{entropy} and $D\eta$ denotes the gradient of $\eta$ with respect to $U=(v,u,w)$. Similarly, the composite wave $\overline{U}$ defined as
\[\overline{U}=
\begin{pmatrix}
    \ov\\\ou\\\ow
\end{pmatrix}=
\begin{pmatrix}
    \vr+\vs-v_m \\ \ur+\us-u_m\\ -\frac{\sqrt{\kappa(\ov)}}{\ov^{5/2}}\ov_x
\end{pmatrix}\] 
satisfies the following equation
\begin{equation}\label{eq:composite_wave}
\partial_t\overline{U} +\partial_x A(\overline{U})=\partial_x \left(M(\oU) \partial_x D \eta(\oU) \right) -\dot{X} \partial_x U^S+ 
\begin{pmatrix}
0 \\ Q_1 \\ Q_2    
\end{pmatrix}
, \quad U^S:=\begin{pmatrix}
v^S\\u^S\\w^S
\end{pmatrix},
\end{equation}
where $Q_1$ and $Q_2$ are defined in \eqref{Q1} and \eqref{Q2}. We now define the relative entropy between $U$ and $\overline{U}$ as
\[\eta(U|\overline{U}):=\eta(U)-\eta(\overline{U})-D\eta(\overline{U})(U-\overline{U}).\]
Also, the relative flux $A(U|\overline{U})$ and the relative entropy flux $G(U;\overline{U})$ are defined as
\[A(U|\overline{U})=A(U)-A(\overline{U})-DA(\overline{U})(U-\overline{U}),\]
and 
\[G(U;\overline{U})=G(U)-G(\overline{U})-D\eta(\overline{U})(A(U)-A(\overline{U})),\]
respectively, where $G$ is the entropy flux for $\eta$ satisfying the condition $D_iG(U)=\sum_{k=1}^3 D_k \eta(U)D_iA_k(U)$ for $i=1,2,3$. For the system \eqref{NSK-w}, we can choose $G(U)=p(v)u$. Therefore, all the relative quantities are explicitly computed as
\begin{align}
\begin{aligned}\label{relative_functional}
&\eta(U|\overline{U})=
\frac{|u-\overline{u}|^2}{2}+Q(v|\overline{v})+\frac{|w-\overline{w}|^2}{2},\quad
A(U|\overline{U})=\begin{pmatrix}
0\\
p(v|\overline{v})\\
0
\end{pmatrix},\\
&G(U;\overline{U})=(p(v)-p(\overline{v}))(u-\overline{u}).
\end{aligned}
\end{align}

We start with the estimate of the weighted relative entropy.

\begin{lemma} %\label{lem:rel-ent}
    Let $a$ be the weight function defined in \eqref{a} and $X:[0,T]\to\bbr$ be any Lipschitz continuous function. Let $U$ be a solution to \eqref{eq:NS-abs},  and $\overline{U}$ be the composite wave satisfying \eqref{eq:composite_wave}. Then,
\begin{equation}\label{est-weight-rel-ent} 
\frac{d}{dt}\int_\mathbb{R} a(t,x) \eta (U(t,x))|\oU(t,x)) \, dx=\dot{X}(t)Y+\mathcal{J}^{\textup{bad}}-\mathcal{J}^{\textup{good}},
\end{equation}
where the terms $Y$, $\mathcal{J}^{\textup{bad}}$, and $\mathcal{J}^{\textup{good}}$ are defined as
\begin{align*}
Y&:=-\int_\mathbb{R} a_x \eta(U|\overline{U})\,d x +\int_\mathbb{R} a D^2\eta(\overline{U})U^S_x (U-\overline{U})\,d x,\\
\mathcal{J}^{\textup{bad}}&:=\int_\mathbb{R} a_x (p(v)-p(\ov))(u-\ou) \, dx-\int_\mathbb{R} a\us_x p(v | \ov) \,d x  -\int_\mathbb{R}  a_x \frac{\mu(u-\ou) (u-\ou)_x}{v}  \, dx\\
&\quad - \int_\mathbb{R} a_x \frac{\sqrt{\kappa} (u-\ou) (w-\ow)_x}{v^{5/2}}  \, dx+ \int_\mathbb{R} a_x \frac{\sqrt{\kappa} (w-\ow)(u-\ou)_x}{v^{5/2}}  \, dx \\
&\quad-  \int_\mathbb{R} a_x   (u-\ou) \left( \frac{\mu}{v}-\frac{\overline{\mu}}{\ov} \right) \ou_x \, dx 
 - \int_\mathbb{R} a_x (u-\ou) \left( \frac{\sqrt{\kappa}}{v^{5/2}}-\frac{\sqrt{\ok}}{\ov^{5/2}} \right) \ow_x  \, dx  \\
 &\quad+ \int_\mathbb{R} a_x (w-\ow) \left( \frac{\sqrt{\kappa}}{v^{5/2}}-\frac{\sqrt{\ok}}{\ov^{5/2}} \right) \ou_x  \, dx -  \int_\mathbb{R} a (u-\ou)_x \left( \frac{\mu}{v}-\frac{\overline{\mu}}{\ov} \right) \ou_x \, dx \\
&\quad -\int_\mathbb{R} a (u-\ou)_x \left( \frac{\sqrt{\kappa}}{v^{5/2}}-\frac{\sqrt{\ok}}{\ov^{5/2}} \right) \ow_x \, dx + \int_\mathbb{R} a (w-\ow)_x \left( \frac{\sqrt{\kappa}}{v^{5/2}}-\frac{\sqrt{\ok}}{\ov^{5/2}} \right) \ou_x \, dx,\\
&\quad-\int_\mathbb{R} a (u-\ou)  Q_1  \, dx -\int_\mathbb{R} a(w-\ow) Q_2 \, dx ,\\
\mathcal{J}^{\textup{good}}&:= \frac{\sigma}{2}\int_\mathbb{R} a_x |u-\ou|^2 \, dx +\sigma \int_\mathbb{R} a_x Q(v|\ov) \, dx+\frac{\sigma}{2}\int_\mathbb{R} a_x |w-\ow|^2 \, dx\\
&\quad+\int_\mathbb{R} a \ur_x p(v|\ov) \, dx +  \int_\mathbb{R} a \frac{\mu|(u-\ou)_x|^2}{v} \,dx.
\end{align*}
\end{lemma}
\begin{proof}
Since $U$ and $\overline{U}$ satisfy general form of the hyperbolic system \eqref{eq:NS-abs} and \eqref{eq:composite_wave}, we can use a similar estimate as in \cite[Lemma 2.3]{KV21} (see also \cite[Lemma 4.2]{HKKL_pre}). Precisely, we have 
\begin{align*}
\frac{d}{dt} \int_\mathbb{R} a \eta (U|\oU) \, dx &=\dot{X}(t)Y-\sigma\int_\mathbb{R} a_x\eta(U|\oU) \, dx+ \sum_{i=1}^5 I_{1i},
\end{align*}
where

\begin{align*}
I_{11}&:=-\int_\mathbb{R} a G(U;\oU)_x \, dx,\quad I_{12}:=-\int_\mathbb{R} a (D \eta (\oU))_x A(U| \oU) \, dx,\\
I_{13}&:=\int_\mathbb{R} a \left(D\eta(U)-D\eta(\oU) \right) \left(M(U)\left(D\eta (U)-D\eta(\oU) \right)_x\right)_x\, dx,\\
I_{14}&:=\int_\mathbb{R} a \left( D \eta(U)-D\eta(\oU)\right)\left( (M(U) -M(\oU))(D \eta(\oU))_x\right)_x\,dx,\\
I_{15}&:=\int_\mathbb{R}a (D\eta)(U|\oU) \left( M(\oU) \partial_x D \eta(\oU) \right)_x\, dx,\quad I_{16}:=-\int_\mathbb{R} a D^2 \eta(\oU)(U-\oU) \begin{pmatrix} 0 \\ Q_1 \\ Q_2 \end{pmatrix} \, dx.
\end{align*} 
Using \eqref{relative_functional}, we explicitly compute each term as
\begin{align*}
I_{11}&=\int_\R a_x (p(v)-p(\ov) )(u-\ou)\, dx, \quad I_{12}=-\int_\R a \ou_x p(v|\ov) \, dx=-\int_{\R}a(u^R_x+u^S_x)p(v|\ov)\,dx, \\
I_{13}&=- \int_\R a \frac{\mu|(u-\ou)_x|^2}{v} \, dx \\
&\quad -\int_\R a_x \left(  \frac{\mu (u-\ou)(u-\ou)_x}{v}+\frac{\sqrt{\kappa}(u-\ou)(w-\ow)_x}{v^{5/2}}-\frac{\sqrt{\kappa}(w-\ow)(u-\ou)_x}{v^{5/2}}\right) dx,\\
I_{14}&=-\int_\R a_x\Bigg( (u-\ou)\left(\frac{\mu}{v}-\frac{\m}{\ov} \right)\ou_x+ (u-\ou)\left( \frac{\sqrt{\kappa}}{v^{5/2}}-\frac{\sqrt{\ok}}{\ov^{5/2}} \right)\ow_x-(w-\ow) \left( \frac{\sqrt{\kappa}}{v^{5/2}}-\frac{\sqrt{\ok}}{\ov^{5/2}} \right)\ou_x \Bigg) dx \\
&\quad  -\int_\R a \Bigg( (u-\ou)_x\left(\frac{\mu}{v}-\frac{\m}{\ov} \right)\ou_x+ (u-\ou)_x\left( \frac{\sqrt{\kappa}}{v^{5/2}}-\frac{\sqrt{\ok}}{\ov^{5/2}} \right)\ow_x-(w-\ow)_x \left( \frac{\sqrt{\kappa}}{v^{5/2}}-\frac{\sqrt{\ok}}{\ov^{5/2}} \right)\ou_x \Bigg) dx,\\
I_{15}&=0, \quad
I_{16}=-\int_\mathbb{R} a (u-\ou) Q_1 \, dx -\int_\R a (w-\ow) Q_2 \, dx.
\end{align*}
After expanding the terms, one can obtain the desired estimate.
\end{proof}

\subsection{Decomposition of the right-hand side of \eqref{est-weight-rel-ent}} 

Now, we decompose the terms on the right-hand side of \eqref{est-weight-rel-ent}. First of all, we need to control the terms in $\mathcal{J}^{\textup{bad}}$, and the most problematic term is the first term:
\[\int_\R a_x(p(v)-p(\overline{v}))(u-\overline{u})\,dx,\]
which is a cross term between $v$ and $u$. Therefore, we manipulate this cross-term and separate the terms regarding $v$ and $u$ variables by using the other terms. To this end, we need the following lemma, which extracts an exact quadratic structure $|p(v)-p(\overline{v})|^2$ from the relative quantities.
  
\begin{lemma} %\label{lem:quad}
    There exists a positive constant $C_1$ such that 
    \begin{align}
    \begin{aligned}\label{quadratic estimate}
        -\int_{\mathbb{R}}& a \us_x p(v|\ov) \, dx-\sigma \int_{\mathbb{R}} a_x Q(v|\ov) \, dx\\
        &\le -C_1\int_{\R}a_x|p(v)-p(\ov)|^2\,dx - \frac{\sigma}{2\gamma}\int_\mathbb{R} a_x \left( p(\ov)^{-\frac{1}{\gamma}-1} -p(\vs)^{-\frac{1}{\gamma}-1} \right) |p(v)-p(\ov)|^2\, dx \\ 
        &\quad +C\sqrt{\delta_S}(\sqrt{\delta_S}+\delta_0) \int_\R a_x\big|p(v)-p(\ov)\big|^2 \, d x+C\int_\R a_x\big|p(v)-p(\ov)\big|^3 \, d x.
    \end{aligned}
    \end{align}
\end{lemma}
\begin{proof}
    The proof shares a similar estimate with \cite[Lemma 4.3]{HKKL_pre}, except that we consider the composite wave instead of a single shock. Let us define
    \begin{align*}
        I_{21}:=-\int_{\mathbb{R}} a \us_x p(v|\ov) \, dx =\int_{\mathbb{R}}  a \sqrt{\delta_S} a_x p(v | \ov) \,dx, \quad I_{22}:=\sigma \int_{\mathbb{R}} a_x Q(v|\ov) \, dx,
    \end{align*}
	where we use \eqref{a_x}. We use Lemma \ref{lem : Estimate-relative} and 
    \[|p(\ov)-p(v_m)| \le C \left( |\vs-v_m| + |\vr-v_m|\right) \le C \delta_0 \] 
    to estimate $I_{21}$ as
    \begin{align*}
        I_{21} &\le \frac{1}{2} \left( (\sqrt{\delta_S}+\delta_S)  \frac{\gamma+1}{ \gamma } \frac{1}{p(v_m)} \right) \int_\mathbb{R} a_x  |p(v)-p(\ov)|^2 \, dx \\
         &\quad +C (\sqrt{\delta_S}+\delta_S) \delta_0 \int_\mathbb{R}a_x |p(v)-p(\ov)|^2 \, dx +C \int_\mathbb{R} a_x |p(v)-p(\ov)|^3 \, dx.
    \end{align*}
    Next, we use Lemma \ref{lem : Estimate-relative}, \eqref{shock_speed_est}, and \eqref{shock_speed_est-2} to estimate $I_{22}$ as
        \begin{align*}
        -I_{22} &\le -\frac{1}{2 \sigma_m}  \int_\mathbb{R} a_x |p(v)-p(\ov)|^2 \, dx - \frac{\sigma}{2\gamma}\int_\mathbb{R} a_x \left( p(\ov)^{-\frac{1}{\gamma}-1} -p(\vs)^{-\frac{1}{\gamma}-1} \right) |p(v)-p(\ov)|^2\, dx \\ 
        &\quad + C \delta_S \int_\mathbb{R} a_x |p(v)-p(\ov)|^2 \, dx + C \int_\mathbb{R} a_x |p(v)-p(\ov)|^3 \, dx.
        \end{align*}
        Therefore, one can obtain
        \begin{align*}
        \begin{aligned} %\label{I21I22}
        I_{21} -I_{22} 
        &\le - C_1  \int_\R a_x |p(v)-p(\ov)|^2 \, d x  - \frac{\sigma}{2\gamma}\int_\mathbb{R} a_x \left( p(\ov)^{-\frac{1}{\gamma}-1} -p(\vs)^{-\frac{1}{\gamma}-1} \right) |p(v)-p(\ov)|^2\, dx \\ 
        &\quad +C\sqrt{\delta_S}(\sqrt{\delta_S}+\delta_0) \int_\R a_x\big|p(v)-p(\ov)\big|^2 \, d x+C\int_\R a_x\big|p(v)-p(\ov)\big|^3 \, d x,
        \end{aligned}
        \end{align*}
        where
        \begin{equation}\label{C_star}
        C_1 = \frac{1}{2}\left(\frac{1}{\sigma_m}-\sqrt{\delta_S}\frac{\gamma+1}{\gamma}\frac{1}{p(v_m)}\right).
        \end{equation}
    \end{proof}
    Using \eqref{quadratic estimate}, we estimate the first two terms of $\mathcal{J}^{\textup{bad}}$ and the second term of $\mathcal{J}^{\textup{good}}$ as 
    \begin{align*}
        &\int_{\R}a_x (p(v)-p(\ov))(u-\ou)\,dx-\int_{\R}a\us_xp(v|\ov)\,dx-\sigma\int_{\R}a_xQ(v|\ov)\,dx\\
        &\le  \int_{\R} a_x (p(v)-p(\ov))(u-\ou)\,dx - C_1 \int_{\R} a_x |p(v)-p(\ov)|^2\, dx \\ 
        &\quad - \frac{\sigma}{2\gamma}\int_\mathbb{R} a_x \left( p(\ov)^{-\frac{1}{\gamma}-1} -p(\vs)^{-\frac{1}{\gamma}-1} \right) |p(v)-p(\ov)|^2\, dx \\ 
        &\quad +C\sqrt{\delta_S}(\sqrt{\delta_S}+\delta_0) \int_\R a_x\big|p(v)-p(\ov)\big|^2 \, d x+C\int_\R a_x\big|p(v)-p(\ov)\big|^3 \, d x\\
        &=  -C_1\int_{\R} a_x \left| (p(v) -p(\ov))- \frac{u-\ou}{2 C_1} \right|^2 \,dx+ \frac{1}{4C_1} \int_{\R}a_x |u-\ou|^2\,dx \\ 
        &\quad - \frac{\sigma}{2\gamma}\int_\mathbb{R} a_x \left( p(\ov)^{-\frac{1}{\gamma}-1} -p(\vs)^{-\frac{1}{\gamma}-1} \right) |p(v)-p(\ov)|^2\, dx \\ 
        &\quad +C\sqrt{\delta_S}(\sqrt{\delta_S}+\delta_0) \int_\R a_x\big|p(v)-p(\ov)\big|^2 \, d x+C\int_\R a_x\big|p(v)-p(\ov)\big|^3 \, d x.
    \end{align*}

    Using the above estimate, the right-hand side of \eqref{est-weight-rel-ent} can be bounded and decomposed as
    \begin{equation}\label{est-1}
    \frac{d}{dt}\int_{\bbr}a\eta(U|\oU)\,dx \le \dot{X}(t) Y + \sum_{i=1}^6 B_i-\sum_{i=1}^3 \mathcal{G}_i - \mathcal{G}^R - \mathcal{D}
    \end{equation}
  where
    \begin{align*}
    &B_1 :=\frac{1}{4 C_1} \int_\mathbb{R} a_x  |u-\ou|^2 \, dx,\\
    &B_2:= -\frac{\sigma}{2\gamma}\int_\mathbb{R} a_x \left( p(\ov)^{-\frac{1}{\gamma}-1} -p(\vs)^{-\frac{1}{\gamma}-1} \right) |p(v)-p(\ov)|^2\, dx \\
    &\hspace{2cm} + C\sqrt{\delta_S}(\sqrt{\delta_S}+\delta_0) \int_\R a_x|p(v)-p(\ov)|^2 \, d x 
    +C\int_\R a_x|p(v)-p(\ov)|^3 \, d x,\\
    &B_3:=-\int_\mathbb{R} a_x \left( \frac{ \mu (u-\ou)  (u-\ou)_x}{v}+  \frac{\sqrt{\kappa}  (u-\ou) (w-\ow)_x }{v^{5/2}} -  \frac{\sqrt{\kappa} (w-\ow) (u-\ou)_x}{v^{5/2}}\right) \, dx, \\
    &B_4:=-\int_\mathbb{R} a_x \left( (u-\ou) \left(\left( \frac{\mu}{v}-\frac{\m}{\ov} \right) \ou_x+  \left( \frac{\sqrt{\kappa}}{v^{5/2}}-\frac{\sqrt{\ok}}{\ov^{5/2}} \right) \ow_x \right) -(w-\ow)  \left( \frac{\sqrt{\kappa}}{v^{5/2}}-\frac{\sqrt{\ok}}{\ov^{5/2}} \right) \ou_x \right) \, dx, \\
    &B_5:=-\int_\mathbb{R} a \left( (u-\ou)_x\left(\left( \frac{\mu}{v}-\frac{\m}{\ov} \right) \ou_x + \left( \frac{\sqrt{\kappa}}{v^{5/2}}-\frac{\sqrt{\ok}}{\ov^{5/2}} \right) \ow_x\right)-(w-\ow)_x \left( \frac{\sqrt{\kappa}}{v^{5/2}}-\frac{\sqrt{\ok}}{\ov^{5/2}} \right) \ou_x\right) \, dx, \\
   &B_6:=-\int_\mathbb{R} a \left( (u-\ou) Q_1 + (w-\ow) Q_2 \right) \, dx,
    \end{align*} 
and
    \begin{align*}
    &\mathcal{G}_1:=C_1 \int_\mathbb{R} a_x \left|p(v)-p(\ov)-\frac{u-\ou}{2C_1} \right|^2 \, d x, \quad \mathcal{G}_2:=\frac{\sigma}{2} \int_\mathbb{R} a_x |u-\ou|^2  \, dx,\quad \mathcal{G}_3:=\frac{\sigma}{2} \int_\mathbb{R} a_x |w-\ow|^2 \, dx,\\
    &\mathcal{G}^R:=\int_\mathbb{R} a \ur_x p(v|\ov) \, dx,\quad \mathcal{D}:=  \int_\mathbb{R} a \frac{\mu |(u-\ou)_x|^2}{v} \,dx.
    \end{align*}
    Moreover, recall that
    \begin{align*}
    Y&=-\int_{\bbr} a_x \eta(U|\oU)\,dx +\int_\bbr a D^2\eta(\oU)U^S_x(U-\oU)\,dx\\
    & = -\int_{\bbr} a_x \frac{|u-\ou|^2}{2} \, dx -\int_\mathbb{R} a_x Q(v|\ov)\, dx -\int_\mathbb{R} a_x\frac{|w-\ow|^2}{2}\,dx \\ 
    &\quad +\int_{\bbr} a \us_x (u-\ou)\,dx -\int_{\bbr} a p'(\ov)\vs_x (v-\ov)\,dx+\int_\mathbb{R} a \ws_x(w-\ow) \, dx.
    \end{align*}
    Therefore, we also decompose $Y$ as $Y= \sum_{i=1}^6 Y_{i}$, where
    \begin{align*}
    Y_1 &:=\int_\mathbb{R} a \us_x (u-\ou) \,dx,\quad Y_2:=\frac{1}{\sigma}\int_\mathbb{R} a p'(\vs)\vs_x (u-\ou) \,dx,\\
    Y_3 &:=-\int_\mathbb{R} a (p'(\ov)-p'(\vs)) \vs_x (v-\ov) \, dx, \quad Y_4 :=-\int_\mathbb{R} a p'(\vs)\vs_x \left(v-\ov+\frac{u-\ou}{\sigma}\right) \, dx, \\
    Y_5 &:=-\int_{\mathbb{R}} a_x \left( \frac{|u-\ou|^2}{2} + Q(v|\ov) +\frac{|w-\ow|^2}{2} \right) \, dx,\quad  Y_6 := \int_\mathbb{R} a \ws_x (w-\ow) \, dx.
    \end{align*}
    Finally, since the shift function $X(t)$ is defined as \eqref{ODE_X}, which can be expressed as  
    \begin{equation*}
    \dot{X} = -\frac{M}{\delta_S}(Y_{1}+Y_{2}),\quad X(0)=0,
    \end{equation*}
    the term $\dot{X}Y$ in \eqref{est-1} becomes
    \[\dot{X}Y= -\frac{\delta_S}{M}|\dot{X}|^2+\dot{X}\sum_{i=3}^6Y_{i}\le -\frac{3\delta_S}{4M}|\dot{X}|^2+\frac{C}{\delta_S}\sum_{i=3}^6 |Y_i|^2,\]
    where we used Young's inequality. To sum up, the right-hand side of \eqref{est-1} can be decomposed as
    \begin{align}
    \begin{aligned}\label{est}
    \frac{d}{dt}\int_{\bbr}a\eta(U|\oU)\,dx 
    &\le \underbrace{-\frac{\delta_S}{2M}|\dot{X}|^2 + B_1 -\mathcal{G}_2-\frac{3}{4}\mathcal{D}}_{=: \mathcal{R}_1}\\
    &\quad \underbrace{-\frac{\delta_S}{4M}|\dot{X}|^2 + \frac{C}{\delta_S}\sum_{i=3}^7Y_{i} +\sum_{i=2}^7B_i -\mathcal{G}_1-\mathcal{G}_3-\mathcal{G}^R-\frac{1}{4}\mathcal{D}}_{=:\mathcal{R}_2},
    \end{aligned}
    \end{align}
    where the term $\mathcal{R}_1$ is the leading order term and the term $\mathcal{R}_2$ is the high-order term.

    \subsection{Estimate of $\mathcal{R}_1$} \label{Est-main-part} 
	The estimate of the leading order term $\mathcal{R}_1$ is now well-established in the previous literature \cite{HKKL_pre,KVW23}, using the Poincar\'e-type (or Hardy-Legendre-type) inequality. Therefore, we present the main steps and idea of estimating $\mathcal{R}_1$ and refer to the previous literature for the details. 
	
	We start with defining a new variable $y$ for each fixed time $t$ as
	\begin{equation*} 
	y:=\frac{u_m -\us(x-\sigma t-X(t))}{\delta_S},
	\end{equation*}
	which satisfies
	\begin{equation*}
	\frac{d y}{d x} = -\frac{\us_x (x-\sigma t-X(t))}{\delta_S} >0,\quad\mbox{and}\quad\lim_{x\to -\infty} y=0, \quad \lim_{x \to \infty}y=1.
	\end{equation*}
	Then, for any $f:[0,1] \to \R$ with $\int_0^1 y(1-y) |f'|^2 \, dy <\infty$, the following Poincar\'e-type inequality holds \cite[Lemma 2.9]{KV21}:
	\begin{equation}\label{eq:Hardy-Legendre}
		\int_0^1 \left| f-\int_0^1 f \,dy \right|^2 dy \le \frac{1}{2} \int_0^1 y(1-y)|f'|^2 \,dy.
	\end{equation}
	We will apply \eqref{eq:Hardy-Legendre} to a particular function $f:[0,1]\to\R$ defined as
	\begin{equation*} 
	f:=\left(u(t,\cdot)-\left( \ur(t,\cdot)+\us(\cdot-\sigma t-X(t) )-u_m \right)   \right)\circ y^{-1},
	\end{equation*}
	that is $f(y)=u(t,x)-\ou(t,x)$. Below, we represent the terms in $\mathcal{R}_1$ by using $f$ and $y$, and then use \eqref{eq:Hardy-Legendre} to derive the estimate \eqref{R1} on $\mathcal{R}_1$ below.\\
	
	\noindent $\bullet$  (Estimate of $\frac{\delta_S}{2M}|\dot{X}|^2$): 
	We use the change of variables for $y$ to represent $Y_1$ and $Y_2$ as 
 \begin{align*}
     &Y_1=\int_{\R} a \us_x (u-\ou)\,dx
	=-\delta_S \int_0^1 (1+\sqrt{\delta_S}\, y) f \,dy \\
     &Y_2=-\frac{1}{\sigma^2} \int_{\mathbb{R}} a p'(\vs)\us_x (u-\ou) \, d x=\frac{\delta_S}{\sigma^2} \int_0^1(1+\sqrt{\delta_S}\, y) p'(\vs) f \, dy.
 \end{align*}
	Then, using \eqref{shock_speed_est} and \eqref{shock_speed_est-2} with $\dot{X}=-\frac{M}{\delta_S}(Y_1+Y_2)$, one has
    \begin{align*}
    \left| \dot{X} - 2 M \int_0^1 f \, dy \right| \le \frac{M}{\delta_S} \sum_{i=1}^2 \left| Y_i +\delta_S \int_0^1 f \, dy \right| \le C \sqrt{\delta_S} \int_0^1 |f| \, dy\le C\sqrt{\delta_S}\left(\int_0^1 |f|^2\,dy\right)^{1/2}.
    \end{align*}
   From the above inequality, we obtain
    \begin{equation}\label{est-dotX}
    -\frac{\delta_S}{2M} |\dot{X}|^2 \le -M \delta_S \left(\int_0^1 f \, dy \right)^2 +C \delta_S^2 \int_0^1 |f|^2 \, dy.
    \end{equation}
	
	\noindent $\bullet$ (Estimates of $B_1$ and  $\mathcal{G}_2$):
	Recalling the definition of $B_1$ and $\mathcal{G}_2$, we obtain
	\begin{align*}
	B_1-\mathcal{G}_2=\left(\frac{1}{4C_1} -\frac{\sigma}{2} \right) \int_\mathbb{R} a_x |u-\ou|^2 \, d x=\sqrt{\delta_S}\left(\frac{1}{4C_1}-\frac{\sigma}{2}\right)\int_0^1 |f|^2\,dy,
	\end{align*}
	where $C_1$ defined in \eqref{C_star} can be written as
	\begin{equation*}
	C_1=\frac{1}{2\sigma_m} -\sqrt{\delta_S}\frac{\gamma+1}{2\gamma p(v_m)}=\frac{1}{2 \sigma_m} -\sqrt{\delta_S}\alpha_m \sigma_m, \quad \alpha_m=\frac{\gamma+1}{2 \gamma \sigma_m p(v_m)}.
	\end{equation*}
	 On the other hand, using \eqref{shock_speed_est} and \eqref{shock_speed_est-2}, we obtain the following estimate:
	\begin{align*}
	\sqrt{\delta_S}\left(\frac{1}{4C_*}-\frac{\sigma}{2}\right)
	&\le  \delta_S\sigma_m^3 \alpha_m  +C \delta_S^{3/2},
	\end{align*}
	and consequently,
	\begin{equation}\label{R1.2}
	B_1-\mathcal{G}_2 \le  \left( \sigma_m^3 \alpha_m \delta_S +C \delta_S^{3/2} \right) \int_0^1 |f|^2 \, dy. 
	\end{equation}

	\noindent $\bullet$ (Estimate of $\mathcal{D}$):
	Using $a \geq 1$ and change of variables, we estimate $\mathcal{D}$ as
	\begin{equation*}
	\mathcal{D} \geq  \int_\R \frac{\mu}{v}|(u-\ou)_x|^2 dx=\int_0^1 |\partial_y f|^2 \frac{\mu}{v} \left(\frac{dy}{d x} \right) dy.
	\end{equation*}
	On the other hand, it comes from \cite[Appendix A]{HKKL_pre} that the following estimate holds:
	\begin{equation}\label{Diffusion}
	\left|   \frac{1}{y(1-y)} \frac{ \ms}{\vs} \left( \frac{dy}{dx} \right)-\frac{\sigma}{2 \sigma_m}\frac{\delta_S v''(p_m)}{ |v'(p_m)|^2 } \right| \le C \delta_S^2.
	\end{equation}
	Using \eqref{Diffusion} and the estimate
	\begin{equation*}
    \begin{aligned}
	\left| \frac{\mu}{v}\frac{\vs}{\ms}-1 \right| \le C |\vs-v| \le \left( |\vs-\ov|+|\ov-v|\right) \le C (\delta_0 + \varepsilon_1),
    \end{aligned}
	\end{equation*}
	we have
	\begin{align}
	\begin{aligned}\label{estimate in diffusion}
	&\left|\frac{\mu}{v}\left(\frac{dy}{dx}\right)-\frac{\sigma}{2\sigma_m}\frac{\delta_S v''(p_m)}{|v'(p_m)|^2}y(1-y)\right|\\
	&\le \left|\frac{\ms}{v^S}\left(\frac{dy}{dx}\right)\left(\frac{\mu}{v}\frac{v^S}{\ms}-1\right)\right|+ \left|\frac{\ms}{v^S}\left(\frac{dy}{dx}\right)-\frac{\sigma}{2\sigma_m}\frac{\delta_S v''(p_m)}{|v'(p_m)|^2}y(1-y)\right|\\
	&\le C\delta_S(\delta_0+\e_1)y(1-y)+C\delta_S^2y(1-y)\le C\delta_S(\delta_0+\e_1)y(1-y).
	\end{aligned}
	\end{align}
	Combining \eqref{estimate in diffusion}, \eqref{shock_speed_est}, and the equation
 \begin{align*}
	\frac{1}{2 }\frac{v''(p_m)}{|v'(p_m)|^2 }=\frac{1}{2} (1+\gamma) \frac{1}{v_m}=\sigma_m^3 \alpha_m,
	\end{align*}
 we obtain a lower bound for $\mathcal{D}$ as
	\begin{equation} \label{R1.3}
 \begin{aligned}
     \mathcal{D} 
 &\geq  \left(\frac{\sigma}{2 \sigma_m}\frac{\delta_S v''(p_m)}{ |v'(p_m)|^2 }-C\delta_S(\delta_0+\e_1) \right) \int_0^1 y(1-y) \left| \partial_y f \right|^2 dy\\
 &\ge \sigma_m^3 \alpha_m \delta_S (1-C (\delta_0 + \varepsilon_1)) \int_0^1 y(1-y) |\partial_y f|^2 \, dy.
 \end{aligned}
	\end{equation}
	
We now put the estimates \eqref{est-dotX}, \eqref{R1.2} and \eqref{R1.3} together, and using the smallness of $\delta_0$ and $\varepsilon_1$ to obtain
\begin{align*} 
\mathcal{R}_1&= -\frac{\delta_S}{2M}|\dot{X}|^2+B_1-\mathcal{G}_2-\frac{3}{4}\mathcal{D} \\
&\le -M \delta_S \left(\int_0^1 f \, dy \right)^2 + \sigma_m^3 \alpha_m \delta_S \left(   \frac{9}{8} \int_0^1 |f|^2 \, dy -\frac{5}{8}  \int_0^1 y(1-y) |\partial_y f|^2 \, dy \right).
\end{align*}
Then, we apply the Poincar\'e-type inequality \eqref{eq:Hardy-Legendre}, together with the identity
\[\int_0^1 |f-\overline{f}|^2 \, dy=\int_0^1 |f|^2 \, dy-\overline{f}^2, \quad \overline{f}:=\int_0^1 f \, dy,\]
to derive
\begin{align*}
	\mathcal{R}_1 \le -M \delta_S \left(\int_0^1 f \, dy \right)^2 -\frac{\sigma_m^3 \alpha_m \delta_S}{8} \int_0^1 |f|^2 \, dy + \frac{5 \sigma_m^3 \alpha_m \delta_S}{4} \left(\int_0^1 f\,dy\right)^2.
\end{align*}
Finally, we choose $M=\dfrac{5 \sigma_m^3 \alpha_m}{4}$ so that we close the estimate on $\mathcal{R}_1$ as 
\begin{equation}\label{R1}
	\mathcal{R}_1 \le  -\frac{\sigma_m^3 \alpha_m }{8} \int_0^1 |f|^2 \, dy = -\frac{\sigma_m^3\alpha_m}{8}\int_{\R}|u^S_x||u-\ou|^2\,dx=-C_2 G^S_u,
\end{equation}
where
\begin{equation*}
	\begin{aligned}
	G^S_u :=  \int_\mathbb{R} |\us_x| |u-\ou|^2 \, dx, \quad C_2:=\frac{\sigma_m^3 \alpha_m }{8}.
	\end{aligned}
\end{equation*}
\subsection{Estimate of $\mathcal{R}_2$} \label{Est-Remaining-terms} 
Next, we derive the estimate on the high-order terms. Before we estimate $\mathcal{R}_2$, we extract a new good term by using $\mathcal{G}_1$ and $G^S_u$.
\begin{lemma} There exists a positive constant $C_3$ such that
    \begin{align} \label{C_2}
        -\frac{1}{2}( \mathcal{G}_1 +C_2 G^S_u ) \le -C_3 G^S_v, 
    \end{align}
where 
\begin{align*} % \label{G^S_v}
    G^S_v:=\int_\mathbb{R}|\us_x||v-\ov|^2 \, dx .
\end{align*}
\end{lemma}
\begin{proof}
    Since $|v-\ov|\sim|p(v)-p(\ov)|$, and $|\vs_x| \sim \sqrt{\delta_S}|a_x|$, we observe that there exists a positive constant $C$ such that 
    \begin{align*}
        G^S_v &\le C \int_\mathbb{R} |\vs_x| \left|p(v)-p(\ov)\right|^2 \, dx  \\ 
        &\le C  \int_\mathbb{R} |\vs_x| \left|p(v)-p(\ov) - \frac{u-\ou}{2C_1}\right|^2 \, dx +C  \int_\mathbb{R} |\vs_x| \left|u-\ou\right|^2 \, dx \\ 
        &\le C \sqrt{\delta_S} \mathcal{G}_1 + C G^S_u.
   \end{align*}
Considering the smallness of $\sqrt{\delta_S}$, one can choose small enough $C_3$, but independent of $\delta_0$ and $\e_1$ such that \eqref{C_2} holds.
\end{proof}
    
We now substitute \eqref{R1} and \eqref{C_2} into \eqref{est} to obtain
\begin{equation}\label{est-R}
    \begin{aligned}{}
    \frac{d}{dt} \int_\mathbb{R} a \eta(U|\oU) \, dx &\le \frac{C}{\delta_S}\sum_{i=3}^7 |Y_i|^2 +\sum_{i=2}^7B_i 
     -\frac{\delta_S}{4M} |\dot{X}|^2 -\frac{C_2}{2}G^S_u -C_3 G^S_v -\frac{1}{2}\mathcal{G}_1-\mathcal{G}_3 -\mathcal{G}^R-\frac{1}{4}\mathcal{D}.
    \end{aligned}
\end{equation}

Therefore, to close the estimate of the weighted relative entropy, it suffices to control the terms $|Y_i|^2$ and  $B_i$ in $\mathcal{R}_2$. To this end, we need the following lemma. (See also \cite{HKKL_pre})

\begin{lemma} 
\label{lem: Useful bad terms estimates} There exists a positive constant $C$ such that for $p \ge 0$, $q \ge \frac{1}{3}$,
\begin{equation*} %\label{useful bad term estimates}
    \begin{aligned}
    &\int_\mathbb{R} (|\vs_x|+|\us_x|)^{1+p} \left( \left|u-\ou\right|^2 +\left| v-\ov \right|^2 + \left| p(v)-p(\ov) \right|^2 \right) \, dx \le C \delta_S^{2p}  \left(G^S_u +G^S_v\right) , \\ 
    &\int_\mathbb{R} (|\vr_x|+|\ur_x|)^{1+p} \left(  \left| v-\ov \right|^2 +\left| p(v)-p(\ov) \right|^2 \right) \, dx  \le C \delta_R^p    \mathcal{G}^R  , \\
    &\int_\mathbb{R} (|\ov_x|+|\ou_x|)^{1+p} \left(  \left| v-\ov \right|^2 +\left| p(v)-p(\ov) \right|^2 \right) \, dx   \le C \delta_S^{2p} G^S_v +C \delta_R^p  \mathcal{G}^R ,\\
    &\int_\mathbb{R} |a_x|^{1+q} \left(|u-\ou|^2 + \left|v-\ov\right|^2 \right)\, dx \le C \delta_S^{\frac{3q-1}{2}}  \left( G^S_u + G^S_v \right) .
    \end{aligned}
\end{equation*}
\end{lemma}
\begin{proof} By \eqref{shock-property} in Lemma \ref{lem:shock-property}, we recall that $|\vs_x| \sim |\us_x|$ and $\|\us_x\|_{L^\infty} \le C \delta_S^2$. Also,  we recall the equivalence between $|v-\ov|$ and $|p(v)-p(\ov)|$ in Lemma \ref{lem : Estimate-relative}. Then, the fist inequality can be directly obtained as
    \begin{align*}
        &\int_\mathbb{R} (|\vs_x|+|\us_x|)^{1+p} \left( \left|u-\ou\right|^2 +\left| v-\ov \right|^2 + \left| p(v)-p(\ov) \right|^2 \right) \, dx \\ 
        &\le \|\us_x\|_{L^\infty}^p \int_\mathbb{R} |\us_x| \left( |u-\ou|^2 + |v-\ov|^2 \right) \, dx\le C \delta_S^{2p}  \left(G^S_u +G^S_v\right).
    \end{align*}

The proof of the second inequality follows a similar approach as the first one. It relies on the equivalence between $|v-\ov|$, $|p(v)-p(\ov)|$, and $p(v|\ov)$. Additionally, since $\| \ur_x\|_{L^\infty} \le \delta_R$ from the Lemma \ref{lem:rarefaction_property}, we have
\begin{align*}
     &\int_\mathbb{R} (|\vr_x|+|\ur_x|)^{1+p} \left(  \left| v-\ov \right|^2 +\left| p(v)-p(\ov) \right|^2 \right) \, dx \\ 
     &\le C \|\ur_x\|_{L^\infty}^p \int_\mathbb{R}  \ur_x p(v|\ov) \, dx \le C \delta_R^p \mathcal{G}^R.
\end{align*}
The third inequality can be obtained by combining the first and second inequalities, using the $\ou_x=\ur_x+\us_x$. Finally, the fourth inequality follows immediately from the first inequality by using $|a_x| \sim \frac{|\us_x|}{\sqrt{\delta_S}}$ in \eqref{a}.
\end{proof}

We start with the estimates on $Y_i$.\\

\noindent $\bullet$ (Estimate of $\frac{C}{\delta_S} |Y_i|^2$ for $i=3,4,5,6$): We use $|p'(v)-p'(\vs)| \le |\vr-v_m| \le \delta_R  $ to estimate $Y_3$ as
\begin{align*}
    |Y_3| &\le C \int_\mathbb{R} \left| \vr-v_m \right| |\vs_x| |v-\ov| \, dx  \le \delta_R \int_\mathbb{R} |\vs_x||v-\ov| \, dx \\ 
    &\le C \delta_R \sqrt{\int_\mathbb{R} |\vs_x| \, dx}\sqrt{\int_\mathbb{R} |\vs_x| |v-\ov|^2 \, dx } \le C \delta_R \sqrt{\delta_S}\sqrt{G^S_v }.
\end{align*}
For $Y_4$, we decompose it into $Y_{41}$ and $Y_{42}$ as
\begin{align*}
     Y_4 
     &=-\int_\mathbb{R} a p'(\vs)\vs_x \left(v-\ov+\frac{2C_1}{\sigma} \left( p(v)-p(\ov) \right) \right) \,dx \\&\quad -\frac{2C_1}{\sigma}\int_\mathbb{R} a p'(\vs) \vs_x \left(p(v)-p(\ov) - \frac{u-\ou}{2C_1}\right) \, dx =:Y_{41}+Y_{42}.
\end{align*}
Then, by using the definition of $C_1$ in \eqref{C_star}, we rewrite $Y_{41}$ as
\begin{align*}
    Y_{41} &= - \int_\mathbb{R} ap'(\vs)\vs_x \left( v-\ov+\frac{p(v)-p(\ov)}{\sigma\sigma_m}-\frac{2}{\sigma}\alpha_m\sigma_m \sqrt{\delta_S}(p(v)-p(\ov)) \right) \, dx. 
\end{align*}
Using Taylor expansion of the function $p \mapsto v(p)=p^{-\frac{1}{\gamma}}$, together with \eqref{shock_speed_est} and \eqref{shock_speed_est-2}, we obtain
\[\left| v-\ov+\frac{p(v)-p(\ov)}{\sigma\sigma_m}\right|\le C(\varepsilon_1+\delta_0)|p(v)-p(\ov)|. \]
From this, we estimate $Y_{41}$ as 
\begin{align*}
    |Y_{41}| 
    %&\le C \left( \int_\mathbb{R} |\vs_x| \left| v-\ov+\frac{p(v)-p(\ov)}{\sigma\sigma_m} \right| \, dx + \sqrt{\delta_S}\int_\mathbb{R} |\vs_x| |p(v)-p(\ov)|  \, dx \right) \\
    &\le C (\varepsilon_1+\delta_0+\sqrt{\delta_S}) \int_\mathbb{R} |\vs_x| \left|p(v)-p(\ov)\right| \,dx \\ 
    &\le C (\varepsilon_1+\delta_0+\sqrt{\delta_S}) \sqrt{\int_\mathbb{R} |\vs_x| dx} \sqrt{\int_\mathbb{R} |\vs_x||p(v)-p(\ov)|^2 \, dx} \\
    &\le C(\varepsilon_1+\delta_0+\sqrt{\delta_S})\sqrt{\delta_S} \sqrt{G^S_v} .
\end{align*}
For $Y_{42}$, we use H\"older inequality and \eqref{a_x} to obtain
\begin{align*}
    |Y_{42}| 
    &\le C \sqrt{\int_\mathbb{R} |\vs_x| \, dx} \sqrt{\int_\mathbb{R} |\vs_x| \left| p(v)-p(\ov)-\frac{u-\ou}{2C_1} \right|^2 \, dx} \\ 
    &\le C \sqrt{\delta_S} \sqrt{ \sqrt{\delta_S} \int_\mathbb{R} |a_x| \left| p(v)-p(\ov)-\frac{u-\ou}{2C_1} \right|^2 \, dx } \le C \delta_S^{3/4} \sqrt{ \mathcal{G}_1 }.
\end{align*}
Next, using $Q(v|\ov) \le C |v-\ov|^2$, we have
\begin{align*}
    |Y_5| \le C\int_\mathbb{R} |a_x| \left( | u-\ou|^2 + |v-\ov|^2 + |w-\ow|^2 \right) dx.
\end{align*}
Then, we use $\|a_x\|_{L^\infty} \le C \delta^{3/2}_S$  and \eqref{smallness} to get the following estimate for $Y_5$:
\begin{align*}
|Y_5|^2 &\le C\e_1^2\delta_S^{3/2}\int_\R|a_x|(|u-\ou|^2+|v-\ov|^2+|w-\ow|^2)\,dx\\
&\le C \varepsilon_1^2 \delta_S\left( \int_\mathbb{R} |\vs_x| \left( |u-\ou|^2 +|v-\ov|^2 \right) dx + \sqrt{\delta_S} \int_\mathbb{R} |a_x| |w-\ow|^2 \, dx \right) \\
& \le C \varepsilon_1^2\delta_S \left( G^S_u+G^S_v + \sqrt{\delta_S} \mathcal{G}_3 \right).
\end{align*}
Finally, we use the relations $|\ws_x| \le C |\vs_x|\le C\sqrt{\delta_S} |a_x|$ and  H\"older inequality to obtain
\begin{align*}
    |Y_{6}| &\le C  \int_\mathbb{R} |\vs_x| |w-\ow| \, dx \le C  \sqrt{\int_\mathbb{R} |\vs_x| \, dx} \sqrt{ \sqrt{\delta_S} \int_\mathbb{R} |a_x| |w-\ow|^2 \, dx} \le C \sqrt{\delta_S} \sqrt{\mathcal{G}_3}. 
\end{align*}
Combining all the estimates of $Y_i$, we conclude that
\begin{equation}\label{est-Y}
    \begin{aligned}
    \frac{C}{\delta_S}\sum_{i=3}^6 |Y_i|^2 &\le C \sqrt{\delta_S} (\mathcal{G}_1 + \mathcal{G}_3) +C \left(\varepsilon_1 +\delta_0 \right) \left( G^S_u+ G^S_v \right).
    \end{aligned}
\end{equation}
	   
\noindent $\bullet$ (Estimate of  $B_i$ for $i=2,\ldots, 6$):  We split $B_2$ as 
\begin{align*}
    B_2 &\le C \int_\mathbb{R} |a_x| |\vr-v_m||p(v)-p(\ov)|^2\, dx +C \sqrt{\delta_S}(\sqrt{\delta_S}+\delta_0)\int_{\mathbb{R}} |a_x| \left|p(v)-p(\ov)\right|^2 \, dx \\ 
    &+ C\int_{\mathbb{R}} |a_x| \left|p(v)-p(\ov)\right|^3 \, dx =:B_{21}+B_{22}+B_{23}.
\end{align*}
Then, we use \eqref{a_x}, interpolation inequality, and Lemma \ref{lem: interaction estimates} to obtain
    \begin{align*}
    |B_{21}| &\le C \int_\mathbb{R} |a_x| |\vr-v_m||p(v)-p(\ov)|^2\, dx\le \frac{C}{\sqrt{\delta_S}}  \norm{p(v)-p(\ov)}_{L^4} \norm{|\vs_x||\vr-v_m|}_{L^2} \\ 
    &\le C \delta_R \delta_S e^{-C \delta_S t} \norm{(p(v)-p(\ov))_x}_{L^2}^{1/2} \norm{p(v)-p(\ov)}_{L^2}^{3/2}  \le C  \varepsilon_1^2 \delta_R \delta_S e^{-C \delta_S t}.
    \end{align*}
    Next, we estimate $B_{22}$ as
    \begin{align*}
    |B_{22}|&\le C \sqrt{\delta_S}(\sqrt{\delta_S}+\delta_0)\int_{\mathbb{R}} |a_x| \left|p(v)-p(\ov)\right|^2 \, dx \le C (\sqrt{\delta_S}+\delta_0)(\mathcal{G}_1+G^S_u).
    \end{align*}
  For $B_{23}$, we use an algebraic inequality $|p|^3\le 8(|p-q|^3+|q|^3)$ and the interpolation inequality to obtain
    \begin{align*}
    |B_{23}|&\le C\int_{\mathbb{R}} |a_x| \left|p(v)-p(\ov)-\frac{u-\ou}{2C_1}\right|^3dx+C\int_{\mathbb{R}} |a_x| \left|u-\ou\right|^3dx \\
    &\le C\varepsilon_1\mathcal{G}_1+C\lVert u-\ou \rVert_{L^\infty}^2 \int_{\mathbb{R}} |a_x||u-\ou| \,dx\\
    &\le C\varepsilon_1\mathcal{G}_1+\frac{C}{\sqrt{\delta_S}} \lVert u-\ou \rVert_{L^2}\lVert (u-\ou)_x \rVert_{L^2}\int_{\mathbb{R}} |\vs_x| |u-\ou|dx \\
    &\le C\varepsilon_1\mathcal{G}_1+\tau\mathcal{D}+\frac{C}{\delta_S}\lVert u-\ou \rVert_{L^2}^2 \int_{\mathbb{R}} |\vs_x| dx\int_{\mathbb{R}} |\vs_x| |u-\ou|^2 dx \\
    &\le C\varepsilon_1\mathcal{G}_1+\tau\mathcal{D}+C\varepsilon_1^2G^S_u,
    \end{align*}
	where $\tau$ is a small constant independent of $\delta_0$ and $\e_1$, which will be chosen later. Next, we estimate $B_3$ as
    \begin{align*}
         B_3
         &\le  C \int_\mathbb{R} |a_x| \Big( |u-\ou| |(u-\ou)_x|+ |w-\ow| |(u-\ou)_x|   \Big) \, dx + \frac{C}{\sqrt{\delta_S}} \int_\mathbb{R} |\vs_x| |u-\ou| |(w-\ow)_x| \, dx \\
         &\le \tau \mathcal{D} + C \int_\mathbb{R} |a_x|^2 \left( |u-\ou|^2+ |w-\ow|^2 \right) \, dx +  \frac{C}{\sqrt{\delta_S}} \left( \int_\mathbb{R} |\vs_x|^{3/2} |u-\ou|^2 \, dx + \int_\mathbb{R} |\vs_x|^{1/2} |(w-\ow)_x|^2 \, dx \right) \\ 
    &\le \tau\mathcal{D} + C \delta_S G^S_u + C \delta_S^{3/2} \mathcal{G}_3 +  C \delta_S^{1/2} G^S_u + C \delta_S^{1/2}  \|(w-\ow)_x\|_{L^2}^2.
    \end{align*}
    Similarly, we use Young's inequality, Lemma \ref{lem: Useful bad terms estimates}, $|\ow_x| \le C |\ov_x| $, and  $|\ov_x| \sim |\ou_x|$ to derive
    \begin{align*}
    B_4 &\le \int_\mathbb{R} |a_x| \Bigg[ |u-\ou| |v-\ov|     |\ou_x|  +|u-\ou| |v-\ov|  | \ow_x|  + |w-\ow| |v-\ov| |\ou_x|   \Bigg] \, dx \\ 
    &\le C \int_\mathbb{R} |a_x|^2 \left(|u-\ou|^2+|w-\ow|^2 \right)\,dx +C\int_\mathbb{R} |\ou_x|^2 |v-\ov|^2 \, dx \\ 
    &\le C (\delta_S G^S_u +  \delta_S^{3/2} \mathcal{G}_3 + \delta_S^2 G^S_v + \delta_R \mathcal{G}^R).
    \end{align*}
 	Next, using $|\ow_x| \le C |\ov_x| \le C |\ou_x|$, we have
    \begin{align*}
    B_5 &\le \int_\mathbb{R}  |(u-\ou)_x| |v-\ov|    |\ou_x|  +  |(u-\ou)_x| |v-\ov| | \ow_x|  + |(w-\ow)_x| |v-\ov| |\ou_x|  \, dx \\
    &\le C \int_\mathbb{R} |\ou_x|^{1/2} (|(u-\ou)_x|^2 + |(w-\ow)_x|^2 ) \, dx + C \int_\mathbb{R} |\ou_x|^{3/2} |v-\ov|^2 \, dx \\ 
    &\le \tau \mathcal{D} + C (\delta_S +\sqrt{\delta_R}) \|(w-\ow)_x\|_{L^2}^2+C \delta_S G^S_v+ C \sqrt{\delta_R} \mathcal{G}^R  \\
    &\le \tau \mathcal{D} + C \sqrt{\delta_0} \|(w-\ow)_x\|_{L^2} + C \sqrt{\delta_0} (G^S_v+\mathcal{G}^R).
    \end{align*}
	To estimate $B_6$, we split
	\[B_6 = -\int_{\R}a(u-\ou)Q_1\,dx -\int_{\R}a(w-\ow)Q_2\,dx=:B_{61}+B_{62}.\]
	We first handle $B_{61}$ as
    \begin{align*}
        B_{61} & \le C \int_\mathbb{R} \left( |Q_1^I| + |Q_1^R| \right) |u-\ou| \, dx \le C  \|Q_1^I\|_{L^2} \| u-\ou \|_{L^2} + C \|Q_1^R \|_{L^1} \| u-\ou \|_{L^\infty}  ,
    \end{align*}
    which, together with \eqref{est-interaction}, implies
    \begin{equation*} %\label{E1}
        \begin{aligned}
            B_{61} &\le C \delta_S \delta_R e^{-C \delta_S t} \|u-\ou\|_{L^2} + C \|Q_1^R\|_{L^1} \|(u-\ou)_x\|_{L^2}^{\frac{1}{2}} \|u-\ou\|_{L^2}^{\frac{1}{2}} \\
            &\le C \delta_S \delta_R e^{-C \delta_S t} \|u-\ou\|_{L^2} + \tau \mathcal{D} + C \|Q_1^R\|_{L^1}^{\frac{4}{3}}  \|u-\ou\|_{L^2}^{\frac{2}{3}}\\
            &\le C \e_1\delta_S \delta_R e^{-C \delta_S t} + \tau\mathcal{D} + C \e_1^{2/3}\|Q_1^R\|_{L^1}^{\frac{4}{3}}.
        \end{aligned}
    \end{equation*}
    Finally, we use \eqref{est-interaction} to estimate $B_{62}$ as
    \begin{align*}
     B_{62}
     &\le C \| w-\ow \|_{L^2} \|Q_2\|_{L^2}\le C \e_1^{2} \delta_R \delta_S^{3/2} e^{-C \delta_S t}.
    \end{align*}
    Gathering all the estimates for $B_i$ and using the smallness of $\delta_0$ and $\e_1$, we conclude that
    \begin{equation}\label{est-B,K}
    \begin{aligned}
    \sum_{i=2}^6 B_i  &\le C\tau\mathcal{D} + C (\sqrt{\delta_0} + \e_1 ) (\mathcal{G}_1 + \mathcal{G}_3 + \mathcal{G}^R+ G^S_u + G^S_v ) \\
    & \quad +C \e_1 \delta_R \delta_S e^{-C \delta_S t} + C \sqrt{\delta_0} \| (w-\ow)_x\|_{L^2}^2 + C \e_1^{2/3}\| Q_1^R \|_{L^1}^{\frac{4}{3}}.
    \end{aligned}
    \end{equation}

    \subsection{Proof of Lemma \ref{Main Lemma}}
    We combine \eqref{est-R}, \eqref{est-Y}, and \eqref{est-B,K}, and then choosing the constant $\tau$ small enough so that $C\tau <\frac{1}{8}$ to obtain
    \begin{align*}
    \frac{d}{dt} \int_\mathbb{R} a \eta (U | \oU ) \, dx &\le -\frac{\delta_S}{4M} |\dot{X}|^2 -\frac{1}{8} \left( \mathcal{G}_1 + \mathcal{G}_3 +\mathcal{G}^R + C_2 G^S_u + C_3 G^S_v + \mathcal{D} \right) \\
     & \quad +C \e_1 \delta_R \delta_S e^{-C \delta_S t} + C \sqrt{\delta_0} \| (w-\ow)_x\|_{L^2}^2 + C\e_1^{2/3} \| Q_1^R \|_{L^1}^{\frac{4}{3}}.
    \end{align*}
    After integrating the above inequality on $[0,t]$ for any $t \le T$, we have
    \begin{align*}
    & \int_\mathbb{R} a(t,x) \eta (U (t,x) | \oU (t,x) ) \, dx + \int_0^t \left( \delta_S |\dot{X}|^2 + \mathcal{G}_1+ \mathcal{G}_3 +\mathcal{G}^R + G^S_u+ G^S_v+ \mathcal{D} \right) \, ds \\ 
    & \quad \le C  \int_\mathbb{R} a(0,x) \eta ( U_0(x)| \oU(0,x)) \, dx + C \sqrt{\delta_0} \int_0^t \|(w-\ow)_x\|_{L^2}^2 + C \e_1 \delta_R    + C\e_1^{2/3} \int_0^t \|Q_1^R\|_{L^1}^{\frac43}\,ds.
    \end{align*}
    To close the estimate, we compute $\int_0^t\|Q_1^R\|_{L^1}^{\frac43}\,ds$ as follows. First, recall that
    \[Q_1^R=-\left(\mr\frac{\ur_x}{\vr} \right)_x-\left(\kr \left(-\frac{\vr_{xx}}{(\vr)^5} + \frac{5(\vr_x)^2}{2(\vr)^6} \right) + (\kr)' \frac{(\vr_x)^2}{2(\vr)^5} \right)_x,\]
    which implies 
    \begin{align*}
    \begin{aligned} %\label{Q1R-est}
    \|Q_1^R\|_{L^1} &\le C \|( |\ur_{xx}| + |\ur_{x}| |\vr_{x}| +|\vr_{xxx}|+|\vr_{xx}| |\vr_{x}| + |\vr_{x}|^3)\|_{L^1}\\
    &\le C \left( \|\ur_{xx}\|_{L^1} + \|(\vr_{x},\ur_{x})\|_{L^2}^2 +\|\vr_{xxx}\|_{L^1} + \|\vr_{x}\|_{L^3}^3 \right).
    \end{aligned}
    \end{align*}
    On the other hand, using Lemma \ref{lem:rarefaction_property}, we have
    \begin{align*}
    &\|\ur_{xx}\|_{L^1} \le C \begin{cases}
    \delta_R\quad &\mbox{if}\,1+t\le \delta_R^{-1}\\
    (1+t)^{-1}\quad&\mbox{if}\,1+t\ge \delta_R^{-1},
    \end{cases} \quad  \|(\vr_{x},\ur_{x})\|_{L^2} \le C \begin{cases}
    \delta_R\quad &\mbox{if}\,1+t\le \delta_R^{-1}\\
    \delta_R^{1/2}(1+t)^{-1}\quad&\mbox{if}\,1+t\ge \delta_R^{-1},
    \end{cases}, \\
    &\|\vr_{xxx}\|_{L^1} \le C \begin{cases}
    \delta_R\quad &\mbox{if}\,1+t\le \delta_R^{-1}\\
    (1+t)^{-1}\quad&\mbox{if}\,1+t\ge \delta_R^{-1},
    \end{cases} \quad  \|\vr_{x}\|_{L^3} \le C \begin{cases}
    \delta_R\quad &\mbox{if}\,1+t\le \delta_R^{-1}\\
    \delta_R^{1/3}(1+t)^{-2/3}\quad&\mbox{if}\,1+t\ge \delta_R^{-1}.
    \end{cases}
    \end{align*}
    Using these bounds, we get
    \[\int_0^\infty \|Q_1^R\|_{L^1}^{4/3}\,ds = \int_0^{\delta_R^{-1}-1}\|Q_1^R\|_{L^1}^{4/3}\,ds+\int_{\delta_R^{-1}-1}^\infty\|Q_1^R\|_{L^1}^{4/3}\,ds \le C (\delta_R^{1/3} + \delta_R^{5/3}+\delta_R^3) \le C \delta_R^{1/3}. \]
    Thus, we finally obtain
    \begin{align*}
    & \int_\mathbb{R} a(t,x) \eta (U (t,x) | \oU (t,x) ) \, dx + \int_0^t \left( \delta_S |\dot{X}|^2 + \mathcal{G}_1+ \mathcal{G}_3 +\mathcal{G}^R + G^S_u+ G^S_v+ \mathcal{D} \right) \, ds \\ 
    & \quad \le C  \int_\mathbb{R} a(0,x) \eta ( U_0(x)| \oU(0,x)) \, dx + C \sqrt{\delta_0} \int_0^t \|(w-\ow)_x\|_{L^2}^2  + C\e_1^{2/3}\delta_R^{1/3}.
    \end{align*}
    Since the weight function $a$ has lower and upper bounds, and the relative entropy is equivalent to the $L^2$-norm of the perturbation, one obtains 
    \begin{align*}
    & \| U-\oU \|_{L^2}^2 + \int_0^t \left( \delta_S |\dot{X}|^2 + \mathcal{G}_1+ \mathcal{G}_3 +\mathcal{G}^R + G^S_u+ G^S_v+ \mathcal{D} \right) \, ds \\ 
    & \quad \le C  \| U_0 -\oU (0, \cdot) \|_{L^2} +C \delta_R^{1/3} + C \sqrt{\delta_0} \int_0^t \|(w-\ow)_x\|_{L^2}^2 .
    \end{align*}
    Finally, since $D \sim \mathcal{D}$, $G_1 \sim \mathcal{G}_1$, $G_3 \sim \mathcal{G}_3$ and $G^R \sim \mathcal{G}^R$ we complete the proof of Lemma \ref{Main Lemma}.

\section{Estimate on the $H^1$-perturbation}\label{sec:high-order}
\setcounter{equation}{0}
In this section, we present high-order estimates on the perturbation, thereby completing the proof of Proposition \ref{apriori-estimate}. The goal of this section is to prove the following estimate:
\begin{equation}
\begin{aligned}\label{eq:H1-est}
	& \norm{(v-\ov)(t,\cdot)}_{L^2}^2 +\norm{(u-\ou)(t,\cdot)}_{H^1}^2+\norm{(w-\ow)(t,\cdot)}_{H^1}^2+\delta_S \int_0^t | \dot{X}(s)|^2 \, d s \\ 
	&\quad +\int_0^t \left( G_1+G_3+G^S_u+G^S_v+G^R \right) \, ds +  \int_0^t \left( D_{u_1} + D_{u_2} + G_{w} + D_{w_1} + D_{w_2} \right)\, ds  \\ 
	& \le C \norm{(v-\ov)(0,\cdot)}_{L^2}^2 +\norm{(u-\ou)(0,\cdot)}_{H^1}^2+\norm{(w-\ow)(0,\cdot)}_{H^1}^2+C\delta_R^{1/3},
\end{aligned}
\end{equation}
which is the same estimate as \eqref{a-priori-1} in Proposition \ref{apriori-estimate}.

Let us briefly explain how we obtain \eqref{eq:H1-est}. To obtain the $H^1$-estimate on $(u-\ou,w-\ow)$, we first need to derive $L^2$-estimates on $(u-\ou)_x$ and $(w-\ow)_x$, which is established in Lemma \ref{lem: est-0}. Next, recall that in Lemma \ref{Main Lemma}, we did not completely close the estimate due to the term $\|(w-\ow)_x\|_{L^2}$. Therefore, in Lemma \ref{lem: est-1}, we extract the good terms that can control $(w-\ow)_x$. Of course, in the above two lemmas, new terms appear that need to be bounded. These terms are completely controlled by using Lemma \ref{lem: est-2}. As the proofs of these lemmas are very long and technical, we postpone their proofs to appendices.

For simplicity, we introduce the following notations for the perturbation:
\begin{equation*} %\label{Notations: phi,psi,omega}
\phi(t,x):=v(t,x)-\ov(t,x),\ \psi(t,x):=u(t,x)-\ou(t,x),\ \omega (t,x):=w(t,x)-\ow(t,x),
\end{equation*}
and
\[\phi_0(x):=\phi(0,x),\quad \psi_0(x):=\psi(0,x),\quad \omega_0(x):=\omega(0,x).\]
Then $(\phi,\psi,\omega)$ satisfies 
\begin{equation} \label{Underlying wave with shift}
\begin{aligned}
&\phi_t-\psi_x=\dot{X} \vs_x,\\
&\psi_t + \left(p(v)-p(\ov) \right)_x = \left( \frac{\mu}{v} u_x-\frac{\m}{\ov} \ou_x \right)_x + \left( \frac{\sqrt{\kappa}}{v^{5/2}} w_x - \frac{\sqrt{\ok}}{\ov^{5/2}} \ow_x\right)_x + \dot{X} \us_x -Q_1,\\
&\omega _t= -\left( \frac{\sqrt{\kappa}}{v^{5/2}} u_x - \frac{\sqrt{\ok}}{\ov^{5/2}} \ou_x \right)_x + \dot{X} \ws_x -Q_2.
\end{aligned}
\end{equation}

We note that the a priori assumption \eqref{smallness} and the Sobolev embedding imply the bound on the $L^\infty$-norms for the perturbations
\begin{equation} \label{eq: L^infty bounds}
	\norm{\phi}_{L^\infty((0,T)\times\R)}+\norm{\phi_x}_{L^\infty((0,T)\times\R)}+\norm{\psi}_{L^\infty((0,T)\times\R)} \le C \varepsilon_1.
\end{equation}
Furthermore, we have
\begin{equation} \label{eq: L^infty bounds-2}
\norm{\ou_x}_{L^\infty((0,T)\times\R)}+\norm{\ov_x}_{L^\infty((0,T)\times\R)} \le C \delta_0.
\end{equation}

\begin{lemma} \label{lem: est-0}Under the hypotheses of Proposition \ref{apriori-estimate}, there exists a positive constant $C$ that is independent of $\delta_0$ and $\e_1$, such that for all $t \in [0,T],$
\begin{equation} \label{eq: est-0}
\begin{aligned}
	&\norm{ \left(\psi, \omega \right)_x(t,\cdot)}_{L^2}^2+  \int_0^t D_{u_2} \, d s \\
	&\le C \norm{  \left( \psi,\omega\right)_x(0,\cdot)}_{L^2}^2+C\delta_S \int_0^t |\dot{X}|^2 \, d s + C \int_0^t \| \phi_x\|_{L^2}^2 \, d s \\ 
	& \quad +C (\varepsilon_1+\sqrt{\delta_0})\int_0^t  \left( G^S_v+G^R  + D_{u_1}+D_{w_1} + D_{w_2} \right) \, d s +C \delta_R.
\end{aligned}
\end{equation}
\end{lemma}
\begin{proof}
	Since the proof is complicated, we refer to Appendix \ref{app:proof_high_order_1} for the proof.
\end{proof}
	
\begin{lemma}\label{lem: est-1}
Under the assumptions of Proposition \ref{apriori-estimate}, there exist positive constant $C$ that is independent of $\delta_0$ and $\e_1$, such that for $0 \le t \le T$,
\begin{equation} \label{eq: est-1}
\begin{aligned}
&\norm{\omega}_{L^2}^2+  \int_\mathbb{R} \psi\omega \, dx  +  \int_0^t \left(G_w + D_{w_1} \right) \, d s \\ 
&\le  C \norm{\omega(0,\cdot)}_{L^2}^2 + C \int_\mathbb{R} \psi(0,\cdot)\omega(0,\cdot) \, dx + \int_\mathbb{R} | \ow \phi \omega |\, dx + C \int_\mathbb{R} \left|  \ow(0,\cdot) \phi(0,\cdot)\omega(0,\cdot) \right|\, dx \\
& \quad + C \delta_S \int_0^t | \dot{X}|^2 ds + C \int_0^t D_{u_1} \, d s +C \sqrt{\delta_0} \int_0^t \left( \| \phi_x \|_{L^2}^2 + D_{u_2} + G^S_v +G^R \right) \, d s +C \delta_R.
\end{aligned}
\end{equation}
\end{lemma}

\begin{proof}
	We refer to Appendix \ref{app:proof_high_order_2} for the proof of Lemma \ref{lem: est-1}.
\end{proof}
	
\begin{lemma}\label{lem: est-2}
	Under the assumptions of Proposition \ref{apriori-estimate}, there exists a positive constant $C$ that is independent of $\delta_0$ and $\e_1$, such that for $0 \le t \le T$
	\begin{equation}\label{eq: est-2}
	\begin{aligned}
	&\int_{\mathbb{R}} \psi_x\omega_x \, dx+ \int_{0}^t \left( D_{w_1} + D_{w_2} \right) \, ds 
	\\
	&\le \int_{\mathbb{R}} \psi(0,\cdot)_x\omega(0,\cdot)_x \, dx+C \delta_S \int_0^t |\dot{X}(s)|^2 \, d s + C_2 \int_{0}^{t} D_{u_2} ds \\ 
	&\quad  +C(\varepsilon_1+\sqrt{\delta_0}) \int_0^t \left(\norm{\phi_x}_{L^2}^2+G_w+G^S_v+G^R+D_{u_1} \right) \, d s+C \delta_R.
	\end{aligned}
	\end{equation}	
\end{lemma}
\begin{proof}
	Again, since the proof is technical, we postpone the proof of Lemma \ref{lem: est-2} to Appendix \ref{app:proof_high_order_3}.
\end{proof}

We are now ready to prove the high-order estimate \eqref{eq:H1-est}. In the following, we carefully combine the estimates \eqref{eq: est-0}, \eqref{eq: est-1} and \eqref{eq: est-2} as follows.\\
	
\noindent $\bullet$ (Step 1): Combine \eqref{eq: est-2} and \eqref{eq: est-0} with an appropriate weight to derive the estimate \eqref{eq: est-3}.\\
	
\noindent $\bullet$ (Step 2): Combine \eqref{eq: est-3} and \eqref{eq: est-1}, again by considering an appropriate weights to derive \eqref{eq: est-4}.\\
	
\noindent $\bullet$ (Step 3): Finally, combine \eqref{eq: est-4} and the $L^2$-estimates in Lemma \ref{Main Lemma} to obtain \eqref{eq:H1-est}. \\
	
We present the details of each step.\\
	
\noindent $\bullet$ (Step 1): By rearranging the first term $\int_\mathbb{R} \psi_x \omega_x \, dx$ in \eqref{eq: est-2} and applying Young's inequality, we have
\begin{equation} \label{eq: est-2,1}
\begin{aligned}
	&\int_{0}^t \left( D_{w_1} + D_{w_2} \right) \, ds 
	\\
	&\le \frac{1}{2} \norm{\psi_x}_{L^2}^2+\frac12 \norm{\omega_x}_{L^2}^2  + \frac12 \norm{\psi_{0x}}_{L^2}^2+\frac12\norm{\omega_{0x}}_{L^2} +C \delta_S \int_0^t |\dot{X}|^2 \, d s + C_2 \int_{0}^{t} D_{u_2} \, ds \\ 
	&\quad  +  C(\varepsilon_1+\sqrt{\delta_0}) \int_0^t \left( \norm{\phi_x}_{L^2}^2+D_{u_1}+G_w +G^S_v+G^R\right) \, d s+C \delta_R.
\end{aligned}
\end{equation}
We multiply \eqref{eq: est-2,1} by $\frac{1}{2 \max(1,C_2)}$ and then add \eqref{eq: est-0} in Lemma \ref{lem: est-0} to derive
\begin{equation*} 
\begin{aligned}
	&\frac{1}{2} \norm{\psi_x}_{L^2}^2+\frac12 \norm{\omega_x}_{L^2}^2 +  \frac{1}{2} \int_0^t D_{u_2} \, d s + \frac{1}{2 \max(1,C_2)} \int_{0}^t  (D_{w_1}+D_{w_2}) \, ds  \\
	&\le C \norm{\psi_{0x}}_{L^2}^2+C \norm{\omega_{0x}}_{L^2}^2   +C\delta_S \int_0^t |\dot{X}|^2 \, d s + C \int_0^t \norm{\phi_x}_{L^2}^2 \, d s \\ 
	& \quad +C (\varepsilon_1+\sqrt{\delta_0}) \int_0^t  \left( \norm{\phi_x}_{L^2}^2+ D_{u_1} + G_w+D_{w_1}+D_{w_2} +G^S_v+G^R \right) \, d s +C \delta_R.\\  
\end{aligned}
\end{equation*} 
Using the smallness of $\delta_0$ and $\varepsilon_1$, there exists a positive constant $C_3$, which is independent of $\delta_0$ or $\e_1$ such that
\begin{equation} \label{eq: est-3}
\begin{aligned}
	&\norm{ \psi_x }_{L^2}^2+\norm{ \omega_x }_{L^2}^2+  \int_0^t D_{u_2} \, d s + \int_{0}^t  (D_{w_1}+D_{w_2}) \, ds  \\
	&\le C \norm{\psi_{0x}}_{L^2}^2+C\norm{\omega_{0x}}_{L^2}  +C\delta_S \int_0^t |\dot{X}|^2 \, d s + C_3 \int_0^t \norm{\phi_x}_{L^2}^2 \, d s \\
	&\quad +
 C(\varepsilon_1+\sqrt{\delta_0}) \int_0^t \left( D_{u_1}+G_w +G^S_v+G^R\right) \, d s+C \delta_R.\\  
\end{aligned}
\end{equation} 
	
\noindent $\bullet$ (Step 2): Similarly, we move the second term $\int_\mathbb{R} \psi \omega \, dx$ in the left-hand side of \eqref{eq: est-1} and apply Young's inequality, $|\ow| \le C |\ov_x|$, and Lemma \ref{lem:shock-property} to get
\begin{equation} \label{eq: est-1,1}
\begin{aligned}
	\norm{\omega}_{L^2}^2+  \int_0^t \left( G_w + D_{w_1} \right) \, d s &\le C\left(\norm{\phi_0}_{L^2}^2 + \norm{\psi_0}_{L^2}^2+\norm{\omega_0}_{L^2}^2\right)+ \frac{1}{2} \left( \norm{ \psi}_{L^2}^2 + \norm{\omega}_{L^2}^2 \right) \\
	&\quad + C \delta_0 \left( \norm{ \phi}_{L^2}^2 + \norm{\omega}_{L^2}^2 \right) + C \delta_S \int_0^t | \dot{X}|^2 \, ds+ C \int_0^t D_{u_1} \, d s \\
	&\quad +C \sqrt{\delta_0} \int_0^t \left( \norm{\phi_x}_{L^2}^2 +D_{u_2}+ G^S_v + G^R \right) \, d s + C \delta_R .
\end{aligned}
\end{equation}
On the other hand, it follows from the definition of $\omega$ that 
\begin{align*}
\omega = w-\ow=-\frac{\sqrt{\kappa}}{v^{5/2}} \phi_x -\ov_x\left(\frac{\sqrt{\kappa}}{v^{5/2}}-\frac{\sqrt{\ok}}{\ov^{5/2}}\right).
\end{align*}
Since
\[\|\phi_x\|_{L^2}\le C\|\omega\|_{L^2}+C\int_\R |\ov_x||v-\ov|\,dx,\]
there exists a positive constant $c_*$ independent of $\delta_0$, $\varepsilon_1$ such that
\begin{align*}
\|\phi_x\|_{L^2}^2&\le \frac{1}{c_*}\|\omega\|_{L^2}^2+C\int_\R |\ov_x|\,dx\int_{\R}|\ov_x||p(v)-p(\ov)|^2\,dx\le \frac{1}{c_*}\|\omega\|_{L^2}^2+C\delta_0(G^S_v +G^R),
\end{align*}
which yields
\begin{equation} \label{eq: omega 2-norm inequality}
\begin{aligned}
	c_* \norm{\phi_x}_{L^2}^2\le \norm{\omega}_{L^2}^2 + C\delta_0(G^S_v +G^R).
\end{aligned}
\end{equation}
We multiply the inequality \eqref{eq: est-3} by $\frac{c_*}{4 \max(1,C_3)}$, add with \eqref{eq: est-1,1}, and then use the smallness of $\delta_0$ to obtain
\begin{equation*}
\begin{aligned}
&\frac{1}{4}\norm{\omega}_{L^2}^2+\frac{c_*}{4 \max(1,C_3)} (\norm{\psi_x}_{L^2}^2+\|\omega_x\|_{L^2}^2) \\
&\quad +\frac{c_*}{8 \max(1,C_3)} \int_0^t \left( D_{u_2} + D_{w_1} + D_{w_2} \right) \, d s+ \frac{3}{4} \int_0^t \left( G_w + D_{w_1} \right) \, d s \\ 
&\le C (\norm{\phi_0}_{L^2}^2+\norm{\psi_0}_{H^1}^2+\norm{\omega_0}_{H^1}^2) + \frac{1}{2} \norm{\phi}_{L^2}^2+\frac12 \norm{\psi}_{L^2}^2 + C \delta_S \int_0^t | \dot{X}|^2 \, ds \\
&\quad + C \int_0^t D_{u_1} \, d s + \frac{c_*}{2} \int_0^t \norm{\phi_x}_{L^2}^2 \, d s  +C (\varepsilon_1 + \sqrt{\delta_0} )\int_0^t \left( G^S_v + G^R \right) \, d s +C \delta_R.
\end{aligned}
\end{equation*}
Finally, using the inequality \eqref{eq: omega 2-norm inequality}, one obtains
\begin{equation} \label{eq: est-4}
\begin{aligned}
	&\norm{\psi_x}_{L^2}^2+ \norm{\omega}_{L^2}^2+\norm{\omega_x}_{L^2}^2 + \int_0^t \left( D_{u_2} + G_w+D_{w_1}+D_{w_2}  \right) \, d s \\
	&\le C (\norm{\phi_0}_{L^2}^2+\norm{\psi_0}_{H^1}^2+\norm{\omega_0}_{H^1}^2)\\
	&\quad + C_4 \left( \norm{\phi}_{L^2}^2+\norm{\psi}_{L^2}^2 +  \delta_S \int_0^t | \dot{X}|^2 ds + (\varepsilon_1 + \sqrt{\delta_0} )\int_0^t
\left(  G^S_v+G^R \right) ds +  \int_0^t D_{u_1}\, d s \right)+C \delta_R,
\end{aligned}
\end{equation}
where $C_4$ is a constant independent of $\delta_0$ or $\e_1$.\\

\noindent $\bullet$ (Step 3): We multiply \eqref{eq: est-4} by $\frac{1}{2 \max (1,C_4)}$, add it to \eqref{energy-est}, and use the smallness of $\delta_S$ to derive
\begin{equation*}
\begin{aligned}
& \frac{1}{2} \left(\norm{\phi}_{L^2}^2+\norm{\psi}_{L^2}^2+\norm{\omega}_{L^2}^2\right) + \frac{1}{2 \max (1,C_4)}(\norm{\psi_x}_{L^2}^2+\norm{\omega_x}_{L^2}^2)+ \frac{\delta_S}{2} \int_0^t | \dot{X}|^2 \, d s  \\
&\quad  +\frac{1}{2} \int_0^t \left( G_1+G_3+G^S_u+G^S_v+G^R+D_{u_1} \right) \, ds + \frac{1}{2 \max(1,C_4)} \int_0^t \left(D_{u_2} + G_w+D_{w_1}+D_{w_2}  \right)\, ds \\ 
& \le C( \norm{\phi_0}_{L^2}^2+\norm{\psi_0}_{H^1}^2+\norm{\omega_0}_{H^1}^2 )+C\delta_R^{1/3},
\end{aligned}
\end{equation*}
which implies \eqref{eq:H1-est}. This completes the proof of Proposition \ref{apriori-estimate}.

\begin{appendix}

\section{Proof of Lemma \ref{lem:rarefaction_property}}\label{sec:app-rarefaction-proof}
\setcounter{equation}{0}
In this appendix, we provide proof of the estimates on the high-order derivatives of the smooth approximation of the 1-rarefaction wave. As in \cite{MN86}, the smooth approximation of 1-rarefaction wave $w(t,x)$ is given by
\[w(t,x)=w_0(x_0(t,x)),\quad x=x_0(t,x)+w_0(x_0(t,x))t.\]
Therefore, we obtain
\[w_{xxx}(t,x) = \frac{w_0'''(x_0)}{(1+w'_0(x_0)t)^4}-\frac{3t(w''_0(x_0))^2}{(1+w_0'(x_0)t)^5}.\]
On the other hand, we note that
\begin{align}
	\begin{aligned}\label{w_0}
		&w'_0(x_0) = \frac{\delta_R}{2}(1-\tanh^2 x_0),\\
		&w''_0(x_0) = -\delta_R \tanh x_0 (1-\tanh^2 x_0),\\
		&w'''_0(x_0) = -\delta_R (1-3\tanh^2 x_0)(1-\tanh^2 x_0)
	\end{aligned}
\end{align}
which implies $|w''_0(x_0)|\le 2|w'_0(x_0)|$ and $|w'''_0(x_0)|\le 8|w'_0(x_0)|$ and therefore,
\[|w_{xxx}(t,x)|\le \frac{8|w'_0(x_0)|}{(1+w'_0(x_0)t)^4}+\frac{tw_0'(x_0)}{1+w'_0(x_0)t}\frac{12|w'_0(x_0)|}{(1+w'_0(x_0)t)^4}\le \frac{20|w'_0(x_0)|}{(1+w'_0(x_0)t)^4}.\]
The above estimate gives
\begin{align*}
	\int_{\R}|w_{xxx}(t,x)|^p\,dx&\le C\int_{\R} |w'_0(x_0)|^p(1+w'_0(x_0)t)^{1-4p}\,dx_0,
\end{align*}
which yields 
\begin{equation}\label{w_xxx-1}
	\|w_{xxx}\|_{L^p}^p\le \|w_0'\|_{L^p}^p\le C\delta_R^p.
\end{equation}
Furthermore, we split the integral for $\|w_{xxx}\|_{L^p}^p$ as
\begin{align*}
	\int_\R |w_{xxx}(t,x)|^p\,dx&=C\int_{0}^\infty |w_0'(x_0)|^p(1+w'_0(x_0)t)^{1-4p}\,dx_0\\
	&\quad +C\int_{-\infty}^0 |w_0'(x_0)|^p(1+w'_0(x_0)t)^{1-4p}\,dx_0=I_1+I_2.
\end{align*}
We only consider the estimate of $I_1$, since $I_2$ can be treated similarly. Using change of variables $y=w'_0(x_0)t$, we observe that
\begin{align*}
	I_1=\int_0^\infty |w_0'(x_0)|^p(1+w'_0(x_0)t)^{1-4p}\,dx_0=\int_0^{\frac{\delta_Rt}{2}}\left(\frac{y}{t}\right)^p(1+y)^{1-4p}(t|w''_0(x_0)|)^{-1}\,dy.
\end{align*}
However, since
\[t|w''_0(x_0)|=\delta_R t (\tanh x_0)(1-\tanh ^2 x_0)=2w_0'(x_0) t(\tanh x_0)=2y\sqrt{1-\frac{2}{\delta_Rt}y},\]
we have
\begin{align*}I_1 &= \int_0^{\frac{\delta_R t}{2}}\left(\frac{y}{t}\right)^p(1+y)^{1-4p}\frac{1}{2y}\frac{1}{\sqrt{1-\frac{2}{\delta_R t}y}}\,dy\\
	&=\int_0^1 \left(\frac{\delta_R z}{2}\right)^p \left(1+\frac{\delta_R tz}{2}\right)^{1-4p}\frac{1}{2z}\frac{1}{\sqrt{1-z}}\,dz\le C\delta_R^p\int_0^1 z^{p-1}\left(1+\frac{\delta_R tz}{2}\right)^{1-p}\frac{1}{\sqrt{1-z}}\,dz\\
	&\le C\delta_R^p\int_0^1 \left(\frac{2}{\delta_R t}\right)^{p-1}\frac{1}{\sqrt{1-z}}\,dz\le C \delta_R \frac{1}{t^{p-1}}. 
\end{align*}
Therefore, we conclude that
\begin{equation}\label{w_xxx-2}
\|w_{xxx}\|_{L^p}^p\le \frac{C\delta_R}{(1+t)^{p-1}}.
\end{equation}
Finally, it follows from \eqref{w_0} that
\[|w'''_0(x_0)|=\delta_R|(1-\tanh^2 x_0)^2-2\tanh x_0(1-\tanh^2 x_0)|\le \frac{4}{\delta_R}|w'_0(x_0)|^2+2|w''_0(x_0)|.\]
This yields
\begin{align*}
	|w_{xxx}| &\le \frac{|w'''_0(x_0)|}{(1+w'_0(x_0)t)^4}+\frac{3t|w''_0(x_0)|^2}{(1+w'_0(x_0)t)^5}\\
	&\le \frac{4|w'_0(x_0)|^2}{\delta_R(1+w'_0(x_0)t)^4}+\frac{2|w''_0(x_0)|}{(1+w'_0(x_0)t)^4}+\frac{6t|w'_0(x_0)||w''_0(x_0)|}{(1+w'_0(x_0)t)^5}\\
	&\le \frac{4|w'_0(x_0)|^2}{\delta_R(1+w'_0(x_0)t)^4}+ \frac{8|w''_0(x_0)|}{(1+w'_0(x_0)t)^3}\le \frac{4|w'_0(x_0)|^2}{\delta_R(1+w'_0(x_0)t)^4}+8|w_{xx}|
\end{align*}
where we used $w_{xx}(t,x) = \frac{w''_0(x_0)}{(1+w'_0(x_0)t)^3}$ in the last inequality. Therefore, 
\begin{align*}
	\|w_{xxx}\|_{L^p}^p\le \frac{C}{\delta_R^p}\int_\R |w'_0(x_0)|^{2p}(1+w'_0(x_0)t)^{1-4p}dx_0 + C\|w_{xx}\|_{L^p}^p.
\end{align*}
Now, we estimate the first integral by using the same argument above as
\begin{align*}
	&\frac{C}{\delta_R^p}\int_\R |w'_0(x_0)|^{2p}(1+w'_0(x_0)t)^{1-4p}dx_0\\
	&\le \frac{C}{\delta_R^p} \int_0^1 \left(\frac{\delta_R z}{2}\right)^{2p}\left(1+\frac{\delta_R tz}{2}\right)^{1-4p} \frac{1}{2z}\frac{1}{\sqrt{1-z}}\,dz\\
	&\le C\delta_R^p\int_0^1 z^{2p-1}\left(1+\frac{\delta_R tz}{2}\right)^{1-2p}\frac{1}{\sqrt{1-z}}\,dz \le C\delta^p_R \left(\frac{2}{\delta_R t}\right)^{2p-1}\le \frac{C\delta_R^{1-p}}{(1+t)^{2p-1}}.
\end{align*}
Thus, we obtain the third estimate on $\|w_{xxx}\|_{L^p}$:
\begin{equation}\label{w_xxx-3}
\|w_{xxx}\|_{L^p}\le \|w_{xx}\|_{L^p}+\frac{C\delta_R^{1/p-1}}{(1+t)^{2-1/p}}\le \frac{1}{1+t}+\frac{\delta_R^{1/p-1}}{(1+t)^{2-1/p}}.
\end{equation}
Combining \eqref{w_xxx-1}, \eqref{w_xxx-2}, and \eqref{w_xxx-3}, we obtain the required estimate for the third-order derivative. The estimate on the fourth-order derivatives can be obtained in a similar manner.

\section{Proof of Lemma \ref{lem: est-0}}\label{app:proof_high_order_1}
\setcounter{equation}{0}
In this appendix, we provide a detailed proof for Lemma \ref{lem: est-0}. We multiply \eqref{Underlying wave with shift}$_2$  by $-\psi_{xx}$ to obtain
\begin{equation}\label{eq: psi_t psi_xx}
\begin{aligned}
-\psi_t \psi_{xx} - \left(p(v)-p(\ov) \right)_x \psi_{xx} &= -\left( \frac{\mu}{v} u_x- \frac{\m}{\ov} \ou_x \right)_x \psi_{xx} - \left( \frac{\sqrt{\kappa}}{v^{5/2}} w_x - \frac{\sqrt{\ok}}{\ov^{5/2}} \ow_x\right)_x \psi_{xx} \\
&\quad - \dot{X} \us_x \psi_{xx} + Q_1 \psi_{xx}.
\end{aligned}
\end{equation}
Similarly, we multiply \eqref{Underlying wave with shift}$_3$ by  $-\omega_{xx}$ to obtain

\begin{equation}\label{eq: omega_t omega_xx}
-\omega _t \omega_{xx}= \left( \frac{\sqrt{\kappa}}{v^{5/2}} u_x - \frac{\sqrt{\ok}}{\ov^{5/2}} \ou_x \right)_x \omega_{xx} - \dot{X} \ws_x \omega_{xx} +Q_2 \omega_{xx}.
\end{equation}
We combine \eqref{eq: psi_t psi_xx} and \eqref{eq: omega_t omega_xx} to get
\begin{align*}
&\frac{d}{dt}\int_{\mathbb{R}} \frac{|\psi_x|^2}{2} \,dx+\frac{d}{dt}\int_{\mathbb{R}} \frac{|\omega_x|^2}{2} \,dx + \mathcal{D}_{u_2} = \sum_{i=1}^5 K_5,
\end{align*}
where
\[\mathcal{D}_{u_2}:=\int_{\mathbb{R}}\frac{\mu}{v} |\psi_{xx}|^2\,dx,\]
and
\begin{align*}
K_1&=-\dot{X}\int_{\mathbb{R}} (\us_x\psi_{xx}+\ws_x \omega_{xx})\,dx,\\ 
K_2&=-\int_{\mathbb{R}} \left(\frac{\mu}{v}\right)_x\psi_x\psi_{xx} \, dx + \int_{\mathbb{R}}\psi_x \left( \frac{\sqrt{\kappa}}{v^{5/2}}\right)_x \omega_{xx} \, dx -\int_{\mathbb{R}}\omega_x \left( \frac{\sqrt{\kappa}}{v^{5/2}}\right)_x \psi_{xx} \, dx,\\     
K_3&=-\int_{\mathbb{R}} \overline{u}_{xx}\left(\frac{\mu}{v}-\frac{\m}{\ov}\right)\psi_{xx} \, dx + \int_{\mathbb{R}}\ou_{xx} \left( \frac{\sqrt{\kappa}}{v^{5/2}}-\frac{\sqrt{\ok}}{\ov^{5/2}} \right) \omega_{xx}dx + \int_{\mathbb{R}}\ow_{xx} \left( \frac{\sqrt{\kappa}}{v^{5/2}}-\frac{\sqrt{\ok}}{\ov^{5/2}} \right) \psi_{xx} \, dx, \\
K_4&=\int_{\mathbb{R}} (p(v)-p(\ov))_x \psi_{xx}\,dx-\int_{\mathbb{R}} \overline{u}_x\left(\frac{\mu}{v}-\frac{\m}{\ov}\right)_x\psi_{xx} \, dx \\ 
&\quad +\int_{\mathbb{R}} \ou_x \left( \frac{\sqrt{\kappa}}{v^{5/2}}-\frac{\sqrt{\ok}}{\ov^{5/2}} \right)_x \omega_{xx} \, dx+\int_{\mathbb{R}}\ow_x \left( \frac{\sqrt{\kappa}}{v^{5/2}}-\frac{\sqrt{\ok}}{\ov^{5/2}} \right)_x \psi_{xx} \, dx,\\
K_5&=\int_\mathbb{R} (Q_1 \psi_{xx} + Q_2 \omega_{xx}) \, dx.
\end{align*}

\noindent $\bullet$ (Estimate of $K_1$):  Using $|w^S_x|\le C(|v^S_{xx}|+|v^S_x|)$ and \eqref{shock-property}, we estimate $K_1$ as
\begin{align*}
|K_1|
&\le C |\dot{X}|\sqrt{\int_\mathbb{R} |\vs_x|^2 \, dx } \left( \sqrt{\int_{\mathbb{R}} |\psi_{xx}|^2 \,dx} + \sqrt{\int_\mathbb{R} |\omega_{xx}|^2 \, dx}\right) \\
&\le C |\dot{X}| \delta_S^{3/2} \left( \norm{\psi_{xx}}_{L^2} + \norm{\omega_{xx}}_{L^2} \right) \\ 
&\le \delta_S|\dot{X}|^2+C\delta_S^2 \left( \mathcal{D}_{u_2} + \lVert \omega_{xx} \rVert^2_{L^2} \right) \le  \delta_S|\dot{X}|^2+ \tau \mathcal{D}_{u_2} + C \delta_S^2 \| \omega_{xx} \|_{L^2}^2,
\end{align*}
where $\tau$ is chosen to be small enough later.\\

\noindent $\bullet$ (Estimate of $K_2$):  We use $ |v_x| \le C ( |\phi_x| + |\ov_x| )$ and \eqref{eq: L^infty bounds} to derive
\begin{align*}
K_2 &\le C \int_\mathbb{R} |v_x| ( |\psi_x| |\psi_{xx}| + |\psi_x| |\omega_{xx}| + |\omega_x| |\psi_{xx}|) \, dx \\
&\le C (\varepsilon_1 + \delta_0 ) \left( \|\psi_{x}\|_{H^1} + \|\omega_{x}\|_{H^1}\right) \le \tau \mathcal{D}_{u_2} + C (\varepsilon_1 + \delta_0 ) \left( D_{u_1} + D_{w_1} + D_{w_2} \right).
\end{align*}

\noindent $\bullet$ (Estimate of $K_3$): We use $|\ow_{xx}| \le C  |\ov_x|$, $|\ou_{xx}| \le |\ou_x|$, and $|\ou_x| \sim |\ov_x|$  to estimate $K_{3}$ as
\begin{align*} 
K_3 &\le C \int_\mathbb{R} |\ov_x| |\phi| ( |\psi_{xx}| + |\omega_{xx}| ) \, dx \le C \int_\mathbb{R} |\ov_x|^{3/2} |\phi|^2 \, dx + C \int_\mathbb{R} |\ov_x|^{1/2}  ( |\psi_{xx}|^2 + |\omega_{xx}|^2 ) \, dx \\
&\le \tau \mathcal{D}_{u_2} + C \sqrt{\delta_0} (G^S_v+ G^R) .
\end{align*}

\noindent $\bullet$ (Estimate of $K_4$): Similarly, we have
\begin{align*}
K_4 &\le C \int_\mathbb{R}( |\phi_x| + | \ov_x| |\phi|)|\psi_{xx}| \, dx + C \int_\mathbb{R} |\ov_x| ( |\phi_x| + | \ov_x| |\phi|) |\omega_{xx}| \, dx \\
&\le \tau\int_{\mathbb{R}} |\psi_{xx}|^2 dx + C \int_{\mathbb{R}} |\phi_x|^2 dx + C \int_\mathbb{R} |\ov_x|^2 (|\phi|^2+|\omega_{xx}|^2) \, dx \\ 
&\le \tau \mathcal{D}_{u_2} + C \delta_0 \left(G^S_v+G^R+ D_{w_2} \right).
\end{align*}

\noindent $\bullet$ (Estimate of $K_5$): Finally, we estimate $K_5$ as
\begin{align*}
K_5&=\int_\mathbb{R} (Q_1 \psi_{xx} + Q_2 \omega_{xx}) dx \le C(\|Q_1^I\|_{L^2}^2 + \|Q_1^R\|_{L^2}^2 ) + \tau \mathcal{D}_{u_2} + \frac{C}{\varepsilon_1} \|Q_2\|_{L^2}^2 + \varepsilon_1 \|\omega_{xx}\|_{L^2}^2.
\end{align*}
Therefore, combining all estimates for $K_i$ with small enough $\tau$ so that $5\tau<\frac{1}{2}$, we obtain
\begin{align}
\begin{aligned}\label{est:psi_x}
&\frac{d}{dt}\int_{\mathbb{R}} \frac{|\psi_x|^2}{2} \,dx+\frac{d}{dt}\int_{\mathbb{R}} \frac{|\omega_x|^2}{2} \,dx + \frac{1}{2} \mathcal{D}_{u_2} \\ 
&\le \delta_S |\dot{X}|^2 +C \| \phi_x\|_{L^2}^2+C (\sqrt{\delta_0} +\varepsilon_1)  \left( G^S_v+G^R+D_{u_1} + D_{w_1} + D_{w_2} \right) \\
&\quad + C\left(\|Q_1^I\|_{L^2}^2 + \|Q_1^R\|_{L^2}^2+ \frac{1}{\varepsilon_1}\|Q_2\|_{L^2}^2\right).
\end{aligned}
\end{align}
We note that
\[\|Q_1^R\|_{L^2}^2 \le C \left( \|\ur_{xx}\|_{L^2}^2 +\|(\vr_x,\ur_x)\|_{L^4}^4 +\|\vr_{xxx}\|_{L^2}^2+ \|\vr_x\|_{L^6}^6 \right).\]
Since 
\begin{align*}
&\|\ur_{xx}\|_{L^2} \le C \min \left\{  \delta_R, \frac{1}{1+t} \right\}, \quad  \|(\vr_{x},\ur_{x})\|_{L^4} \le C \min \left\{  \delta_R, \frac{\delta_R^{1/4}}{(1+t)^{3/4}} \right\}, \\
&\|\vr_{xxx}\|_{L^2} \le C \min \left\{  \delta_R, \frac{1}{1+t}+\frac{\delta_R^{-1/2}}{(1+t)^{3/2}} \right\}, \quad  \|\vr_{x}\|_{L^6} \le C \min \left\{  \delta_R, \frac{\delta_R^{1/6}}{(1+t)^{5/6}} \right\},
\end{align*}
we have
\begin{equation} \label{Q1R-L2est}
\int_0^t \|Q_1^R\|_{L^2}^2 ds \le C  \left( \delta_R + \delta_R^3 +\delta_R +\delta_R^5\right) \le C \delta_R.   
\end{equation}	
We integrate \eqref{est:psi_x} from $0$ to $t$ for any $t\in \left[0, T  \right]$ and use \eqref{est-interaction} and \eqref{Q1R-L2est} to get
\begin{align*}
&\frac{1}{2}\norm{ \big(\psi_x, \omega_x \big)}_{L^2}^2+ \frac{1}{2} \int_0^t \mathcal{D}_{u_2} \, d s \\
&\le \frac{1}{2} \norm{ \big(\psi_{0x}, \omega_{0x} \big)}_{L^2}^2+\delta_S \int_0^t |\dot{X}|^2 \, d s + C \int_0^t \| \phi_x\|_{L^2}^2 \, d s \\ 
& \quad +C (\sqrt{\delta_0} +\varepsilon_1) \int_0^t  \left( G^S_v+G^R + D_{u_1}+D_{w_1} + D_{w_2} \right) \, d s  +C \delta_R.
\end{align*}
Since $D_{u_2} \sim \mathcal{D}_{u_2}$, this completes the proof of Lemma \ref{lem: est-0}.

\section{Proof of Lemma \ref{lem: est-1}}\label{app:proof_high_order_2}
\setcounter{equation}{0}
In this appendix, we present the proof of Lemma \ref{lem: est-1}. We multiply $\eqref{Underlying wave with shift}_2$ by $\omega$ and use the definition of $w$ and $\ow$ to obtain
\begin{equation*} %\label{omega_t psi}
\omega \psi_t - \left( \frac{p'(v) v^{5/2}}{\sqrt{\kappa}}w-\frac{p'(\ov)\ov^{5/2}}{\sqrt{\ok}}\overline{w} \right) \omega - \left( \frac{\sqrt{\kappa}}{v^{5/2}} w_x - \frac{\sqrt{\ok}}{\ov^{5/2}} \ow_x\right)_x \omega =\left( \frac{\mu}{v} u_x-\frac{\m}{\ov} \ou_x \right)_x \omega + \dot{X}   \us_x \omega -Q_1 \omega.
\end{equation*}
Similarly, multiplying \eqref{Underlying wave with shift}$_3$ by $\psi$, we obtain 
\begin{align*}
\omega_t \psi = -\left( \frac{\sqrt{\kappa}}{v^{5/2}}u_x-\frac{\sqrt{\ok}}{\ov^{5/2}} \ou_x \right)_x \psi + \dot{X} \ws_x \psi -Q_2 \psi.
\end{align*}
Then, we add the above equations, integrate over $\mathbb{R}$ to get
\begin{equation} \label{computation 1}
\begin{aligned}
\frac{d}{dt}& \int_\mathbb{R} \omega \psi \, dx  -\int_\mathbb{R}  \frac{p'(v) v^{5/2}}{\sqrt{\kappa}} \omega^2 \, dx + \int_\mathbb{R} \frac{\sqrt{\kappa}}{v^{5/2}} \omega_x^2 \, dx \\ 
&=\dot{X} \int_\mathbb{R} (\ws_x \psi + \us_x \omega) \, dx+\int_\mathbb{R} \left( \frac{\mu}{v} \psi_x\right)_x \omega \, dx \\
&\quad+\int_\mathbb{R} \ow \left(  \frac{p'(v) v^{5/2}}{\sqrt{\kappa}}-\frac{p'(\ov)\ov^{5/2}}{\sqrt{\ok}} \right) \omega \, dx -\int_\mathbb{R} \ow_x \left( \frac{\sqrt{\kappa}}{v^{5/2}}-\frac{\sqrt{\ok}}{\ov^{5/2}} \right) \omega_x \, dx \\ 
&\quad+\int_\mathbb{R} \frac{\sqrt{\kappa}}{v^{5/2}} \psi_x^2 \, dx + \int_\mathbb{R}  \ou_x \left(\frac{\sqrt{\kappa}}{v^{5/2}}-\frac{\sqrt{\ok}}{\ov^{5/2}} \right) \psi_x \, dx  - \int_\mathbb{R} \ou_x \left( \frac{\mu}{v}-\frac{\m}{\ov} \right) \omega_x \, dx -\int_\mathbb{R} (Q_1 \omega +Q_2 \psi) \, dx.
\end{aligned}
\end{equation}
We use $\eqref{Underlying wave with shift}_1$ and the following identity
\begin{equation*}
\phi_x =-\left(\frac{v^{5/2}}{\sqrt{\kappa}} w - \frac{\ov^{5/2}}{\sqrt{\ok}} \ow \right)=-\left(\frac{v^{5/2}}{\sqrt{\kappa}} \omega - \ow \left(\frac{v^{5/2}}{\sqrt{\kappa}}-\frac{\ov^{5/2}}{\sqrt{\ok}}  \right) \right)
\end{equation*}
to write the second term of the right-hand side of  \eqref{computation 1} as 
\begin{equation} \label{computation 2}
\begin{aligned}
&\int_\mathbb{R} \left( \frac{\mu}{v} \psi_x\right)_x \omega \, dx =\int_\mathbb{R} \frac{\mu}{v} \psi_{xx} \omega \, dx+\int_\mathbb{R} \psi_x \left( \frac{\mu}{v} \right)_x \omega \, dx \\
&=\int_\mathbb{R} \frac{\mu}{v} \phi_{tx} \omega \, dx - \dot{X} \int_\mathbb{R} \frac{\mu}{v} \vs_{xx} \omega \, dx +\int_\mathbb{R} \psi_x \left( \frac{\mu}{v} \right)_x \omega \, dx \\ 
&=-\int_\mathbb{R} \frac{\mu}{v} \left(\frac{v^{5/2}}{\sqrt{\kappa}} \omega -\ow \left(\frac{v^{5/2}}{\sqrt{\kappa}}-\frac{\ov^{5/2}}{\sqrt{\ok}} \right) \right)_t \omega \, dx - \dot{X} \int_\mathbb{R} \frac{\mu}{v} \vs_{xx} \omega \, dx +\int_\mathbb{R} \psi_x \left( \frac{\mu}{v} \right)_x \omega \, dx.
\end{aligned}
\end{equation}
However, the first term of the right-hand side of \eqref{computation 2} becomes
\begin{equation} \label{computation 3}
\begin{aligned}
&-\int_\mathbb{R} \frac{\mu}{v} \left(\frac{v^{5/2}}{\sqrt{\kappa}} \omega -\ow \left(\frac{v^{5/2}}{\sqrt{\kappa}}-\frac{\ov^{5/2}}{\sqrt{\ok}} \right) \right)_t \omega \, dx \\
&=- \int_\mathbb{R} \frac{\mu }{\sqrt{\kappa}}v^{3/2} \left(\frac{\omega^2}{2}\right)_t \, dx - \int_\mathbb{R}   \frac{\mu}{v} \partial_v \left( \frac{v^{5/2}}{\sqrt{\kappa}} \right) v_t \omega^2 \, dx  -\int_\mathbb{R}\frac{\mu}{v} \left( \left(\frac{v^{5/2}}{\sqrt{\kappa}}-\frac{\ov^{5/2}}{\sqrt{\ok}}\right) \ow \right)_t \omega \, dx   \\
&=-\frac{1}{2} \frac{d}{dt} \int_\mathbb{R} \frac{\mu }{\sqrt{\kappa}}v^{3/2} \omega^2 \, dx - \int_\mathbb{R}    \left( \frac{\mu}{v} \partial_v \left( \frac{v^{5/2}}{\sqrt{\kappa}} \right)-\frac{1}{2}  \partial_v \left( \frac{\mu }{\sqrt{\kappa}}v^{3/2} \right) \right) v_t \omega^2 \, dx \\
&\quad -\frac{d}{dt}\int_\mathbb{R} \frac{\mu}{v}  \left(\frac{v^{5/2}}{\sqrt{\kappa}}-\frac{\ov^{5/2}}{\sqrt{\ok}}\right) \ow  \omega \, dx+\int_\mathbb{R} \left(\frac{v^{5/2}}{\sqrt{\kappa}}-\frac{\ov^{5/2}}{\sqrt{\ok}}\right) \ow  \left(  \frac{\mu}{v} \omega \right)_t \, dx .
\end{aligned}
\end{equation}
Therefore, we combine \eqref{computation 1}, \eqref{computation 2}, and \eqref{computation 3} to obtain
\begin{equation*} %\label{eq: High order-main}
\begin{aligned}
&\frac{1}{2} \frac{d}{dt} \int_\mathbb{R} \frac{\mu }{\sqrt{\kappa}}v^{3/2} \omega^2 \, dx +\frac{d}{dt} \int_\mathbb{R} \omega \psi \, dx +\frac{d}{dt}\int_\mathbb{R} \frac{\mu}{v}  \left(\frac{v^{5/2}}{\sqrt{\kappa}}-\frac{\ov^{5/2}}{\sqrt{\ok}}\right) \ow  \omega \, dx +\mathcal{G}_w +\mathcal{D}_{w_1} = \sum_{i=1}^{7} S_i, 
\end{aligned}
\end{equation*}
where
\begin{align*}
\mathcal{G}_w :=-\int_\mathbb{R}  \frac{p'(v) v^{5/2}}{\sqrt{\kappa}} \omega^2 \, dx,\quad\mathcal{D}_{w_1} :=  \int_\mathbb{R} \frac{\sqrt{\kappa}}{v^{5/2}} \omega_x^2 \, dx
\end{align*}
and
\begin{align*}
S_1 &:=\dot{X}(t)\int_\mathbb{R}  \left( \ws_x \psi + \us_x \omega- \frac{\mu}{v}\vs_{xx}  \omega \right) \, dx,\\
S_2 &:= - \int_\mathbb{R}   \left( \frac{\mu}{v} \partial_v \left( \frac{v^{5/2}}{\sqrt{\kappa}} \right)-\frac{1}{2}  \partial_v \left( \frac{\mu }{\sqrt{\kappa}}v^{3/2} \right) \right) v_t \omega^2   \, dx,\quad S_3 := \int_\mathbb{R} \left(\frac{v^{5/2}}{\sqrt{\kappa}}-\frac{\ov^{5/2}}{\sqrt{\ok}}\right) \ow  \left(  \frac{\mu}{v} \omega \right)_t \, dx, \\ 
S_4 &:=  \int_\mathbb{R}  \ow \left(  \frac{p'(v) v^{5/2}}{\sqrt{\kappa}}-\frac{p'(\ov)\ov^{5/2}}{\sqrt{\ok}} \right) \omega  \, dx - \int_\mathbb{R}  \ow_x \left( \frac{\sqrt{\kappa}}{v^{5/2}}-\frac{\sqrt{\ok}}{\ov^{5/2}} \right) \omega_x  \, dx \\
&\quad - \int_\mathbb{R}  \ou_x \left( \frac{\mu}{v}-\frac{\m}{\ov} \right) \omega_x  \, dx + \int_\mathbb{R}  \ou_x \left(\frac{\sqrt{\kappa}}{v^{5/2}}-\frac{\sqrt{\ok}}{\ov^{5/2}} \right) \psi_x \, dx,\\ 
S_{5} &:=  \int_\mathbb{R} \frac{\sqrt{\kappa}}{v^{5/2}} \psi_x^2 \, dx, \quad S_{6} :=  \int_\mathbb{R} \psi_x \left( \frac{\mu}{v} \right)_x \omega  \, dx  \quad 
S_{7} := -\int_\mathbb{R} (Q_1 \omega + Q_2 \psi) \, dx. 
\end{align*}

\noindent $\bullet$ (Estimate of $S_1$): We use H\"older's inequality, $|\ws_x| \le C \left( |\vs_x|^2 +|\vs_{xx}| \right) \le C \delta_S |\ov_x| $, and Young's inequality to estimate $S_1$ as
\begin{align*}
S_1 &\le C  |\dot{X}| \left( \sqrt{ \int_\mathbb{R} |\ws_x| \, dx}\sqrt{\int_\mathbb{R} |\ws_x| |\psi|^2 \, dx }+\sqrt{ \int_\mathbb{R} |\us_x|^2 \, dx}\sqrt{\int_\mathbb{R} |\omega|^2 \, dx}+\sqrt{ \int_\mathbb{R} |\vs_{xx}|^2 \, dx}\sqrt{\int_\mathbb{R} |\omega|^2 \, dx} \right) \\
&\le C \delta_S |\dot{X}| \left(\sqrt{\int_\mathbb{R} |\vs_x| |\psi|^2 \, dx }+\sqrt{\int_\mathbb{R} |\omega|^2 \, dx}\right) \le \frac{\delta_S}{2} |\dot{X}|^2 + C \delta_S \left( G^S+\mathcal{G}_w \right) \\ 
&\le \frac{\delta_S}{2}|\dot{X}|^2 +\tau \mathcal{G}_w + C \delta_S G^S.
\end{align*}
\noindent$\bullet$ (Estimate of $S_2$):
We estimate $S_2$ by using $v_t=u_x=\psi_x + \ou_x$ as
\begin{align*}
S_2 &\le C \int_\mathbb{R} |u_x| |\omega|^2 \, dx \le C \int_\mathbb{R} |\psi_x| |\omega|^2 \, dx + C \int_\mathbb{R} |\ou_x| |\omega|^2 \, dx .
\end{align*}
Then, we use $|\omega| \le C  \left( | \phi_x| + |\ov_x| |\phi | \right) $ and \eqref{eq: L^infty bounds} to obtain
\begin{align*}
S_2 &\le C  \int_\mathbb{R} |\psi_x| \left( | \phi_x| + |\ov_x| |\phi| \right)  |\omega| \, dx + C \delta_0 \int_\mathbb{R} |\omega|^2 \, dx\\
&\le \int_\mathbb{R} |\psi_x| |\phi_x| |\omega| \, dx + C \int_\mathbb{R} |\psi_x| |\ov_x| |\phi| |\omega| \, dx + C \delta_0 \mathcal{G}_w\\
&\le C \varepsilon_1 \left( \int_\mathbb{R} |\psi_x| |\omega| \, dx + \int_\mathbb{R} |\ov_x| |\psi_x| |\omega| \, dx \right) + C \delta_0 \mathcal{G}_w\\
&\le C \varepsilon_1 \left( \int_\mathbb{R} |\psi_x|^2 \, dx + \int_\mathbb{R}  |\omega|^2 \, dx   \right) + C \delta_0 \mathcal{G}_w  \le \tau \mathcal{G}_w +  C \varepsilon_1 D_{u_1}.
\end{align*}
\noindent $\bullet$ (Estimate of $S_3$): We split $S_{3}$ into $S_{31}$ and $S_{32}$ as 
\begin{align*}
S_{3} &=   \int_\mathbb{R} \left(\frac{v^{5/2}}{\sqrt{\kappa}}-\frac{\ov^{5/2}}{\sqrt{\ok}}\right) \ow \left( \frac{\mu}{v} \right)_t \omega  \,dx +\int_\mathbb{R} \left(\frac{v^{5/2}}{\sqrt{\kappa}}-\frac{\ov^{5/2}}{\sqrt{\ok}}\right) \ow  \frac{\mu}{v}  \omega_t \,dx=: S_{31} + S_{32}.
\end{align*}
For $S_{31}$, we use $|\ow| \le C |\ov_x|$, $|v_t|=|u_x| \le |\psi_x|+|\ou_x|$, and \eqref{eq: L^infty bounds} to have
\begin{align*}
S_{31} &\le C \int_\mathbb{R} |\ov_x| |\phi| |\psi_x| |\omega| \, dx+ C \int_\mathbb{R} |\ov_x| |\phi|  |\ou_x| |\omega| \, dx \\ 
&\le C \varepsilon_1 \left( \int_\mathbb{R} |\ov_x||\psi_x|^2 \, dx + \int_\mathbb{R} |\ov_x| |\omega|^2 \, dx \right)+C \left( \int_\mathbb{R} |\ov_x|^2|\phi|^2 \, dx + \int_\mathbb{R} |\ou_x|^2 |\omega|^2 \, dx \right) \\ 
&\le C \varepsilon_1 \delta_0 (D_{u_1}+\mathcal{G}_w) + C \delta_0 ( G^S_v+G^R) \le \tau \mathcal{G}_w + C \delta_0 (D_{u_1} + G^S_v+G^R).
\end{align*}
To estimate $S_{32}$, we first bound $\omega_t$ by using $\eqref{Underlying wave with shift}_2$, $v_x = \phi_x + \ov_x$, Lemma \ref{lem:rarefaction_property} and  Lemma \ref{lem:shock-property} as
\begin{equation}\label{eq: omega_t}
\begin{aligned}
\|\omega_t\|_{L^2}^2 
&\le C \Big( \norm{\phi_x}_{L^\infty}^2 \norm{\psi_x}_{L^2}^2+ \norm{\ov_x}_{L^\infty}^2 \big(\norm{\phi_x}_{L^2}^2 + \norm{\psi_x}_{L^2}^2 \big)  +\|\psi_{xx}\|_{L^2}^2 \\
& \qquad +\norm{\ov_x \phi}_{L^2}^2 + |\dot{X}|^2 \norm{\ws_x}_{L^2}^2 +\|Q_2\|_{L^2}^2 \Big) \le C \varepsilon_1.
\end{aligned}
\end{equation}
Next, $S_{32}$ is estimated by using $|\ow| \le |\ov_x|$ and Young's inequality as 
\begin{align*}
S_{32} &\le C \int_\mathbb{R} |\phi| |\ow| |\omega_t| \, dx \le C \int_\mathbb{R} |\phi| |\ov_x| |\omega_t| \, dx \\
&\le C \int_\mathbb{R} |\ov_x|^{3/2} |\phi|^2 \, dx +  C \int_\mathbb{R} |\ov_x|^{1/2} |\omega_t|^2 \, dx=: S_{321}+S_{322}. 
\end{align*}
Due to Lemma \ref{lem: Useful bad terms estimates}, we obtain
\begin{equation} \label{eq: cS_21}
\begin{aligned}
S_{321} \le C \delta_0 \left( G^S_v+G^R \right).
\end{aligned}
\end{equation}
On the other hand, we estimate $S_{322}$ by using \eqref{eq: omega_t} and Lemma \ref{lem: Useful bad terms estimates} as 
\begin{align*}
S_{322} \le C \norm{\ov_x}_{L^\infty}^{1/2} &\Big( \norm{\phi_x}_{L^\infty}^2 \norm{\psi_x}_{L^2}^2+ \norm{\ov_x}_{L^\infty}^2 \big(\norm{\phi_x}_{L^2}^2 + \norm{\psi_x}_{L^2}^2 \big)  +\|\psi_{xx}\|_{L^2}^2\\ 
& \quad +\delta_0 ( G^S_v+G^R) + |\dot{X}|^2 \norm{\ws_x}_{L^2}^2 +\|Q_2\|_{L^2}^2 \Big).
\end{align*}
In addition, we use \eqref{eq: L^infty bounds} and $|\ws_x| \le C (|\vs_x|^2 +|\vs_{xx}|)\le C |\vs_x|$ in Lemma \ref{lem:shock-property} to obtain 
\begin{equation} \label{eq: cS_22 }
\begin{aligned}
S_{322}
& \le C \sqrt{\delta_0} \left( ( \varepsilon_1^2 +\delta_0^2) D_{u_1}+ \delta_0^2  \norm{\phi_x}_{L^2}^2 +D_{u_2} + \delta_0 (G^S_v+G^R)+  | \dot{X}|^2 \norm{\vs_x}_{L^2}^2+\|Q_2\|_{L^2}^2 \right)\\ 
&\le  C \delta_S^2 |\dot{X}|^2 + C \sqrt{\delta_0} \left(\norm{\phi_x}_{L^2}^2+ D_{u_1} +D_{u_2} +G^S_v+G^R +\|Q_2\|_{L^2}^2 \right).
\end{aligned}
\end{equation}
Combining \eqref{eq: cS_21} and \eqref{eq: cS_22 }, we have
\begin{align*}
S_{32} \le \frac{\delta_S}{2} |\dot{X}|^2 + C \sqrt{\delta_0} \left( \|\phi_x\|_{L^2}^2 + D_{u_1}+D_{u_2} +G^S_v+G^R  +\|Q_2\|_{L^2}^2 \right).
\end{align*}
Therefore, we estimate $S_{3}$ as
\begin{align*}
S_{3} \le \frac{\delta_S}{2} |\dot{X}|^2 + \tau \mathcal{G}_w + C \sqrt{\delta_0} \left( \|\phi_x\|_{L^2}^2 + D_{u_1}+D_{u_2} +G^S_v+G^R   +\|Q_2\|_{L^2}^2\right).
\end{align*}
\noindent $\bullet$ (Estimate of $S_4$): Using $|\ow| \le C |\ov_x| $, and Lemma \ref{lem: Useful bad terms estimates}, we estimate $S_4$ as
\begin{align*}
S_4 &\le C \int_\mathbb{R} |\ov_x| |\phi| \left( |\omega|+ |\omega_x| + |\psi_x| \right) \, dx \\ 
&\le  C \int_\mathbb{R} |\ov_x|^{3/2} |\phi|^2 \, dx + C \int_\mathbb{R} |\ov_x|^{1/2} \left( |\omega|^2+ |\omega_x|^2 + |\psi_x|^2 \right)\, dx \\ 
&\le C \delta_0^{1/2} (G^S_v + G^R) +C \delta_0^{1/2} (G_w + D_{w_1} + D_{u_1}).
\end{align*}

\noindent $\bullet$ (Estimate of $S_{5}$): We simply estimate $S_{5}$ as
\begin{align*}
S_{5} \le C \int_\mathbb{R} |\psi_x|^2 \, dx.
\end{align*}

\noindent $\bullet$ (Estimate of $S_6$): We estimate $S_{6}$ by using $v_x= \phi_x + \ov_x$ and \eqref{eq: L^infty bounds} as
\begin{align*}
S_{6} & \le C \int_{\mathbb{R}} |v_x| |\psi_x|  |\omega| \, dx 
\le \int_{\mathbb{R}} | \phi_x|  |\psi_x| | \omega | \, dx + \int_{\mathbb{R}} |\ov_x| |\psi_x |  |\omega | \, dx \\ 
&\le C\left( \norm{\phi_x}_{L^\infty} \norm{\psi_x}_{L^2 } \norm{\omega}_{L^2}+\norm{\ov_x}_{L^\infty}  \norm{\psi_x}_{L^2} \norm{\omega}_{L^2} \right)\le C (\varepsilon_1+\delta_0 ) \left(  D_{u_1}  + \mathcal{G}_w \right) \\ 
&\le \tau \mathcal{G}_w + C (\varepsilon_1+\delta_0) D_{u_1}.
\end{align*}
\noindent$\bullet$ (Estimate of $S_{7}$)
Finally, we estimate $S_{7}$ as
\begin{align*}
S_{7} \le C \int_\mathbb{R} (|Q_1| |\omega| + |Q_2| |\psi|) dx\le  C \left(\|Q_1^I\|_{L^2}^2 + \|Q_1^R\|_{L^2}^2 + \varepsilon_1 \|Q_2\|_{L^1} \right) + \tau \mathcal{G}_w .
\end{align*}
We gather all the estimates for $S_i$ and use the smallness of $\tau$, $\delta_0$ and $\varepsilon_1$ to derive
\begin{align*}
&\frac{1}{2}\frac{d}{dt} \int_\mathbb{R} \frac{\mu }{\sqrt{\kappa}}v^{3/2} \omega^2 \, dx + \frac{d}{dt} \int_\mathbb{R} \omega \psi \, dx + \frac{d}{dt} \int_\mathbb{R} \frac{\mu}{v} (v^{5/2}-\ov^{5/2}) \ow \omega \, dx + \frac{1}{2} \left( \mathcal{G}_w + \mathcal{D}_{w_1} \right) \\ 
&\le \delta_S | \dot{X}|^2 + C \|\psi_x \|_{L^2}^2 
+ C  \sqrt{\delta_0} \left( \| \phi_x \|_{L^2}^2 +D_{u_2} +G^S_v + G^R + \|Q_2\|_{L^2}^2  \right) \\
&\quad +C \left(\|Q_1^I\|_{L^2}^2 + \|Q_1^R\|_{L^2}^2 + \varepsilon_1 \|Q_2\|_{L^1} \right).
\end{align*}
After intergrating the above estimate over $[0,t]$, we have
\begin{align*}
\begin{aligned} %\label{est-vw}
&\int_\mathbb{R} \frac{v^{3/2} \omega^2}{2} \, dx +  \int_\mathbb{R} \psi \omega \, dx + \int_\mathbb{R} \frac{\mu}{v} (v^{5/2} -\ov^{5/2}) \ow \omega \, dx + \frac{1}{2} \int_0^t \left(\mathcal{G}_w + \mathcal{D}_{w_1} \right) \, d s \\ 
&\le  \int_\mathbb{R} \frac{v_0^{3/2} \omega_0^2}{2} \, dx +  \int_\mathbb{R} \psi_0 \omega_0 \, dx + \int_\mathbb{R} \frac{1}{v_0} (v_0^{5/2} -\ov(0,\cdot)^{5/2}) \ow(0,\cdot) \omega_0 \, dx \\
& \quad + \delta_S \int_0^t| \dot{X}|^2\,d s + C \int_0^t D_{u_1} \, d s +C \sqrt{\delta_0} \int_0^t \left( \| \phi_x \|_{L^2}^2 +D_{u_2} +G^S_v+G^R  \right) \, d s + C \delta_R. 
\end{aligned}
\end{align*}
Since $G_{w}\sim\mathcal{G}_{w}$, $D_{w_1} \sim \mathcal{D}_{w_1}$, and
\begin{align*}
\|\omega\|_{L^2}^2 \sim  \int_\mathbb{R} \frac{v^{3/2} |\omega|^2}{2} \, dx, \quad \left|\int_\mathbb{R} \frac{\mu}{v}(v^{5/2}-\ov^{5/2})\ow \omega \, dx\right|\le C\int_\mathbb{R} |\phi \ow  \omega| \, dx,
\end{align*} 
we complete the proof of Lemma \ref{lem: est-1}.

\section{Prof of Lemma \ref{lem: est-2}}\label{app:proof_high_order_3}
\setcounter{equation}{0}
In this appendix, we provide the proof of Lemma \ref{lem: est-2}. We differentiate $\eqref{Underlying wave with shift}_2$ with respect to $x$ and then multiply by $\omega_x$ to obtain
\begin{equation} \label{psi_tx omega_x}
\begin{aligned}
& \psi_{tx} \omega_x - \left(  \frac{p'(v) v^{5/2}}{\sqrt{\kappa}} \omega + \ow \left(  \frac{p'(v) v^{5/2}}{\sqrt{\kappa}}-\frac{p'(\ov)\ov^{5/2}}{\sqrt{\ok}} \right) \right)_x \omega_x-\left( \frac{\sqrt{\kappa}}{v^{5/2}}\omega_x + \ow_x \left( \frac{\sqrt{\kappa}}{v^{5/2}}-\frac{\sqrt{\ok}}{\ov^{5/2}} \right) \right)_{xx} \omega_x \\
&\; =\left( \frac{\mu}{v} \psi_x+\ou_x \left( \frac{\mu}{v}-\frac{\m}{\ov} \right) \right)_{xx} \omega_x + \dot{X} \us_{xx} \omega_x -(Q_1)_x\omega_x.
\end{aligned}
\end{equation}
Similarly, we use $\eqref{Underlying wave with shift}_3$  and to obtain
\begin{equation} \label{omega_tx psi_x}
\begin{aligned}
& \omega_{tx} \psi_x = -\left( \frac{\sqrt{\kappa}}{v^{5/2}} \psi_x+\ou_x \left( \frac{\sqrt{\kappa}}{v^{5/2}}-\frac{\sqrt{\ok}}{\ov^{5/2}} \right) \right)_{xx} \psi_x + \dot{X} \ws_{xx} \psi_x -(Q_2)_x \psi_x.
\end{aligned}
\end{equation}
Combining \eqref{psi_tx omega_x} and \eqref{omega_tx psi_x}, we obtain
\begin{align*}
\frac{d}{dt}\int_{\mathbb{R}}\psi_x\omega_x \, dx+\mathcal{G}_{w_2}+\mathcal{D}_{w_2}=\sum\limits_{i=1}^{6}T_i,
\end{align*}
where
\begin{align*}
\mathcal{G}_{w_2}:=-\int_{\mathbb{R}}  \frac{p'(v) v^{5/2}}{\sqrt{\kappa}} |\omega_x|^2dx,\quad
\mathcal{D}_{w_2}:=\int_{\mathbb{R}} \frac{\sqrt{\kappa}}{v^{5/2}} |\omega_{xx}|^2 dx,
\end{align*}
and
\begin{align*}
T_1&:=\dot{X} \int_\mathbb{R} \left( \us_{xx} \omega_x+ \ws_{xx} \psi_x \right) \, dx, \\
T_2&:=\int_\mathbb{R} \left( \frac{p'(v) v^{5/2}}{\sqrt{\kappa}}\right)_x \omega \omega_x \, dx -\int_\mathbb{R} \omega_x \left( \frac{\sqrt{\kappa}}{v^{5/2}} \right)_x \omega_{xx} \, dx + \int_\mathbb{R} \psi_x \left( \frac{\mu}{v} \right)_x \omega_{xx} \, dx +\int_\mathbb{R} \psi_{xx} \psi_x \left(\frac{\sqrt{\kappa}}{v^{5/2}}\right)_x \, dx, \\
T_3&:=\int_\mathbb{R} \ow_x \left( \frac{p'(v) v^{5/2}}{\sqrt{\kappa}} -\frac{p'(\ov)\ov^{5/2}}{\sqrt{\ok}} \right) \omega_x \, dx -\int_\mathbb{R} \ow_{xx} \left( \frac{\sqrt{\kappa}}{v^{5/2}}-\frac{\sqrt{\ok}}{\ov^{5/2}}\right) \omega_{xx} \, dx \\ 
&+\int_{\mathbb{R}} \ou_{xx} \left( \frac{\mu}{v}-\frac{\m}{\ov} \right)\omega_{xx} \, d x+\int_\mathbb{R} \psi_{xx}\ou_{xx} \left( \frac{\sqrt{\kappa}}{v^{5/2}}-\frac{\sqrt{\ok}}{\ov^{5/2}} \right) \, dx, \\
T_4&:=\int_\mathbb{R} \ow \left( \frac{p'(v) v^{5/2}}{\sqrt{\kappa}} -\frac{p'(\ov)\ov^{5/2}}{\sqrt{\ok}} \right)_x \omega_x \, dx-\int_\mathbb{R} \ow_x \left( \frac{\sqrt{\kappa}}{v^{5/2}} -\frac{\sqrt{\ok}}{\ov^{5/2}}\right)_x \omega_{xx} \, dx \\
&+\int_\mathbb{R} \ou_x \left( \frac{\mu}{v}-\frac{\m}{\ov} \right)_x \omega_{xx} \, dx+\int_\mathbb{R} \psi_{xx}\ou_x \left(\frac{\sqrt{\kappa}}{v^{5/2}}-\frac{\sqrt{\ok}}{\ov^{5/2}}\right)_x  \, dx,\\
T_{5}&:=-\int_\mathbb{R} \frac{\mu}{v}\psi_{xx} \omega_{xx} \, dx+\int_\mathbb{R}  \frac{\sqrt{\kappa}}{v^{5/2}} |\psi_{xx}|^2 \, dx, \quad 
T_{6}:=-\int_\mathbb{R}( (Q_1)_x \omega_x + (Q_2)_x \psi_x) \, dx.
\end{align*}

\noindent $\bullet$ (Estimate of $T_1$): We use 
\[|\ws_{xx}| \le C (|\vs_x|^3 + |\vs_x||\vs_{xx}| + |\vs_{xxx}|) \le C \delta_S |\vs_x|, \quad |\us_{xx}| \le C \delta_S |\us_x|\]
and Young's inequality to obtain
\begin{align*}
T_1
&\le C \delta_S |\dot{X}| \left( \sqrt{\int_\mathbb{R} |\us_x|^2 \, dx} \sqrt{\int_\mathbb{R} |\omega_x|^2 \, dx } + \sqrt{\int_\mathbb{R} |\vs_x|^2 \, dx} \sqrt{\int_\mathbb{R} |\psi_x|^2 \, dx } \, \right) \\ 
&\le C \delta_S^{\frac{5}{2}} |\dot{X}| \left( \norm{\omega_x}_{L^2} + \norm{\psi_x}_{L^2} \right) \le \delta_S |\dot{X}|^2 + \tau \mathcal{G}_{w_2} + C \delta_S^4 D_{u_1}.
\end{align*}
\noindent $\bullet$ (Estimate of $T_2$):  We use $|v_x| \le C (|\phi_x| + |\ov_x|)$, \eqref{eq: L^infty bounds}, and \eqref{eq: L^infty bounds-2} to estimate $T_2$ as
\begin{align*}
T_2 &\le C \int_{\mathbb{R}} |v_x| \Big(|\omega| |\omega_x| + |\omega_x| |\omega_{xx}| +|\psi_x||\omega_{xx}| + |\psi_x| |\psi_{xx}| \Big)\, dx \\
&\le C(\varepsilon_1 +\delta_0) \left( \norm{\psi_x}_{H^1}^2 + \norm{\omega}_{H^2}^2  \right) \le C(\varepsilon_1 +\delta_0) \left( D_{u_1} + D_{u_2} +G_w  +D_{w_1}+D_{w_2}\right).
\end{align*}

\noindent $\bullet$ (Estimate of $T_3$): Using the definition of $\ow$, Lemma \ref{lem:rarefaction_property}, and Lemma \ref{lem:shock-property},
\begin{align*}
T_3 &\le C \int_\mathbb{R}  (|v^R_x|+|v^S_x|) |\phi| ( |\omega_x| + |\omega_{xx}| + |\psi_{xx}| ) \, dx \\ 
&\le C  \int_{\mathbb{R}} (|v^R_x|+|v^S_x|)^{3/2}|\phi|^2 \, dx + C\int_\mathbb{R} (|v^R_x|+|v^S_x|)^{1/2} (|\omega_x|^2 +|\omega_{xx}|^2 + |\psi_{xx}|^2 ) \, dx \\ 
&\le C \sqrt{\delta_0} \left( G^S_v + G^R + D_{u_2} \right) + \tau (D_{w_1}+D_{w_2}).
\end{align*}

\noindent $\bullet$ (Estimate of $T_4$): Similar to $T_3$, we estimate $T_4$ as
\begin{align*}
T_4 &\le C \int (|v^R_x|+|v^S_x|) (|\phi_x|+|\ov_x||\phi|) (|\omega_x| + |\omega_{xx}| + |\psi_{xx}|) \, dx \\ 
& \le C \int_\mathbb{R} (|v^R_x|+|v^S_x|)^2 |\phi|^2 \, dx + C \int_\mathbb{R} (|v^R_x|+|v^S_x|) (|\phi_x|^2 + |\omega_x|^2 +|\omega_{xx}|^2 + |\psi_{xx}|^2) \,dx \\ 
&\le C \delta_0 (G^S_v +G^R + \|\phi_x\|_{L^2}^2 + D_{u_2}) + \tau (\mathcal{G}_{w_2} +\mathcal{D}_{w_2}). 
\end{align*}

\noindent $\bullet$ (Estimate of $T_5$):  By using Young's inequality, we get
\begin{align*}
T_5 \le C D_{u_2} + \tau \mathcal{D}_{w_2}.
\end{align*}

\noindent $\bullet$ (Estimate of $T_6$): 
To estimate $T_{6}$ we need to estimate $(Q_1)_x$ and $(Q_2)_x$. We note that
\begin{align*}
|(Q_1^I)_x| \le& C\Big( |(\vr_{xx},\vr_{xxx},\vr_{xxx}\vr_x,(\vr_{xx})^2,\vr_{xxxx},(\vr_x)^2\vr_{xx})| |\vs-v_m| \\ 
&\quad + |(\vr_x,\vr_{xx},\vr_{xx}\vr_x,\vr_{xxx},(\vr_x)^3)| |\vs_x|\\
&\quad + |(\vs_{xx},\vs_{xxx},\vs_{xxx}\vs_x,(\vs_{xx})^2,\vs_{xxxx},(\vs_x)^2(\vs_{xx}))| | \vr-v_m|  \\
&\quad +|(\vs_x,\vs_{xx},\vs_{xx}\vs_x,\vs_{xxx},(\vs_x)^3)| |\vr_x| \\ 
&\quad +|( \ur_{xx},\vr_x \vr_{xx},\vr_{xxx})| |\vs_x| +|( \ur_x,(\vr_x)^2,\vr_{xx})| |\vs_{xx}| \\ 
&\quad +|(\us_{xx}, \vs_x \vs_{xx},\vs_{xxx})| |\vr_x|+|(\us_x, (\vs_x)^2,\vs_{xx})| |\vr_{xx}|\Big)
\end{align*}
and 
\begin{equation*} %\label{Q1Rx-est}
|(Q_1^R)_x| \le C \left( |\ur_{xxx}| + |\ur_{xx}| |\vr_{x}|+ |\ur_{x}| |\vr_{xx}| +|\vr_{xxxx}|+|\vr_{xxx}| |\vr_{x}| +|\vr_{xx}|^2 + |\vr_{x}|^2|\vr_{xx}| \right).
\end{equation*}
Therefore, we obtain  
\begin{align*}
\|(Q^I_1)_x\|_{L^2}^2 &\le C \left\|\left( |\vr_x|  |\vs-v_m| + |\vr_x| |\vs_x| +  |\vs_x| |\vr-v_m| + |\vr_x| |\vs_x| \right)\right\|_{L^2}^2 \\ 
&\le C \delta_S \delta_R e^{-C \delta_S t}
\end{align*}
and
\begin{align*}
\int_0^t &\|(Q_1^R)_x\|_{L^2}^2 ds \\
&\le C \int_0^t
\big(\|\vr_{xxxx}\|_{L^2}^2+ \| \vr_{xxx}\|_{L^4}^4+\|\ur_{xxx}\|_{L^2}^2+ \|(\vr_{xx},\ur_{xx})\|_{L^4}^4  + \|(\vr_x,\ur_x)\|_{L^4}^4+ \|\vr_x\|_{L^8}^8 \big) ds\\
&\le C \left( \delta_R + \delta_R^3 + \delta_R + \delta_R^3 + \delta_R^3 + \delta_R^8 \right) \le C \delta_R.
\end{align*}
On the other hand, we also derive 
\begin{align*}
\| (Q_2)_x \|_{L^2}^2 &\le C |\dot{X}| \left( |\vs_{xxx}||\vr-v_m| +|\vs_{xx}|\vr_x| + |\vs_{x}| |\vr_{xx}|  \right) \le C \varepsilon_1 \delta_S \delta_R e^{-C \delta_S t}.
\end{align*}
Therefore, $T_{6}$ can be bounded as
\begin{align*}
T_{6} &\le \int_\mathbb{R} |(Q_1)_x| |\omega_x| + |(Q_2)_x| |\psi_x| \, dx \\
&\le C \left( \|(Q_1^I)_x\|_{L^2}^2 + \|(Q_1^R)_x\|_{L^2}^2 + \frac{1}{\varepsilon_1
} \|(Q_2)_x\|_{L^2}^2 \right) + \tau \mathcal{G}_{w_2} + C \varepsilon_1 D_{u_1}.
\end{align*}
Thus, we combine the estimates on $T_i$ and use the smallness of $\tau$, $\delta_0$, and $\e_1$ to get
\begin{align*}
\frac{d}{dt}\int_{\mathbb{R}}& \psi_x\omega_x \, dx+\frac{1}{2}\left( \mathcal{G}_{w_2} + \mathcal{D}_{w_2} \right) \\
&\le \delta_S |\dot{X}|^2 + C D_{u_2}+C \delta_0 \left(\norm{\phi_x}_{L^2}^2+G^S_v + G^R \right)+C(\varepsilon_1+\delta_0) \left(D_{u_1}+\norm{\omega}_{L^2}^2 \right)\\
&\quad +C \left( \|(Q_1^I)_x\|_{L^2}^2 + \|(Q_1^R)_x\|_{L^2}^2 + \frac{1}{\varepsilon_1
} \|(Q_2)_x\|_{L^2}^2 \right).
\end{align*}
After integrating with respect to time, we obtain
\begin{align*}
\int_{\mathbb{R}} &\psi_x\omega_x \, dx+\frac{1}{2} \int_{0}^t \left( \mathcal{G}_{w_2} + D_{w_2} \right) \, ds 
\\
&\le \int_{\mathbb{R}} \psi_{0x}\omega_{0x} \, dx+ \delta_S \int_0^t |\dot{X}|^2 \, d s + C \int_{0}^{t} D_{u_2} \, ds \\ 
&\quad +C(\varepsilon_1+\sqrt{\delta_0}) \int_0^t \left(\norm{\phi_x}_{L^2}^2+G_w+G^S_v+G^R+D_{u_1} \right) \, d s+C \delta_R.
\end{align*}
Finally, since $D_{w_1}\sim \mathcal{G}_{w_2}$ and $D_{w_2}\sim \mathcal{D}_{w_2}$, this completes the proof of Lemma \ref{lem: est-2}.

\end{appendix}

\bibliographystyle{amsplain}
\bibliography{reference} 
	
\end{document}